\documentclass[smallextended,envcountsect]{svjour3} 
\smartqed 

\usepackage[utf8x]{inputenc}
\usepackage[margins]{trackchanges}
\usepackage{rotating}
\usepackage{xr}
\usepackage{makecell}  
\usepackage{booktabs}  
\usepackage{rotating}  
\usepackage{pifont, placeins}
\usepackage{parskip}
\usepackage{array}

\usepackage{pifont}
\usepackage{comment}
\usepackage{graphicx}
\usepackage{algorithm,algpseudocode}
\usepackage{mathrsfs}
\usepackage{amsmath}  
\usepackage{mathtools}
\usepackage{amssymb}
\usepackage{amsfonts} 
\usepackage{threeparttable}
\usepackage{cite}
\usepackage{hyperref}
\usepackage{url}
\usepackage{color}
\usepackage{soul}
\allowdisplaybreaks[1]
\usepackage{pdflscape}

\newtheorem{observation}{Observation}
\newtheorem{assumption}{Assumption}

\DeclarePairedDelimiterX\Set[2]{\lbrace}{\rbrace}%
 { #1 \,\delimsize| \,\mathopen{} #2 }
\newcommand\norm[1]{\left\lVert#1\right\rVert} 
\newcommand{\bs}[1]{\boldsymbol{#1}}
\newcommand{\E}{\mathbb{E}}

\newcommand{\mX}{\mathcal{X}}
\newcommand{\mY}{\mathcal{Y}}
\newcommand{\mP}{\mathcal{P}}

\newcommand{\mbR}{\mathbb{R}}
\newcommand{\mI}{\mathcal{I}}

\newcommand{\cN}{\mathcal{N}}
\newcommand{\conv}{\textrm{conv}}
\newcommand{\mQ}{\mathcal{Q}}

\newcommand{\dom}{\textrm{dom}}
\newcommand{\ri}{\mathrm{relint}}
\usepackage{multirow}

\newcommand{\opt}{\textrm{OPT}}
\newcommand{\mtc}{\mathcal}

\newcommand{\x}{\boldsymbol{x}}

\newcommand{\mF}{\mathcal{F}}

\newcommand{\mbZ}{\mathbb{Z}}

\newcommand{\cS}{\mathcal{S}}
\newcommand{\mC}{\mathcal{C}}

\newcommand{\dist}{\textrm{dist}}
\DeclarePairedDelimiterX{\inp}[2]{\langle}{\rangle}{#1, #2}

\newcommand{\xe}{x^{\epsilon}}
\newcommand{\wt}{\widetilde}

\newcommand{\beq}{\begin{equation}}
\newcommand{\eeq}{\end{equation}}
\newcommand{\ba}{\begin{aligned}}
\newcommand{\ea}{\end{aligned}}
\newcommand{\bdm}{\begin{displaymath}}
\newcommand{\edm}{\end{displaymath}}

\newcommand{\sub}{\textrm{Sub}}

\usepackage{xcolor}
\newcommand{\qrevision}[1]{{\color{black}{#1}}}
\newcommand{\srevision}[1]{{\color{black}{#1}}}

\begin{document}
\title{A Cutting-plane and Benders' Decomposition Algorithm for 
Two-Stage Distributionally Robust Convex programs}
\author{Fengqiao Luo \and Shibshankar Dey  \and  Sanjay Mehrotra }
\institute{Fengqiao Luo \at
              Department of Industrial Engineering and Management Science, Northwestern University \\
              Evanston, Illinois\\
              fengqiaoluo2014@u.northwestern.edu
            \and
            Shibshankar Dey \at
              Department of Industrial Engineering and Management Science, Northwestern University \\
              Evanston, Illinois\\
              shibshankardey2025@u.northwestern.edu 
           \and
              Sanjay Mehrotra,  Corresponding author  \at
              Department of Industrial Engineering and Management Science, Northwestern University \\
              Evanston, Illinois\\
              mehrotra@northwestern.edu
}

\date{Received: date / Accepted: date}
\maketitle
\numberwithin{equation}{section}
\begin{abstract}
We present a finitely convergent cutting-plane algorithm for solving a general mixed-integer convex program given an oracle for solving a general convex program. This method is extended to solve a family of two-stage mixed-integer convex programs using cutting planes, with applications to solving distributionally-robust two-stage stochastic mixed-integer convex programs. Analysis is also given for the case where convex programming oracle provides an $\epsilon-$optimal solution. We combine the cut generation with a branch-and-union scheme to develop a more practical algorithm. Computational results on generated test problems show the practicality of our algorithm. Specifically, results show that in the tested problems our algorithm achieves $<5\%$ optimality gap in 12 hours. This gap is $>17\%$ with a commercial solver.
\end{abstract}

\section{Introduction}
We consider a family of mixed-integer convex programs in the form:
\beq\label{opt:micp-xy}
\ba
&\min\; c^\top x + h^\top y \\
&\;\text{ s.t. } \srevision{Tx + Wy = q,} \\
& \qquad\; g_i(x,y)\le 0 \quad\forall i\in\mtc{I}, \\
&\qquad\; x\in\mtc{X}\cap\{0,1\}^{l_1},\; y\in\mtc{Y}\cap(\mbZ^{l_2}\times\mbR^{l_3}),
\ea
\tag{JMICP}
\eeq
where $x$ are binary variables, and $y$ are
mixed-integer variables. \srevision{$Tx + Wy = q$ represents polyhedral constraints.} The functions $g_i(x,y)$ for 
$i\in\mtc{I}$ are convex but not necessarily differentiable, 
and $\mtc{X}$ and $\mtc{Y}$ are bounded polyhedral sets for $x$ and $y$ including simple-bound constraints.
Note that the following problem can be formulated as \eqref{opt:micp-xy}:
\beq\label{opt:micp-xy-conv-obj}
\ba
&\min\; g_0(x,y) \\
&\;\text{ s.t. } \srevision{Tx + Wy = q,} \\
& \qquad\; g_i(x,y)\le 0 \quad\forall i\in\mtc{I}, \\
&\qquad\; x\in\mtc{X}\cap\{0,1\}^{l_1},\; y\in\mtc{Y}\cap(\mbZ^{l_2}\times\mbR^{l_3}), 
\ea
\eeq
where $g_0(x,y)$ is a convex function of ${(x,y)}$. It can be achieved 
by moving the convex objective to the set of constraints with an auxiliary variable $\eta$
as $g_0(x,y)\le \eta$. In this case, the new objective is to minimize $\eta$ and we redefine 
a new $\mtc{Y}$-space variable $\tilde{y}$ as $\tilde{y}:={(y,\eta)}$, and \eqref{opt:micp-xy-conv-obj}
set with $c=\bs{0}$, $h=(\bs{0}, 1)$. 
An important special case of \eqref{opt:micp-xy-conv-obj} is the following two-stage
stochastic mixed-integer convex program (TSS-MICP) under finite support assumption:
\beq\label{opt:TSS-MICP-intro}
\ba
&\min_x\; c(x) + \E_{\xi}[\mQ(x,\xi)] \\
&\textrm{ s.t. }\; Ax\le b,\; x\in\mC, \\
&\qquad\; x\in\{0,1\}^{l_1},
\ea
\eeq
where $x$ is the vector of first-stage variables that are pure binary, $\mtc{C}$ is a closed convex set, 
and $\xi$ is a vector of random parameters that follow the joint distribution $P$ 
on a finite support $\Omega$.
The recourse function $\mQ(x,\xi^\omega)$ for scenario $\omega\in\Omega$ is given as:
\beq\label{opt:second-stage-intro}
\ba
&\mQ(x,\xi^\omega)=\min_{y^{\omega}}\; h^{\omega\top}y^{\omega} \\
&\qquad\qquad\quad\textrm{ s.t. } \srevision{T^{\omega}x + W^{\omega}y^{\omega} = q^{\omega},} \\
& \qquad\qquad\qquad\quad 
g^{\omega}_i(x,y^{\omega})\le 0 \qquad\forall i\in \cal I, \\
&\qquad\qquad\qquad\quad  y^{\omega}\in\mbZ^{l_2}\times\mbR^{l_3},
\ea
\eeq
where $g_j$ are convex functions.
Under finiteness assumption of support set $\Omega$, \eqref{opt:TSS-MICP-intro}--\eqref{opt:second-stage-intro} 
admit an extended reformulation in terms of variables $x$ and $y^\omega$ for $\omega\in\Omega$:
\beq\label{opt:extend}
\ba
&\min\; c(x) + \sum_{\omega\in\Omega}p^\omega h^{\omega\top}y^\omega \\
&\textrm{ s.t. }\; \srevision{T^{\omega}x + W^{\omega}y^{\omega} = q^{\omega},}\\
&\qquad\; g^{\omega}_i(x, y^{\omega})\le 0 \qquad \qquad \qquad \, \forall i\in \cal I, \\
&\qquad Ax\le b, \; x\in\mC,\\
&\qquad x\in\{0,1\}^n, \; y^{\omega}\in\mbZ^{l_1}\times\mbR^{l_2}, \; {\forall \omega \in \Omega,}
\ea
\eeq
where $p^\omega$ is the probability of scenario $\omega$.
Note that \eqref{opt:extend} is in the form of \eqref{opt:micp-xy-conv-obj}. 
Thus \eqref{opt:extend} can be solved directly using an algorithm for solving {mixed-integer} convex programs. 
However, a Benders' decomposition approach is typically preferable when the number of scenarios is large.

In addition to developing an algorithm for solving {mixed-integer} convex programs, a major result of this paper is to develop a finitely convergent cutting-plane decomposition algorithm as well as a linearize-branch-and-union decomposition
algorithm for solving \eqref{opt:micp-xy} under mild regularity conditions.  Parametric cuts are generated within the decomposition algorithm. An oracle for generating the parametric cuts is also developed. 
We consider the following notions and assumptions in this paper.

\qrevision{
\begin{definition}
Let $f$ be a convex function on $\mbR^n$. A restricted function $\tilde{f}$ of $f$
induced by an index set ${\cal J}\subseteq\{1,\dots,n\}$ and a vector $v\in\mbR^{|\cal J|}$
is a convex function defined on $\mbR^{n-|\cal J|}$ as follows: 
{$\tilde{f}(y)=f(x_{\{1,\dots,n\}\setminus{\cal J}}=y, x_{\cal J}=v)$},
i.e., evaluated by fixing all variables indexed by $\cal J$ to be the constant vector $v$. 
\end{definition}

\begin{definition}
Let $\mF(\mbR^n)$ be a collection of convex functions such that 
for each $f\in\mF(\mbR^n)$, $\dom f=\mbR^k$ for some $1\le{k}\le{n}$. 
$\mF(\mbR^n)$ is called a \textit{nested} collection if for any $f\in\mF(\mbR^n)$,
every restricted function of $f$ is also a member in $\mF(\mbR^n)$. 
\end{definition}
}

\begin{assumption}\label{ass:subgrad-decomp}
For any $(x_0,y_0)\in\mathbb{R}^{l_1+l_2+l_3}$, the function $g_i\;\forall i\in\mtc{I}\cup\{0\}$ 
satisfies $\partial g_i(x_0,y_0)=\partial_x g_i(x_0,y_0)\times \partial_y g_i(x_0,y_0)$, where $\partial g_i(x_0,y_0)$
is the set of sub-gradients of $g_i$ at the point $(x_0,y_0)$, $\partial_x g_i(x_0,y_0)$ is the set of sub-gradients
of the function $g_i(x,y_0)$ at $x=x_0$, and $\partial_y g_i(x_0,y_0)$ is defined similarly.
\end{assumption}

\begin{assumption}\label{ass:non-empty relative interior}
\qrevision{For any $\hat{x}\in\mX\cap\{0,1\}^{l_1}$ and any $\hat{z}\in\mathbb{Z}^{l_2}$, 
either $\mtc{S}(\hat{x},\hat{z})$ is empty or it has non-empty relative interior,
where $\mtc{S}(\hat{x},\hat{z})=\{y\in\mbR^{l_2+l_3}:\;T\hat{x} + Wy = q,\;g_i(\hat{x},y)\le 0 \; \forall i\in\mtc{I},\;y\in\mtc{Y},\; y_j=\hat{z}_j\;\forall j\in[l_2]\}$.}
\end{assumption}

\begin{assumption}\label{ass:no-error}
\qrevision{(i) Every constraint function from every optimization problem formulated in this paper
is a member of the nested collection {$\mF(\mbR^n)$}. 
(ii) We have an oracle that can solve to optimality or detect infeasibility 
of a convex optimization problem 
generated by any finite subset of member functions from {$\mF(\mbR^n)$}. 
(iii) We have an oracle that can numerically generate a sub-gradient $\partial{f}(x)$ 
for any member function {$f\in \mF(\mbR^n)$} at any point $x$. Furthermore,
for any closed convex set $S\subseteq\mbR^k$ generated by a finite subset of member functions
in {$\mF(\mbR^n)$}, this oracle can numerically generate the normal cone ${\cal N}_S(x)$ at any point $x\in\mbR^k$.
}
\end{assumption}
{Note that the above assumptions allow the functions $g_i$ ($i\in{\cal I}\cup\{0\}$) 
or the members of $\mF(\mbR^n)$ to be non-differentiable. This
ensures that the methods developed in this paper are sufficiently general to apply to a broad class of problems, 
provided that we have oracles specified above.}

\subsection{\qrevision{Justification of assumptions}}
Assumption~\ref{ass:subgrad-decomp} ensures that a valid inequality derived in the $y$-space for
\eqref{opt:micp-xy} can be extended to the $(x,y)$-space. 
Use of this assumption will be discussed in Remark~\ref{rmk:param-representation}.
Assumption~\ref{ass:non-empty relative interior} ensures that
the normal cone of a point in $\text{bd}(\mtc{S}(\hat{x},\hat{z}))$ 
is decomposable based on the fact that 
$\cN_{C_1\cap{C_2}}=\cN_{C_1}+\cN_{C_2}$ if and only
if $\ri(C_1\cap{C_2})$ is non-empty. The decomposition of a normal cone 
enables us to derive equivalent optimality conditions.

In all of the above assumptions, we have not imposed any differentiability on 
the candidate functions. In general, they can be convex but non-differentiable. 
For example, the candidate functions can contain non-differentiable operators such as max($\cdot$)
and $|\cdot|$, etc.

\qrevision{Assumption~\ref{ass:no-error} assumes that every convex program considered can be solved to optimality or detected as infeasible. 
We are interested in how the advanced mixed-integer problems studied in this paper 
can be solved to exact optimality using the pure cutting-plane approach
by reducing them into basic mixed-integer linear programs and convex programs.
It follows the literature of understanding how cutting-plane methods can solve a {mixed-integer linear program (MILP)} to exact optimality under the assumption that the {linear program (LP)} can be efficiently solved to exact optimality. However, an analysis of the algorithms for mixed-integer convex program (MICP), joint mixed-integer convex program (JMICP) and TSS-MICP without Assumption~\ref{ass:no-error} is provided in Appendix~\ref{sec:micp-epsilon} under an $\epsilon-$optimality assumption of the convex programming oracle.}

\subsection{Literature review}
\qrevision{We begin by comparing our work with existing literature in Table~\ref{tab:literature-comparison}.}

\begin{table}
	\scriptsize
	\centering
	\rotatebox{90}{
	\begin{minipage}{\textheight}
	\caption{\qrevision{Comparison of related works and problem settings in this work.}}
	\begin{tabular}{ m{2cm} | c | c | c | c | c | c | c  }
	\hline\hline
	\textbf{Literature} & \makecell{\textbf{Problem} \\ \textbf{Family}} & \makecell{\textbf{Pure} \\ \textbf{Cutting} \\  \textbf{Plane}} & \makecell{\textbf{Finite} \\ \textbf{Convergence}} & \makecell{\textbf{Exact} \\ \textbf{Solution}} & \makecell{\textbf{Error} \\ \textbf{Analysis}} & \textbf{Implementation} & \makecell{\textbf{Distributional} \\ \textbf{Robustness} \\ \textbf{Consideration}} \\
	\hline
	Balas (1993) \cite{balas1993-lift-proj-cut-plane-0-1-LP} & Binary LP & \ding{51} & \ding{51} & \ding{51} & \ding{55} & \ding{55} & {\ding{55}}\\
	\hline
	Owen \& Mehrotra (2001) \cite{owen2001disjunctive} & MILP & \ding{51} & \ding{51} &  \ding{55} &  \ding{55} &  \ding{55} & \ding{55} \\
	\hline
	J{\"o}rg (2007) \cite{jorg2007-k-disj-cut} & MILP & \ding{51} & \ding{51} &  \ding{55} &  \ding{55}  & \ding{55} & \ding{55} \\
	\hline
	Chen et. al. (2011) \cite{ChenKucukyavuzSen2011} & MILP & \ding{55} & \ding{51} & \ding{51} &  \ding{55} & \ding{55} & \ding{55} \\
	\hline
	Lubin et. al. (2018) \cite{LubinYamangilBentPablo2018} & MI-Convex-P & \ding{51} & \ding{55} & \ding{55} &  \ding{55} &  \ding{51} & \ding{55} \\
	\hline
	Luo \& Mehrotra (2019) \cite{luo2019_DR-TSS-MICP} & \makecell{TSS-MI-Conic-P, \\ DR-TSS-MI-Conic-P} & \ding{55} & \ding{51} & \ding{51} &  \ding{55} &  \ding{51} & \ding{51} \\
	\hline
	This paper & \makecell{MI-Convex-P, \\ TSS-MI-Convex-P, \\ DR-TSS-MI-Convex-P} & \ding{51} & \ding{51} &  \ding{51} &  \ding{51} &  \ding{51} & \ding{51} \\
	\hline
    \label{tab:literature-comparison}
	\end{tabular}
	\end{minipage}}
\end{table}
The development of algorithms for solving mixed-integer convex programs (MICPs) benefits from solvers for MILPs. 
We provide a brief review of work for solving MILPs with cutting planes, the outer approximation approach for solving MICPs based on cutting-plane methods, 
and the generalization of these methods for solving two-stage stochastic {mixed-integer}
programs.  

Branch-and-cut algorithms \cite{padberg1991-branch-and-cut,mitchell2002-branch-and-cut} 
are well developed for solving an MILP.
The cutting planes are generated at the root node or some other node of the branch-and-bound
tree to strengthen the linear relaxation, or cut the current solution having fractional components. Many families of valid inequalities have been developed since the 1950s to serve as cutting planes in the 
branch-and-cut algorithm. In particular, we have general purpose cuts  \cite{cornuejols2008-valid-ineq-milp}
such as Gomory cuts \cite{balas1996-gomory-cuts}, disjunctive cuts \cite{balas1985}, mixed-integer rounding cuts, and polyhedral structure-based cuts such as flow cover inequalities \cite{padberg1985-valid-ineq} and flow-path inequalities \cite{VanRoy1985-valid-ineq}.
Generation of these cuts plays an important role in a branch-and-cut algorithm.
A natural question is whether a MILP can be
solved to optimality by adding cutting planes only. For a general mixed 0-1
linear program, an affirmative answer is given in
\cite{balas1993-lift-proj-cut-plane-0-1-LP}, in which a systematic way of adding
disjunctive cuts to the linear relaxation problem is given. The disjunctive cuts are generated
using the lift-and-project method investigated in \cite{balas1985,balas1998}, 
which amounts to solving a linear program to obtain the coefficients of a disjunctive cut.
Owen and Mehrotra \cite{owen2001disjunctive} developed a 
cutting-plane algorithm for solving a general mixed-integer linear program. To guarantee the algorithm's convergence to an optimal solution, cuts are added at all $\gamma$-optimal 
(a notion defined in \cite{owen2001disjunctive})
vertices of the LP-relaxation master problem at every iteration. As an alternative to exploring the $\gamma$-optimal vertices, 
Chen et. al.~\cite{ChenKucukyavuzSen2011}
proposed a convergent cutting-plane algorithm for solving a general mixed-integer linear
program by keeping the union of certain disjunctive polytopes.  
J{\"o}rg \cite{jorg2007-k-disj-cut} provided an alternative convergent cutting-plane algorithm
that can solve a general MILP to optimality by introducing the notion of $k$-disjunctive cuts,
which is a generalization of the disjunctive cuts given in \cite{balas1985,balas1998}.  

Outer approximation (OA) methods for solving a {mixed-integer} convex program (MICP) also have a long history.  The idea can be traced back to \cite{DuranGrossman1986}. 
In an outer approximation method, the convex set is outer approximated 
with a polyhedron which relaxes the MICP with a MILP in every master iteration.
The polyhedron approximation is progressively refined by adding more valid 
inequalities that are tangent to the convex set.  
The theoretical question on whether a general MICP can be solved to optimality with cutting
planes is not well addressed, particularly when the convex functions are not differentiable. 
Towards addressing this question, Lubin et. al.~\cite{LubinYamangilBentPablo2018} use 
a polishing step after solving the master MILP, which solves a continuous
convex optimization problem by fixing the integer variables to be the values of the integer part in the current 
master solution. This polishing step is similar to the solution polishing done within the algorithm described in \cite{owen2001disjunctive}. The solution from the polishing-step problem is compared with the master problem solution 
to check whether an optimal solution is identified.
However, the proof given in \cite{LubinYamangilBentPablo2018} is incomplete and does not apply to the case where the convex functions are not differentiable. The convex functions considered in \cite{LubinYamangilBentPablo2018} are assumed
to be differentiable, in which case, the outer approximation cutting plane is unique at a specific boundary
point, and it can be obtained by taking the gradient of a constraint function. 
For a MICP involving general convex functions, this approach requires generalization.
In particular, if the convex set is a convex cone, the number of tangent planes passing
through the origin can be infinite. 

{Cutting-plane} methods have also been used for solving two-stage stochastic mixed-integer 
linear programs (TSS-MILP) with mixed-integer second-stage variables.
One approach is to generate parametric Gomory cuts that sequentially convexify the feasible set 
\cite{sen2014_decomp-alg-gomory-cuts-ts-smip,kucukyavuz2014_finite-convergent-decomp-alg-ts-smip}. 
Gade et al. \cite{sen2014_decomp-alg-gomory-cuts-ts-smip} have shown the finite-convergence of this algorithm 
for solving TSS-MILPs with pure-binary first-stage variables and pure-integer second-stage variables based 
on generating Gomory cuts that are parameterized by the first-stage solution. 
This approach is generalized by Zhang and K{\"u}{\c c}{\"u}kyavuz~\cite{kucukyavuz2014_finite-convergent-decomp-alg-ts-smip}
for solving TSS-MILP with pure-integer variables in both stages.  
Recent work has also provided insights into developing tighter 	
formulations by identifying globally valid parametric inequalities 
(see \cite{bansal2018_tight-ts-smip}, and references therein).  
For two-stage stochastic mixed-integer conic optimization, 
Bansal and Zhang \cite{bansal2020_conic-mip-cuts} have developed nonlinear sparse cuts 
for tightening the second-stage formulation of a class of two-stage stochastic $p$-order conic 
mixed-integer programs by extending the results of \cite{atamturk2008_conic-mip-cuts} on 
convexifying a simple polyhedral conic mixed-integer set.

Sen and Sherali \cite{sen2006_decomp-BB-TSS-MIP} developed a decomposition framework 
for solving a deterministic pure-binary linear program with two sets of binary 
decision variables coupled by some linking constraints. The algorithm iteratively
solves a master problem and a sub-problem, each involving different sets of variables.
The Benders' cut added to the master problem is a disjunctive valid inequality 
generated via taking union of all leaf nodes in the branch-and-bound tree 
constructed for solving the sub-problem. This framework is
applied to solve two-stage stochastic mixed-integer programs with pure binary
first-stage variables, as shown in \cite{sen2006_decomp-BB-TSS-MIP}. 
\cite{luo2019_DR-TSS-MICP} showed that the decomposition framework with a branch-and-cut approach can be further developed for the two-stage stochastic mixed-integer conic programming using the branch-and-union approach as well as an approach where the scenario problems are solved using cutting planes. 

\qrevision{This work can benefit distributionally robust (nonlinear) convex optimization problems
where candidate distributions share a common finite support. Ambiguity sets used in these problems
can be defined using the Wasserstein metric 
\cite{esfahani2018-dro-wass,gao2022-dro-wass-reg,luo2019-dro-wass-reg-model,gao2022-dro-wass}, 
the $\phi$-divergence \cite{namkoong2016-sgd-dro-phi-diverg,lam2019-stat-dro-phi-diverg,duchi2021-stat-ro-emp-likelihood}, 
and the total variation distance \cite{rahimian2019-eff-scen-dro-tot-var,dixit2022-dr-pred-ctrl-tot-var-dist}.
This family of problems has broad applications in transportation science and logistics 
\cite{sun2022-dro-fair-transit-res-alloc,shehadeh2022-dro-stoch-mob-fac}, 
supply chain management \cite{fu2018-dro-decent-supp-chain,xin2022-dr-invent-ctrl}, 
energy system \cite{zhang2022-building-load-ctrl-dro,wang2022-dro-unit-commit-flex-gen-res}, 
healthcare system \cite{wang2021-chance-constr-oper-room-sched,chow2022-dro-surgery-alloc}
and financial decision making \cite{blanchet2021-dr-mean-var-portfolio-select-wass,costa2023-dr-portf-construction}, etc.
We note that some algorithms and convergence results developed in this work
can be applied to solve scenario sub-problems of certain distributionally robust optimization problems
in which the underlying deterministic problem is a mixed-integer convex program. 
For example, this work targets distributionally robust mixed-integer convex optimization problems, 
such as those that appeared in
\cite{ketkov2024-dro-micp-wass-incomplete-data,luo2019_DR-TSS-MICP,xie2018-dro-simple-integer-recourse,delage2022-value-rand-sol-micp-dro}. 
 }

\subsection{Contributions of this paper}
This paper makes the following contributions:
\begin{itemize}
	\item A cutting-plane algorithm is developed for solving a general 
		mixed-integer convex program (MICP) (Section~\ref{sec:alg-micp}). 
		It is proved that the algorithm can identify an optimal solution of MICP 
		in finitely many iterations with finitely many cutting planes  (Section~\ref{sec:convg-analy}). To the best of our knowledge, this is the first formal proof of finite convergence of a cutting-plane algorithm for solving a general MICP.
	\item A decomposition algorithm is developed for solving \eqref{opt:micp-xy}. A novel aspect of this algorithm is an oracle that develops a parametric cut that is added to the master problem in the decomposition algorithm (Section~\ref{sec:decomp-cut-plane}). This decomposition algorithm is further generalized to solve a family of distributionally robust two-stage stochastic {mixed-integer} convex programs (Section~\ref{sec:cut-plane-two-stage-convex-prog}).
    \item For practical implementation, a cutting-plane algorithm often needs to be coupled with a branch-and-bound technique. Therefore, we combined the cut generation and the branch-and-bound method to develop a linearize-branch-and-union algorithm that is more suitable for implementation (Section~\ref{sec:decomp-branch-union}).
	\item The main algorithms developed in this paper were implemented and tested to solve a family of two-stage stochastic and distributionally-robust mixed-integer convex programs with pure-binary first-stage decision variables and mixed-integer second-stage decision variables (Section~\ref{sec:num_exp}). Computational results indicate that, within a 12-hour time limit in a two-stage stochastic setting, our algorithm achieves an optimality gap of less than 5\%, whereas Gurobi exceeds 17\%.
    \item We provide an analysis of our methods for the case where the convex-optimization oracle can only solve convex programs to a desired accuracy (Section~\ref{sec:micp-epsilon}).
\end{itemize}

We demonstrate the organization of this paper in Table~\ref{tab:organization} and use the following notations throughout the next sections.

\textbf{Notations and definitions:} We use $\mtc{C}$ to represent a closed convex set. $\dist(x, \mtc{C})$ denotes the minimum distance between point $x$ and set $\mtc{C}$. $\dist(x, y)$ is the distance between any two point $x$ and $y$. 
\qrevision{The notations $\textrm{int}(\mtc{C})$, $\ri(\mtc{C})$ and $bd(\mtc{C})$
respectively represent the set of interior points, the relative interior points, and the boundary points of set $\mtc{C}$,
where by definition $bd(\mtc{C})=\mtc{C}\setminus\text{int}(\mtc{C})$.} For a non-empty convex set $\mtc{C}\subset\mbR^n$,
$x\in bd(\mtc{C})$ if and only if there is a hyperplane supporting $\mtc{C}$ at $x$. A normal cone at some point $x \in bd(\mtc{C})$ is denoted by $\mtc{N}_{\mtc{C}}(x)$. 
$\|.\|$ denotes $\ell_2$-norm unless any other norm is specified with a subscript {.} We let $[K]$ be the index set $\{1, 2, \cdots, K\}$ and $|K|$ be its cardinality. $\bs{1}$ and {$\bs{0}$ denote vectors with all entries being 1 and 0, respectively}. The projection of point $x_0$ on a set $S$ is denoted by {$\mathrm{Proj}_{S}{x_0}$}.

\begin{table}
	\scriptsize
	\centering
	\caption{\qrevision{Organization of results and algorithms in this paper.}}
	\renewcommand{\arraystretch}{2}
	\begin{tabular}{ c| c | c | c | c | c }
	\hline\hline
	\textbf{Problem} & \textbf{Section} & \makecell{\textbf{Sub-} \\ \textbf{problem}} & \makecell{\textbf{Main} \\ \textbf{Algorithm}} & \makecell{\textbf{Algorithm} \\ \textbf{Calls}} & \makecell{\textbf{Pure} \\ \textbf{Cutting-} \\ \textbf{plane}} \\
	\hline
	{\makecell{Mixed-integer  \\ convex program \\ \eqref{opt:micp} }} & Section~\ref{sec:alg-micp} & N/A & Algorithm~\ref{alg:micp-oracle} & N/A & \ding{51}  \\
	\hline
	{\makecell{Parametric \\ MICP\\ \eqref{opt:y-micp2} } } & Section~\ref{sec:param-cut-plane} & N/A & Algorithm~\ref{alg:param-micp} & N/A & \ding{51} \\
	\hline
	{\makecell{Joint MICP\\ \eqref{opt:micp-xy} } }  & Section~\ref{sec:decomp-alg} & \eqref{opt:y-micp2} & Algorithm~\ref{alg:decomp-cutting-plane} & Algorithm~\ref{alg:param-micp} & \ding{51}  \\
	\hline
    {\makecell{Distributionally \\ robust two-stage \\  stochastic MICP \\  \eqref{opt:TSS-MICP} }}  & Section~\ref{sec:cut-plane-two-stage-convex-prog} & \makecell{\eqref{opt:micp} \& \\ \eqref{opt:y-micp2}} & Algorithm~\ref{alg:decomp-BC} & {Algorithms}~\ref{alg:micp-oracle},~\ref{alg:param-micp} & \ding{51} \\
	\hline
    \eqref{opt:TSS-MICP} & Section~\ref{sec:decomp-branch-union} & \makecell{\eqref{opt:micp} \& \\ \eqref{opt:y-micp2}} & Algorithm~\ref{alg:decomp-BU} & \makecell{Modified \\ Algorithm~\ref{alg:param-micp} \& \\ Branch-\&-Cut} & \ding{55} \\
	\hline
    \makecell{\eqref{opt:micp-xy} \& \\  \eqref{opt:TSS-MICP} \\ with error \\ analysis} & \makecell{Appendix \\~\ref{sec:micp-epsilon}} & \makecell{\eqref{opt:micp} \& \\ \eqref{opt:y-micp2}} & \makecell{{Algorithms}~\ref{alg:decomp-JMICP-e},\\  ~\ref{alg:decomp-BC-error}} & Algorithm~\ref{alg:micp-e} & \ding{51} \\
	\hline
	\end{tabular}
    \label{tab:organization}
\end{table}

\section{An Algorithm for Mixed-Integer Convex Programs Using Cutting Planes}
\label{sec:alg-micp}
Let us consider a mixed-integer convex program in the form:
\beq\label{opt:micp}
\ba
&\min\; c^{\top}x \\
&\textrm{ s.t. }\; Ax = b, \quad x\in {\cal C},\; x\in\mathbb{Z}^{l_1}\times\mathbb{R}^{l_2},
\ea
\tag{MICP}
\eeq
where $Ax= b$ are linear constraints \qrevision{and $\mC$ is a closed bounded convex set.
The finite-convergent cutting-plane algorithm for solving \eqref{opt:micp} to exact optimality, developed and analyzed in this section, is the foundation for handling other families of problems studied in this paper.}

\subsection{\qrevision{Algorithm development}}
\qrevision{In this section, we make the following assumption for \eqref{opt:micp} which is a counterpart 
of Assumption~\ref{ass:non-empty relative interior}:
\begin{assumption}\label{ass:micp-ri}
For any $\hat{z}\in\mathbb{Z}^{l_1}$, 
either $\mtc{S}(\hat{z})$ is empty or it has a non-empty relative interior,
where $\mtc{S}(\hat{z})=\{x:\;Ax=b,\;x\in{\cal C},\; x_i=z_i\;\forall i\in[l_1]\}$.
\end{assumption}
}

\begin{proposition}\label{prop:lin-approx-convex-prog}
Let \qrevision{$\cal B$} be a convex set in $\mathbb{R}^n$
generated by a finite set of member functions from {$\mF(\mbR^n)$}. 
Consider the following convex program:
\beq\label{opt:convex-opt}
\ba
&\min\;c^{\top}x \\
&\emph{ s.t. } {\bar{A}} x = {\bar{b}},\;
x \in \qrevision{\cal B}. 
\ea
\eeq
\qrevision{
Let $x^*$ be an optimal solution of \eqref{opt:convex-opt}.
Then a vector $\alpha\in\cN_{\cal B}(x^*)$ can be computed such that the following linear program is a relaxation of \eqref{opt:convex-opt}:
\beq\label{opt:conv-linear-opt}
\ba
&\min\;c^{\top}x \\
&\emph{ s.t. } {\bar{A}} x = {\bar{b}},\; \alpha^\top(x-x^*) \le 0.
\ea
\eeq
Furthermore, if the relative interior of the feasible set of \eqref{opt:convex-opt}
is non-empty, the optimal value of \eqref{opt:conv-linear-opt} is equal to $c^{\top}x^*$.
}
\end{proposition}
\begin{proof} 
See Appendix~\ref{sec:proofs}.
\end{proof}

The pseudo-code is given in Algorithm~\ref{alg:micp-oracle}. 
We now provide an overview of our algorithm.  
We will show that the algorithm terminates with an optimal solution after adding finitely many cuts. 
At iteration $n$, we consider the following mixed-integer linear program as a master problem:
\beq\label{opt:micp-master}
\ba
&\min_{x\in\mathbb{Z}^{l_1}\times\mathbb{R}^{l_2}}\; c^{\top}x \\
&\textrm{ s.t. } Ax = b, \\
&\qquad (x^{(k)}-z^{(k)})^{\top}x\le (x^{(k)}-z^{(k)})^{\top}z^{(k)} \quad \forall k\in [n-1], \\
&\qquad C^{(k)} x \le d^{(k)} \qquad \qquad \qquad \qquad \qquad \quad \, {\forall k\in\mI^{(n-1)},}  
\ea
\tag{MS-$n$}
\eeq 
where $(x^{(k)}-z^{(k)})^{\top}x\le (x^{(k)}-z^{(k)})^{\top}z^{(k)}$ is the cutting plane generated at iteration $k$ to separate the master
problem solution $x^{(k)}$ from $\mtc{C}$ at iterations $k < n$, 
and $C^{(k)}x\le d^{(k)}$ generated at iteration $k\in[n-1]$ 
are outer-approximation cuts for approximating $\mtc{C}$. 
The way of generating these cuts will be described later. 
Since these cuts are not necessarily available at every iteration, the algorithm uses the index set
$\mI^{(n-1)}$ to record the iteration indices in $[n-1]$ in which these cuts are generated. 
Generation of these cuts will become clearer as we explain the algorithm.

\qrevision{Note that \eqref{opt:micp-master} is a MILP, and hence it can be solved to optimality
by adding cutting planes without branching, using the oracle from \cite{ChenKucukyavuzSen2011}.
Since we are interested in whether a pure cutting-plane method can be developed 
to solve optimization problems concerned in this work, all the oracles used by our algorithms
should also be based on the pure cutting-plane approach.}

If  \eqref{opt:micp-master} is infeasible, 
we stop and conclude that \eqref{opt:micp} is infeasible. 
Otherwise, let $x^{(n)}$ be a solution of the master problem \eqref{opt:micp-master}. 
If $x^{(n)} \in {\cal C}$, then $x^{(n)}$ is optimal to \eqref{opt:micp} and the algorithm stops. 
Otherwise, consider the projection problem:
\beq
\min_{z\in {\cal C}}\; \big\|z-x^{(n)}\big\|^2_2  \label{eqn:IterateProjection}
\eeq
and let $z^{(n)}$ be its optimal solution. Then the algorithm generates the following separation cut:
\beq \label{eqn:OuterApproximationCut}
		(x^{(n)}-z^{(n)})^{\top}x\le (x^{(n)}-z^{(n)})^{\top}z^{(n)}.
\eeq 
{The above inequality is valid for $\cal C$, and thus also for \eqref{opt:micp}, 
since it corresponds to a tangent plane to $\cal C$ passing through the boundary point $z^{(n)}$. 
Moreover, because $x^{(n)} \notin \cal C$, the cut is either violated by $x^{(n)}$ or active at $x^{(n)}$.}
The algorithm first adds \eqref{eqn:OuterApproximationCut} to \eqref{opt:micp-master}.
Then it solves the following continuous convex program:
 \beq\label{opt:convex-FixedInteger}
\ba
&\min_{ x\in {\cal C}}\; c^{\top}x \\
&\text{ s.t. } Ax = b,\\
&\qquad x_j =x_j^{(n)} \quad j\in [l_1].
\ea
\eeq 
{where integer variables in $x$  are fixed.} The problem in \eqref{opt:convex-FixedInteger} can be viewed as a solution polishing step. 
It also plays a key role in our finite convergence analysis. 
If \eqref{opt:convex-FixedInteger} is infeasible, 
the algorithm sets $\mI^{(n)}\gets\mI^{(n-1)}$, updates $n\gets n+1$, 
and proceeds to the next iteration. 
Otherwise, the solver can find an optimal solution $\tilde{x}^{(n)}$ of \eqref{opt:convex-FixedInteger} 
that is on the boundary of $\cal C$ due to the linearity of the objective function.
\qrevision{The algorithm generates outer-approximation inequalities $C^{(n)}x\le{d}^{(n)}$ 
induced by $\tilde{x}^{(n)}$ which consists of the inequality
$\alpha^\top(x-\bar{x}^{(n)})\le0$ with $\alpha\in\cN_{\cal C}(\bar{x}^{(n)})$ derived 
according to Proposition~\ref{prop:lin-approx-convex-prog} 
({The equality constraints $Ax=b,\; x_j=x^{(n)}_j\; j\in[l_1]$ and set $\cal C$ 
are counterparts of the constraints $\overline{A}x=\overline{b}$ and set $\cal B$ in Proposition~\ref{prop:lin-approx-convex-prog}, respectively.})
and some other outer-approximation inequalities derived based on
active smooth constraints.}
Since Assumption~\ref{ass:micp-ri} holds, the results of Proposition~\ref{prop:lin-approx-convex-prog} can be applied to give the following property:
\beq\label{eqn:Cx<d}
\ba
&\{\min c^\top x\;|\; Ax = b,\; x_j =x_j^{(n)} \;\,\quad  \forall j\in [l_1],\; x\in{\cal C} \}	\\
&= \{\min c^\top x\;| \; Ax = b,\; x_j =x_j^{(n)} \; \forall j\in [l_1],\; C^{(n)} x \leq d^{(n)} \}.
\ea
\eeq
The algorithm repeats the above procedures.

 Some key notations used 
in Algorithm~\ref{alg:micp-oracle} are given below:
\begin{itemize}
	\item $x^{(k)}$: an optimal solution of the master problem at iteration $k$;
	\item $z^{(k)}$: the projection of $x^{(k)}$ on $\cal C$, computed at iteration $k$;
	\item $(x^{(k)}-z^{(k)})^{\top}x\le (x^{(k)}-z^{(k)})^{\top}z^{(k)}$: the separation cut generated
		at iteration $k$;
	\item $\tilde{x}^{(k)}$: an optimal solution of the continuous convex program \newline
		$\Set*{\min\; c^\top x}{Ax = b,\; x_j=x^{(k)}_j\;\forall j\in[l_1],\;x\in\cal C}$
		obtained at iteration $k$;
	\item $C^{(k)} x \leq d^{(k)}$: the outer-approximation inequalities, generated at iteration $k$
		when \eqref{opt:convex-FixedInteger} is feasible;
	\item $\mI^{(n)}$: this is the subset $\Set*{k\in[n]}{\textrm{problem \eqref{opt:convex-FixedInteger} is feasible at iteration }k}$ of 
		iteration indices for which the outer-approximation inequalities are generated.
\end{itemize}

\begin{algorithm}
	\caption{An algorithm for solving a {mixed-integer} convex program \eqref{opt:micp}.}
	\label{alg:micp-oracle}
	\begin{algorithmic}[1]
	\State{Set $n\gets 1$ and $\mI^{(0)}\gets\emptyset$.}
	\Loop
		\State{(Start iteration $n$.)}
		\State{\textbf{Step 1:} \qrevision{Solve the master problem \eqref{opt:micp-master} using
		the cutting-plane oracle  from \cite{ChenKucukyavuzSen2011}} to get an optimal solution $x^{(n)}$ and optimal objective denoted as $L^{(n)}$}.
		\State{If the master problem is infeasible, then stop and claim that \eqref{opt:micp} is infeasible.}\label{lin:micp-infeasb}
		\State{Let $x^{(n)}=(x_Z^{(n)},x_R^{(n)})$, where $x_Z^{(n)} \in \mathbb{Z}^{l_1}$ and $x_R^{(n)} \in \mathbb{R}^{l_2}$}.
		\If{$x^{(n)}\in\cal C$} \label{lin:x_in_C}
			\State{Stop and return $x^{(n)}$ as an optimal solution of \eqref{opt:micp};}
		\Else
			\State{Continue.}
		\EndIf
		\State{\textbf{Step 2:} Solve \eqref{eqn:IterateProjection} to get the projection point $z^{(n)}$ of $x^{(n)}$ on $\cal C$.}
		\State{Add the separation cut $(x^{(n)}-z^{(n)})^{\top}x\le (x^{(n)}-z^{(n)})^{\top}z^{(n)}$ to \eqref{opt:micp-master}.}\label{lin:sep-ineq}
		\State{\textbf{Step 3:} Solve the continuous convex program:}
		\beq\label{opt:conv-fix-x1}
		\ba
		&\min\; c^{\top}x \quad \textrm{ s.t. } \srevision{Ax = b}, \;x_Z=x_Z^{(n)},  \quad x\in {\cal C}.
		\ea
		\eeq
		\If{\eqref{opt:conv-fix-x1} is feasible} \label{cond:feasb}
			\State{\qrevision{Obtain an optimal solution $\tilde{x}^{(n)}\in\text{bd}(\cal C)$ of \eqref{opt:conv-fix-x1}.}
			{Such a boundary-point optimal 
            solution exists, since the objective is linear and other constraints are equalities.}}
			\State{\qrevision{This is achievable as the objective function is linear.}}
			\State{Let $U^{(n)}=c^\top\tilde{x}^{(n)}$.}
			\If{$L^{(n)}=U^{(n)}$} \label{lin:L=U_conv}
				\State{Stop and return $\tilde{x}^{(n)}$ as an optimal solution of \eqref{opt:micp}.}
			\EndIf
			\State{\qrevision{Generate outer-approximation inequalities $C^{(n)}x \le d^{(n)}$ as given in \eqref{eqn:Cx<d} }}  
			\State{\qrevision{and add them to \eqref{opt:micp-master}.}}
			\State{Set $\mI^{(n)}\gets\mI^{(n-1)}\cup\{n\}$.}
		\Else
			\State{Set $\mI^{(n)}\gets\mI^{(n-1)}$. }
		\EndIf
		\State{$n\gets n+1$.}
	\EndLoop
	\end{algorithmic}
\end{algorithm}

\subsection{\qrevision{Convergence analysis for the {cutting-plane} algorithm}}
\label{sec:convg-analy}
We first provide a few technical results that are used for proving
the main result (Theorem~\ref{thm:finite-oracle-mixed-int-convex}) of this section.
\begin{observation}
The inequalities $C^{(n)}x\le{d}^{(n)}$ added to the master problem in each iteration
when they are generated are valid for \eqref{opt:micp}. Hence, the master problem
\eqref{opt:micp-master} is a relaxation of \eqref{opt:micp} at every iteration.
\end{observation}

\begin{proposition}\label{prop:int-seq-convg}
 If $\{x^{(n)}\}^{\infty}_{n=1}$ is a convergent sequence 
in $\mathbb{Z}^{l_1}\times\mathbb{R}^{l_2}$, then the limit point of
this sequence is also in $\mathbb{Z}^{l_1}\times\mathbb{R}^{l_2}$.
There exists an index $N$ and a vector $v\in\mathbb{Z}^{l_1}$
such that $x^{(n)}_i=v_i$ for all $i\in[l_1]$ and all $n\ge N$. 
\end{proposition}
\begin{proof}
See Appendix~\ref{sec:proofs}.
\end{proof}

\qrevision{
\begin{proposition}\label{prop:separation}
Let $S\subset\mbR^n$ be a nonempty closed convex set, $x_0\notin{S}$, and {$z_0=\mathrm{Proj}_S({x_0})$}.
Then, for any $\delta<\norm{x_0-z_0}/2$, the supporting hyperplane $\{x\in\mbR^n: (x_0-z_0)^\top(x-z_0)=0\}$
separates $S$ and $B_\delta(x_0)$, where $B_\delta(x_0)$ is an open ball of radius $\delta$ centered at $x_0$.
\end{proposition}
}

\begin{lemma}\label{lem:limit-x-in-K}
Suppose Algorithm~\ref{alg:micp-oracle} does not terminate finitely,
and let $\{x^{(n)}\}^{\infty}_{n=0}$ be the infinite sequence
generated by the algorithm. Then every convergent subsequence
of $\{x^{(n)}\}^{\infty}_{n=0}$ has its limit point in $\emph{bd}({\cal C})$.
\end{lemma}
\begin{proof}
We prove the result by contradiction.
Consider a convergent subsequence $\{x^{(n_i)}\}^{\infty}_{i=1}$.
Suppose 
\beq
\hat{x}=\lim_{i\to\infty}x^{(n_i)}.
\eeq
Suppose $\hat{x}\notin\textrm{bd}({\cal C})$.
There are two possibilities: 
(1) $\hat{x} \in\text{int}({\cal C})$; and (2) $\hat{x} \notin {\cal C}$.
In Case (1), there must exist a $x^{(n_i)}$ such that $x^{(n_i)}\in {\cal C}$. 
But, such a solution must be optimum because it is obtained from an outer approximation problem
and Algorithm~\ref{alg:micp-oracle} should have terminated (Condition in Line~\ref{lin:x_in_C} 
of Algorithm~\ref{alg:micp-oracle} is satisfied at iteration $n_i$).
In Case (2), we have $\dist(\hat{x},{\cal C})>0$.

Note that by definition $z^{(n)}$ (in \eqref{eqn:IterateProjection}, 
Line 6 of Algorithm~\ref{alg:micp-oracle}) is the (Euclidean) projection of $x^{(n)}$
onto the set ${\cal C}$, and  the inequality
\beq\label{eqn:ProjectionCut}
\ba
&(x^{(n)}-z^{(n)})^{\top}(x-z^{(n)})\le 0
\ea
\eeq
added in the algorithm is a supporting valid inequality for ${\cal C}$ 
at the support point $z^{(n)}$. 
\qrevision{Since $\{x^{(n_i)}\}$ is a convergent subsequence, there exists an
sufficiently large index $j$ and a $\delta>0$ such that $x^{(n_k)}\in{B}_\delta(x^{(n_j)})$
for any $k>j$. Furthermore, by Proposition~\ref{prop:separation},
the cutting plane $\{x: (x^{(n_j)}-z^{(n_j)})^{\top}(x-z^{(n_j)})=0\}$ separates
${\cal C}$ and $B_\delta(x^{(n_j)})$.
This shows that the point $x^{(n_k)}$ ($n_k>n_j$) violates the separation cut 
$(x^{(n_j)}-z^{(n_j)})^{\top}(x-z^{(n_j)})\le 0$ 
generated at a previous iteration $n_j$,
which leads to a contradiction since $x^{(n_j)}$
is an optimal solution of (MS-$n_j$) which includes the cut $(x^{(n_j)}-z^{(n_j)})^{\top}(x-z^{(n_j)})\le 0$
based on the algorithm. This concludes the proof. \qed}
\end{proof}
The next {lemma} shows that the limit point of the infinite subsequence 
considered in Lemma~\ref{lem:limit-x-in-K} can be identified after a finite number
of iterations by solving a convex program. 

\begin{lemma}\label{lem:convex-convg-subseq}
Suppose Algorithm~\ref{alg:micp-oracle} does not terminate finitely,
and let $\{x^{(n_i)}\}^{\infty}_{i=1}$ be a convergent subsequence
generated by this algorithm. Let $x^*$ be the limit point of this subsequence. Then, for a sufficiently large $n_i$, $x^*$ is an optimal solution
of the problem:
\beq\label{opt:convex-x1*}
\ba
&\min_{ x\in {\cal C}}\; c^{\top}x \\
		&\emph{ s.t. } \srevision{Ax = b}, \;x_j =x_j^{(n_i)},\quad j\in [l_1].
\ea
\eeq 
\end{lemma}
\begin{proof}
By Lemma~\ref{lem:limit-x-in-K}, the limit point $x^*$ of the subsequence
is in $\textrm{bd}({\cal C})$. By Proposition~\ref{prop:int-seq-convg}, 
there exists an iteration index $N$
such that $x^{(n_i)}_j=x^*_j\in\mathbb{Z}$ for all $n_i > N$ and $j\in[l_1]$.
Therefore, $x^*$ is a feasible solution of \eqref{opt:micp}. For simplicity let $x = (x_{{Z}},x_{{R}})$, where $x_{Z}$ denote the first $l_1$ (integer) variables and $x_{R}$ represent the continuous variables.  
For any $n_i > N$, let
\begin{align}
U^* &=\min\; c^{\top}x \quad\textrm{s.t. } \srevision{Ax = b}, \;x_j=x_j^{*},\; j \in [l_1], \;x\in {\cal C}. \label{opt:U*} 
\end{align}
We now show that $x^{*}$ is an optimal solution of \eqref{opt:U*}.
If not, then assume that $\hat{x}$ is an optimal solution 
of \eqref{opt:U*}, and it satisfies 
$c_{{R}}^{\top}\hat{x}_{{R}}<c_{R}^{\top}x_{R}^{*}$.
Obviously, $(x_{{Z}}^{(n_i)},\hat{x}_{{R}})=(x^*_Z,\hat{x}_R)$ is a feasible solution
of \eqref{opt:U*} for sufficiently large $n_i$. Moreover, we have 
\bdm
\ba
& c^{\top}(x_{{Z}}^{(n_i)},\hat{x}_{{R}}) = c_{{Z}}^{\top}x_{{Z}}^{(n_i)} 
+c_{{R}}^{\top}x_{{R}}^{*}-
c_{{R}}^{\top}\left(x_{{R}}^{*}-\hat{x}_{{R}}\right) 
< c_{{Z}}^{\top}x_{{Z}}^{(n_i)}+c_{{R}}^{\top}x^*_R.
\ea
\edm
Since $\lim_{i\to\infty}x^{(n_i)}_R=x^*_R$, the above inequality
implies that $c^{\top}(x_{{Z}}^{(n_i)},\hat{x}_{{R}})<c^\top x^{(n_i)}$
for sufficiently large $n_i$. But this contradicts with that $x^{(n_i)}$
is an optimal solution of the master problem at iteration $n_i$
which is a relaxation of \eqref{opt:micp}. Therefore, we have proved that 
$x^{*}$ is an optimal solution of \eqref{opt:U*}
for any $n_i$ that is sufficiently large. \qed
\end{proof}

\begin{theorem} \label{thm:finite-oracle-mixed-int-convex}
Let Assumptions \ref{ass:no-error} and \ref{ass:micp-ri} hold.
Algorithm~\ref{alg:micp-oracle} terminates in finitely many iterations after adding a finite number 
of cutting planes. \qrevision{When it terminates, either it detects that \eqref{opt:micp} is infeasible or
it returns an optimal solution of \eqref{opt:micp}.}
\end{theorem}
\begin{proof}
By contradiction, assume that Algorithm~\ref{alg:micp-oracle}
does not terminate in a finite number of iterations. 
Then let $\{x^{(n)}\}^{\infty}_{n=0}$ be the master problem solutions
generated by the algorithm.
Since the infinite sequence is bounded, it must contain
a convergent subsequence, namely $\{x^{(n_i)}\}^{\infty}_{i=1}$.
Assume that $x^*$ is the limit point of this subsequence.
By Lemma~\ref{lem:limit-x-in-K}, we know that 
$x^*\in\textrm{bd}(\cal C)$.
We focus on proving that $L^{(n_i)}=U^{(n_i)}$ for a sufficiently large $n_i$,
and hence the algorithm should terminate (Line~\ref{lin:L=U_conv}).

From Proposition~\ref{prop:int-seq-convg}, and Lemma~\ref{lem:convex-convg-subseq} we have an $\hat{i}$ such that
$x_{{Z}}^{(n_i)}=x_{{Z}}^{*}\in\mathbb{Z}^{l_1}$ for any $i\ge \hat{i}$ and we have $U^{(n_i)}=U^*$, where $U^*$ is defined as
\beq\label{opt:U*-defTheorm} 
\ba
 U^{*}  & =\min\; c^{\top}x \\  
 & \qquad\textrm{s.t. } Ax = b,  \;x_{{Z}}=x_{{Z}}^{*},\;x\in {\cal C}.
\ea
\eeq 
From Algorithm~\ref{alg:micp-oracle}, $x^{(n_i)}$ and $L^{(n_i)}$ are an optimal solution 
and the optimal value of the following {mixed-integer} linear program:
\beq\label{opt:relax1-thm}
\ba
&\min_{x\in\mathbb{Z}^{l_1}\times\mathbb{R}^{l_2}}\; c^{\top}x \\
& \,\, \quad \textrm{ s.t. } \quad \srevision{Ax = b}, \\
&\qquad \qquad  (x^{(k)}-z^{(k)})^{\top}x\le (x^{(k)}-z^{(k)})^{\top}z^{(k)} \quad\forall k\in[n_i-1], \\
&\qquad \qquad C^{(k)} x \le d^{(k)} \qquad \qquad \qquad \qquad \qquad \quad \, {\forall k\in\mI^{(n_i-1)}.}
\ea
\eeq
Note that the constraints $C^{(k)} x \le d^{(k)}$ for all $k\in\mI^{(n_i-1)}$ are the out-approximation valid inequalities generated at Step 3 
which are active at the solution $\tilde{x}^{(k)}$ of \eqref{opt:U*-defTheorm}.
Therefore, the following relation holds:
\bdm
\ba
L^{(n_i)} &=\min\; c^{\top}x 
\textrm{ s.t. }\left\{
\def\arraystretch{1.5}
\begin{array}{l}
\srevision{Ax = b}, \quad x\in\mathbb{Z}^{l_1}\times\mathbb{R}^{l_2},\\
(x^{(k)}-z^{(k)})^{\top}x\le (x^{(k)}-z^{(k)})^{\top}z^{(k)}\quad \forall k\in [n_i-1], \\
C^{(k)} x \le d^{(k)} \qquad \qquad \qquad \qquad \qquad \quad \, \forall k\in\mI^{(n_i-1)}
\end{array}
\right. \\
&=\min\; c^{\top}x 
\textrm{ s.t. }\left\{
\def\arraystretch{1.5}
\begin{array}{l}
Ax = b, \quad x_{{Z}}=x_{{Z}}^{*},\quad x\in\mathbb{R}^{l_1+l_2}, \\
(x^{(k)}-z^{(k)})^{\top}x\le (x^{(k)}-z^{(k)})^{\top}z^{(k)} \quad \forall k\in [n_i-1],  \\
C^{(k)} x \le d^{(k)} \qquad \qquad \qquad \qquad \qquad \quad \, \forall k\in\mI^{(n_i-1)}
\end{array}
\right. 
\ea
\edm
\begin{displaymath}
\begin{aligned}
&\ge \min\; c^{\top}x 
\textrm{ s.t. }\left\{
\def\arraystretch{1.5}
\begin{array}{l}
Ax = b, \quad x_{{Z}}=x_{{Z}}^{*},\quad x\in\mathbb{R}^{l_1+l_2},  \\
C^{(k)} x \le d^{(k)} \qquad\forall k\in\mI^{(n_i-1)} 
\end{array}
\right.  \\
&= \min\; c^{\top}x 
\textrm{ s.t. }\left\{
\def\arraystretch{1.5}
\begin{array}{l}
Ax = b, \quad x_{{Z}}=x_{{Z}}^{*}, \quad x\in\mathbb{R}^{l_1+l_2},\quad x \in {\cal C},\;\qrevision{\textrm{using Proposition~\ref{prop:lin-approx-convex-prog}}}
\end{array}
\right. \\
&=U^*=U^{(n_i)}.
\end{aligned}
\end{displaymath}
\qrevision{This proves that the algorithm must terminate in a finite number of iterations.}

\qrevision{If the algorithm terminates and claims that the problem is infeasible, 
Line~\ref{lin:micp-infeasb} of the algorithm indicates that the 
master problem \eqref{opt:micp-master} is infeasible. However,
since the master problem is a relaxation of \eqref{opt:micp},
it further implies that \eqref{opt:micp} is also infeasible. 
On the other hand, if the algorithm returns a solution, it must be an optimal
solution because the solution is optimal to the master problem and feasible
to the original problem. This concludes the proof.  \qed}
\end{proof}

\section{A Finitely Convergent Cutting-plane Decomposition Algorithm for Solving \textrm{(JMICP)}}
\label{sec:decomp-cut-plane}
We now focus on developing a finitely convergent cutting-plane 
decomposition algorithm for \eqref{opt:micp-xy}. 
The general idea of this algorithm is as follows:
The problem \eqref{opt:micp-xy} is decomposed into a first-stage master problem 
in the space of $[x,\eta]$ and a second-stage problem in the space of $y$ variables, where $\eta$ is an auxiliary variable that links the first-stage and the second-stage problems. 
In particular, the master problem at the beginning of iteration $m$  
is formulated as the following mixed 0-1 linear program:
\beq\label{opt:xy-master}
\ba
&\min_{x\in\mX\cap\{0,1\}^{l_1}}\; c^\top x + \eta \\
&\quad\;\text{ s.t. }\; \eta \ge a^{k\top} x+b^k \quad \qrevision{\forall k\in\mtc{R}^{m-1},} \\
&\qquad\quad\; \qrevision{u^{k\top}x+v^k\ge0 \quad \, \forall k\in\mtc{A}^{m-1},}
\ea
\eeq
where $\eta\ge a^{k\top}x+b^k$ is a valid inequality added to the first-stage
at a previous iteration $k$ \qrevision{and $u^{k\top}x+v^k\ge0$ is a feasibility cut added
to the first-stage problem at iteration $k$. The $\mtc{R}^{m-1}$ (resp. $\mtc{A}^{m-1}$)
is the set of iteration indices at which a valid inequality (resp. a feasibility cut) is
added until the end of iteration $m-1$. }
The details on the generation of these inequalities are given later. 
The algorithm calls the cutting-plane oracle from 
\cite{balas1993-lift-proj-cut-plane-0-1-LP} \qrevision{(i.e., without branching)} to solve the master problem \qrevision{(a mixed 0-1 linear program) to optimality} and obtains the first-stage solution $x^m$.

At the second stage of iteration $m$, one needs to solve the following problem to optimality:
\beq\label{opt:y-micp}
\ba
&\min_{ y\in\mY\cap(\mbZ^{l_2}\times\mbR^{l_3})} h^\top y  \\
&\qquad\text{s.t } \srevision{T x^m + W y = q,}\\
&\qquad\quad\,\, g_i(x^m,y)\le 0 \quad \forall i\in\mtc{I}.
\ea
\eeq
We will develop a generalized parametric cutting-plane algorithm 
(presented in Section~\ref{sec:param-cut-plane}) which shares the same 
spirit as the cutting-plane algorithm developed in Section~\ref{sec:alg-micp}
to solve \eqref{opt:y-micp} without using explicit branching on integer variables, 
such that upon termination, the parametric cutting-plane algorithm generates the following linear relaxation
problem of \eqref{opt:y-micp} if it is feasible:
\beq\label{opt:y-lp}
\ba
&\min_{y\in\mbR^{l_1}\times\mbR^{l_2}}\; h^\top y \\
&\;\text{ s.t } \srevision{T x^m + W y = q,}\\
& \qquad C^m x^m + D^m y \le F^m,
\ea
\eeq
where the matrices satisfy the following conditions:
\begin{itemize}
	\item There exists a common optimal solution (denoted as $y^m$)
		of \eqref{opt:y-micp} and \eqref{opt:y-lp};
	\item The linear inequalities $C^mx+D^my\le F^m$ are valid for the original feasible set 
		$\{[x,y]:\; \srevision{T x + W y = q,}\; g_i(x,y)\le 0 \; \forall i\in\mtc{I},\; x\in\mX\cap\{0,1\}^{l_1},\; 
		   y\in\mY\cap(\mbZ^{l_2}\times\mbR^{l_3})\}$ of \eqref{opt:micp-xy}.
\end{itemize}
Let $\mu^m$ and $\lambda^m$ be the optimal dual vectors corresponding to the equality and inequality constraints in \eqref{opt:y-lp}
\srevision{($\mu^m$ and $\lambda^m$ satisfies the constraints $W \mu^m = h$,
 ${D^m}^\top \lambda^m + h = 0$ and $\lambda^m\ge 0$)}.
We can then generate the following Benders' cut using strong duality:
\beq\label{eqn:bender-cut}
\srevision{\eta \ge (\lambda^{m\top}C^m - \mu^{m \top} T) x - (\lambda^{m\top}F^m - \mu^{m \top} q)}.
\eeq 
The above inequality will serve as an additional inequality added to the 
master problem \eqref{opt:xy-master} at the end of iteration $m$, i.e., 
$\eta\ge a^{m\top} x+b^m$ with \srevision{$a^m=(\lambda^{m\top}C^m - \mu^{m \top} T)^{\top}$ and
$b^m=- (\lambda^{m\top}F^m - \mu^{m \top} q)$}. 

\qrevision{If our generalized parametric cutting-plane algorithm detects that \eqref{opt:y-micp} is infeasible with respect to the first-stage solution $x^m$, we  generate a feasibility cut or a cover inequality as follows:
\beq\label{eqn:feasb-cut}
\sum_{j\in{\mtc{J}_0}}x_j +  \sum_{j\in{\mtc{J}_1}}(1-x_j) \ge 1,
\eeq
where $\mtc{J}_0=\{j:\;x^m_j=0\}$ and $\mtc{J}_1=\{j:\;x^m_j=1\}$.
This inequality is denoted as $u^{m\top}x+v^m\ge0$ and added to the master problem \eqref{opt:xy-master}.
}

The core of the decomposition algorithm is to generate the linear relaxation \eqref{opt:y-lp}
of \eqref{opt:y-micp} with the desired properties. In Section~\ref{sec:param-cut-plane}, 
we present a parametric cutting-plane algorithm to obtain the linear relaxation,
and in Section~\ref{sec:decomp-alg}, we provide the decomposition algorithm based on
the parameter cut generation. {The convergence of these algorithms is analyzed.}

\subsection{A generalized parametric cutting-plane algorithm}
\label{sec:param-cut-plane}
\qrevision{In this section, we utilize a parametric cutting-plane oracle 
developed in Appendix C of \cite{luo2019_DR-TSS-MICP}.
For reference purposes, we denote this oracle as \textbf{oracle-PMILP}.
We first explain its functionality. 
The oracle is designed to solve (or detect the infeasibility of) a 
general MILP of mixed-integer variables $y$ with a parametric variable $x$ as follows:
\beq\label{opt:gen-param-milp}
\ba
&\min_{y\in\mbZ^{l_2}\times\mbR^{l_3}}\; h^\top y  \\
&\text{ s.t. } Tx + W y = q, \\
&\qquad y\in{\cal Y},
\ea
\eeq 
where $\cal Y$ is a polytope.
For any fixed $x=\hat{x}\in\{0,1\}^{l_1}$, the oracle generates parametric cutting planes 
and eventually generates the following LP relaxation:
\beq\label{opt:param-LP}
\ba
&\min_{y\in\mbR^{l_2+l_3}}\; h^\top y  \\
&\text{ s.t. } T\hat{x} + W y = q, \\
&\qquad A\hat{x} + By \le d, \\
&\qquad y\in{\cal Y},
\ea
\eeq 
where $A\hat{x} + By \le d$ are additional cutting planes generated by the oracle such that the following properties hold: (1) An optimal solution of \eqref{opt:param-LP}
obtained by the oracle is also an optimal solution of \eqref{opt:gen-param-milp} for $x=\hat{x}$.  
(2) {The inequalities $Ax + By \le d$ are valid for the feasible set $\{(x,y)\in\{0,1\}^{l_1}\times\mbR^{l_2+l_3},\;y\in\mbZ^{l_2}\times\mbR^{l_3}:\;\;Tx+Wy=q,\;y\in{\cal Y}\}$.} 
We also remark that this oracle is developed based on the pure cutting-plane method provided in 
\cite{ChenKucukyavuzSen2011}.
}

We now develop a generalized parametric cutting-plane algorithm to solve
a parametric mixed-integer convex program in the form:
\beq\label{opt:y-micp2}
\ba
&\min_{y\in\mY\cap(\mbZ^{l_2}\times\mbR^{l_3})}\; h^\top y \\
&\;\text{ s.t } \srevision{T \hat{x} + W y = q,}\\
& \qquad g_i(\hat{x},y)\le 0 \quad \forall i\in\mtc{I}, 
\ea
\tag{PMICP}
\eeq
where $\hat{x}\in\mtc{X}\cap\{0,1\}^{l_1}$ is a first stage solution passed to the problem
as a parameter like in \eqref{opt:y-micp}. 
{We note that} solution $\hat{x}$ is fixed when solving \eqref{opt:y-micp2}, but it is subject to change when solving the master problem in future iterations.

The high-level idea of this algorithm is to first linearize the convex constraints
in \eqref{opt:y-micp2} to get the following parametric mixed-integer linear program 
which is a relaxation of \eqref{opt:y-micp2}:
\beq\label{opt:param-milp}
\ba
&\min_{y\in\mbZ^{l_2}\times\mbR^{l_3}}\; h^\top y  \\
&\text{ s.t. } \srevision{T \hat{x} + W y = q,}\\
& \qquad \bar{T} \hat{x} + \bar{W} y \le \bar{q},
\ea
\eeq 
where $\hat{x}\in\mtc{X}\cap\{0,1\}^{l_1}$ is the problem parameter,
and $(\Bar{T}, \Bar{W}, \bar{q})$ are some matrices induced by linearization. 
\qrevision{It then solves \eqref{opt:param-milp} using \textbf{oracle-PMILP}. }

More specifically, the generalized parametric cutting-plane algorithm works as follows:
At the beginning of the $n^{\text{th}}$ main iteration of the algorithm, 
we encounter a master problem in the form
\beq\label{opt:master-x-y}
\ba
&\min_{y\in\mbZ^{l_2}\times\mbR^{l_3}}\; h^\top y \\
&\text{ s.t. } \srevision{T \hat{x} + W y = q,}\\
&\qquad a^{(k)\top}\hat{x}+b^{(k)\top}y + c^{(k)} \le 0 \quad \forall k\in[n-1] \\
&\qquad \qrevision{D^{(k)}_x\hat{x} +D^{(k)}_yy \le d^{(k)}} \qquad \quad \, \forall k\in\mtc{I}^{(n-1)}, \\
&\qquad \qrevision{y\in\mtc{Y}},
\ea
\tag{PMS-$n$}
\eeq
where $a^{(k)\top}\hat{x}+b^{(k)\top}y + c^{(k)} \le 0$ is the cutting plane
generated at iteration $k$ induced by a projection problem (referred as projection cuts),
\qrevision{and $D^{(k)}_x\hat{x} +D^{(k)}_yy \le d^{(k)}$ are outer-approximation cuts.
The coefficients and matrices in these inequalities are given below. 
Note that \eqref{opt:master-x-y} is in the form \eqref{opt:param-milp},
and hence can be solved using \textbf{oracle-PMILP}. We now explain the procedures for generating the projection cuts and the outer-approximation cuts for the parametric problem.}

Suppose $y^{(n)}$ is an optimal solution of \eqref{opt:master-x-y}.
The projection problem that leads to $a^{(n)\top}\hat{x}+b^{(n)\top}y + c^{(n)} \le 0$
is formulated as
\beq\label{opt:param-proj}
\ba
& \min\; \|x-\hat{x}\|^2+\|y-y^{(n)}\|^2  \\
& \text{ s.t. } g_i(x,y)\le 0 \quad \forall i\in\mtc{I}, \;\; y\in\mtc{Y}.
\ea
\eeq
Let $(x^*,y^*)$ be the optimal solution of the projection problem.
Then the coefficients of the inequality are determined by the hyperplane passing through
$(x^*,y^*)$ with the normal vector $(\hat{x}-x^*,y^{(n)}-y^*)$. Specifically, we have
\beq\label{eqn:coef-abc}
\ba
&a^{(n)}=\hat{x}-x^*, \\
&b^{(n)}=y^{(n)}-y^*, \\ 
&c^{(n)}=-(\hat{x}-x^*)^\top x^*-(y^{(n)}-y^*)^\top y^*.
\ea
\eeq
Note that the general form of the projection cut should be
\beq\label{eqn:proj_cut}
a^{(n)\top}x+b^{(n)\top}y+c^{(n)}\le 0. 
\eeq
We set $x=\hat{x}$ and add the above inequality to update \eqref{opt:master-x-y}.

The \qrevision{outer-approximation cuts (or supporting valid inequalities)} at iteration $n$ are generated 
based on the following procedures similar to the one described in Section~\ref{sec:alg-micp}. 
First, we solve a convex optimization problem in which all integer variables are fixed:
\beq\label{opt:conv-fix-yI}
\ba
&\min\; h^\top y \\
&\;\text{ s.t. } \srevision{T \hat{x} + W y = q,}\\
&\qquad\; g_i(\hat{x},y)\le 0 \quad \forall i\in\mtc{I}, \\
&\qquad\; y_Z = y^{(n)}_Z,\; y\in\mY,
\ea
\eeq
where $y_Z$ represent the integer variables in $y$. {Recall that convex relaxation of \eqref{opt:micp-xy} is denoted as:}
\beq\label{eqn:set_S}
\mtc{S}:=\{(x,y):\;g_i(x,y)\le0\;\forall i\in\mtc{I},\;y\in\mtc{Y}\}.
\eeq
Suppose \eqref{opt:conv-fix-yI} is feasible and {$\hat{y}$} is an optimal solution of \eqref{opt:conv-fix-yI}. 
{
Then point $[\hat{x},y^*]$ is either in $\text{int}(\mathcal{S})$ or at $\text{bd}(\mathcal{S})$. 
In the first case, Algorithm~2 proceeds to next iteration (see line 29). In the second  case we generate outer-approximation cuts.
}

\begin{lemma}\label{lem:h_in_cone}
\qrevision{Let Assumption~\ref{ass:non-empty relative interior} hold.
Let $y^*$ be an optimal solution of \eqref{opt:conv-fix-yI}
such that $(\hat{x},y^*)\in\text{bd}(\mtc{S})$. 
Then it is possible to compute a vector $d_y\in\cN_g(y^*)$
where $\cN_g(y^*)$ is the normal cone of $\{y: g_i(\hat{x},y)\le0\; \forall i\in\mtc{I}\}$ at $y^*$,
such that the following linear program is a relaxation of \eqref{opt:conv-fix-yI}:
\begin{equation}\label{eqn:obj-equal}
\begin{aligned}
\min\; h^\top y \quad\emph{s.t. }
	\left\{\def\arraystretch{1.5}
	\begin{array}{l}
        \srevision{W y = q - T \hat{x}}, \\
	d^\top_y(y-y^*)\le 0, \\
	y_Z=y^{(n)}_Z, \\
	y\in\mtc{Y},
	\end{array}
	\right.
\end{aligned}
\end{equation}
where recall that $\mtc{Y}$ is a polytope. 
Furthermore, \eqref{eqn:obj-equal} and \eqref{opt:conv-fix-yI} have the same optimal value.}
\end{lemma}
\begin{proof}
{The lemma follows directly from Proposition \ref{prop:lin-approx-convex-prog}, 
upon noting that \eqref{opt:conv-fix-yI} is a special case of \eqref{opt:convex-opt}.
In particular, the linear equality constraints $Wy=q-T\hat{x}$ in \eqref{opt:conv-fix-yI} corresponds to 
$Ax=b$ in \eqref{opt:convex-opt}, and the set $\cS$ in \eqref{opt:conv-fix-yI} corresponds to
$\cal B$ in \eqref{opt:convex-opt}. The cone $\cN_{\cal B}$ can be further decomposed as:
$\cN_{\cal B}=\cN_g(y^*)+\cN_Z(y^*)+\cN_{\mtc{Y}}(y^*)$, 
where $\cN_Z(y^*)$ and $\cN_{\mtc{Y}}(y^*)$ 
are respectively normal cones of the sets $\{y: y_Z=y^{(n)}_Z\}$
and $\mtc{Y}$ at $y^*$. The rest of the proof is similar as that of Proposition~\ref{prop:lin-approx-convex-prog}.  \qed}
\end{proof}

\begin{remark}\label{rmk:param-representation}
\qrevision{Due to Assumption~\ref{ass:subgrad-decomp},
the valid inequality $d^\top_y(y-y^*)\le0$ in \eqref{eqn:obj-equal} 
can be equivalently written as $d^\top_x(\hat{x}-\hat{x})+d^\top_y(y-y^*)\le0$,
for some vector $d_x$ such that the joint vector $d=(d_x;d_y)$ is in the normal
cone $\cN_G((\hat{x};y^*))$ where $G=\{(x;y): g_i(x,y)\le0\; \forall i\in\mtc{I}\}$.
Notice that the inequality $d^\top_x(x-\hat{x})+d^\top_y(y-y^*)\le0$ in the 
$(x,y)$-space are valid for the set $\mtc{S}$ due to the convexity of each $g_i$.}
\end{remark}

\qrevision{In Algorithm~\ref{alg:param-micp}, the inequality 
$d^\top_x(\hat{x}-\hat{x})+d^\top_y(y-y^*)\le0$
generated at iteration $k$ is denoted as 
\beq\label{eqn:Dx+Dy}
D^{(k)}_x\hat{x}+D^{(k)}_yy\le{d}^{(k)},
\eeq
where $D^{(k)}_x=d^\top_x$, $D^{(k)}_y=d^\top_y$ and $d^{(k)}=d^\top_x\hat{x}+d^\top_yy^*$.}

\begin{theorem}\label{thm:convg-gen-param-cut-plane}
\qrevision{Suppose Assumptions \ref{ass:non-empty relative interior} and \ref{ass:no-error} hold.
\newline
(1) All inequalities $a^{(k)\top}x+b^{(k)\top}y + c^{(k)} \le 0$
and $D^{(k)}_xx +D^{(k)}_yy \le d^{(k)}$ with the coefficients and matrices
defined by \eqref{eqn:coef-abc} and \eqref{eqn:Dx+Dy} generated 
in Algorithm~\ref{alg:param-micp} are valid for points in the set $\mtc{S}$ defined in \eqref{eqn:set_S}.}
\newline
(2) \qrevision{The generalized parametric cutting-plane algorithm (Algorithm~\ref{alg:param-micp}) 
terminates in a finite number of iterations with either
solving the parametric mixed-integer convex program \eqref{opt:y-micp2} to optimality 
or detecting it is infeasible.}
\end{theorem}
\begin{proof}
\qrevision{(1) First, we note that all cuts are generated at boundary points of $S$
to build its outer approximation. Therefore, they are valid for all points in $\mtc{S}$. 
(2) The proof is similar to the proof of Theorem~\ref{thm:finite-oracle-mixed-int-convex},
in which Lemma~\ref{lem:h_in_cone} is the counterpart of Proposition~\ref{prop:lin-approx-convex-prog}
used for establishing the optimality condition.  \qed}
\end{proof}

\begin{algorithm}
	\caption{An algorithm for solving a parametric mixed-integer convex program \eqref{opt:y-micp2}.}
	\label{alg:param-micp}
	\begin{algorithmic}[1]
        \State{{\textbf{Input:} $\hat{x}$}}
	\State{Set $n\gets 1$, $\mI^{(0)}\gets\emptyset$.}
	\Loop
		\State{(Start iteration $n$.)}
		\State{\textbf{Step 1:} Solve the master problem \eqref{opt:master-x-y} using the \textbf{oracle-PMILP} introduced at the beginning of Section~\ref{sec:param-cut-plane} to get an optimal solution $y^{(n)}$
        and optimal objective denoted as $L^{(n)}$. If \textbf{oracle-PMILP} detects that \eqref{opt:master-x-y} is infeasible, then claim \eqref{opt:y-micp2} is infeasible.}
		\State{Let $y^{(n)}=(y_Z^{(n)},y_R^{(n)})$, where $y_Z^{(n)} \in \mathbb{Z}^{l_2}$ 
		and $y_R^{(n)} \in \mathbb{R}^{l_3}$}.
		\If{$(\hat{x},y^{(n)})\in\mtc{S}$} 
			\State{Stop and return $y^{(n)}$ as an optimal solution of \eqref{opt:y-micp2};}
		\EndIf
		\State{\textbf{Step 2:} Solve \eqref{opt:param-proj} to get the projection point $(x^*,y^*)$ of $(\hat{x},y^{(n)})$ on $\cal S$ (defined in \eqref{eqn:set_S}).}
		\State{Add the separation (projection) cut $a^{(n)\top}\hat{x}+b^{(n)\top}y + c^{(n)} \le 0$ (defined in \eqref{eqn:proj_cut})
        			    to \eqref{opt:master-x-y}.} \label{lin:add_ax+by+c<0}
		\State{\textbf{Step 3:} Solve the continuous convex program \eqref{opt:conv-fix-yI}.}
		\If{\eqref{opt:conv-fix-yI} is feasible} 
			\State{Obtain an optimal solution $\hat{y}^{(n)}$ of \eqref{opt:conv-fix-yI} such that $(\hat{x},\hat{y}^{(n)})$ is at the boundary of the set}
			\State{\qrevision{$\{(x,y): Tx+Wy=q,\;g_i(x,y)\le0\;\forall i\in\mtc{I},\;x\in\mtc{X},\;y\in\mtc{Y}\}$.}}
			\State{\qrevision{This is achievable due to the linearity of the objective function.} }
			\State{Let $U^{(n)}$ be the optimal value of \eqref{opt:conv-fix-yI}.}
			\If{$L^{(n)}=U^{(n)}$} 
				\State{Stop and return $\hat{y}^{(n)}$ as an optimal solution of \eqref{opt:y-micp2}.}
			\EndIf
			\If{\qrevision{$\hat{y}^{(n)}$ satisfies that $(\hat{x},\hat{y}^{(n)})\in\text{bd}({\cal S})$ with $\cS$ defined in \eqref{eqn:set_S}.} } 
				\State{\qrevision{Generate the following outer-approximation inequalities defined in \eqref{eqn:Dx+Dy}:}} \label{lin:Dx+Dy}
				$$\qrevision{D^{(n)}_x\hat{x}+D^{(n)}_yy\le d^{(n)}}$$
				\State{and add them to \eqref{opt:master-x-y}.}  
				\State{Set $\mI^{(n)}\gets\mI^{(n-1)}\cup\{n\}$.}
                \Else 
                    \State{Set $\mI^{(n)}\gets\mI^{(n-1)}$.}
			\EndIf
		\Else
			\State{Set $\mI^{(n)}\gets\mI^{(n-1)}$. }
		\EndIf
		\State{$n\gets n+1$.}
	\EndLoop
	\end{algorithmic}
\end{algorithm}

\subsection{A decomposition cutting-plane algorithm for \eqref{opt:micp-xy}}
\label{sec:decomp-alg}
We now focus on developing a cutting-plane decomposition algorithm for solving
\eqref{opt:micp-xy} to optimality.
The master problem \eqref{opt:xy-master} can be iteratively updated by 
adding Benders' cuts generated from solving the second-stage problem \eqref{opt:y-micp}.
Specifically, Algorithm~\ref{alg:param-micp} is applied to solve \eqref{opt:y-micp} to optimality, and when the algorithm terminates, it generates a mixed-integer linear program (MILP) that
has the same optimal solution as \eqref{opt:y-micp}. \qrevision{The MILP is solved to optimality using  
\textbf{oracle-PMILP} introduced at the beginning of Section~\ref{sec:param-cut-plane}.
For a given $x$-space solution $x^m$, where $m$ is the main iteration index,
the \textbf{oracle-PMILP} generates a linear program \eqref{opt:y-lp} whose optimal solution is the same as 
\eqref{opt:y-micp}.} Then the Benders' cut \eqref{eqn:bender-cut} is generated by dualizing \eqref{opt:y-lp},
and it is added to the master problem for the main iteration $m+1$. 
The pseudo code of the cutting-plane decomposition algorithm is 
given in Algorithm~\ref{alg:decomp-cutting-plane}.

\begin{theorem}
Suppose Assumptions~\ref{ass:non-empty relative interior} and \ref{ass:no-error} hold.
The decomposition cutting-plane algorithm (Algorithm~\ref{alg:decomp-cutting-plane})
terminates in a finite number of iterations. When it terminates, it either
returns an optimal solution to \eqref{opt:micp-xy} or detects infeasibility.
\end{theorem}
\begin{proof}
We first show that the parametric cut generated for the first-stage problem
at every iteration is valid. It suffices to show that for any feasible solution $[x,y]$ 
of \eqref{opt:micp-xy}, the following inequality holds:
\beq
h^\top y \ge {a^m}^\top x +b^m \quad \forall m,
\eeq 
where the coefficients $a^m$ and $b^m$ are from the Benders' inequality 
$\eta\ge {a^m}^\top x+b^m$ in the master problem \eqref{opt:xy-master} 
for solving \eqref{opt:micp-xy}. 
The Benders' inequality is in the form \eqref{eqn:bender-cut}.
Recall that when the generalized parametric cutting-plane algorithm 
(Algorithm~\ref{alg:param-micp}) is applied to solve 
the second-stage problem \eqref{opt:y-micp}, 
it repeatedly solves a master problem \eqref{opt:master-x-y},
and it adds convexification cuts to update the master problem at every iteration. 
By Theorem~\ref{thm:convg-gen-param-cut-plane}, the convexification cuts added to the master problem are valid for the set $\mtc{S}$.
It implies that the linear program obtained upon termination of Algorithm~\ref{alg:param-micp}
has the form 
\beq\label{opt:min_hy_Cx+Dy<F}
\min h^\top y \quad \text{s.t. } \srevision{T x^m + W y = q,}\; C^m x^m + D^m y \le F^m,
\eeq
where $C^m$, $D^m$, $F^m$ are some coefficient matrices as in \eqref{opt:y-lp}.
Then it follows that the Benders' cut
$\eta\ge {a^m}^\top x+b^m$ generated from the linear program should be 
valid for the set $\{(x,y,\eta):\;h^\top y\le\eta, \; \srevision{T x + W y = q,} \;g_i(x,y)\le 0\;\forall i\in\mtc{I},\;y\in\mY\}$.
We also note that the feasibility cut $u^{k\top}x+v^k\ge0$ of the form \eqref{eqn:feasb-cut} 
only cuts off first-stage solutions that lead to infeasible second-stage problems and are valid for other solutions.

{Next, we show that the algorithm terminates in a finite number of iterations.
We prove it by contradiction.}
For the first-stage solution $x^m$ obtained at the main iteration $m$ of Algorithm~\ref{alg:decomp-cutting-plane},
let $\text{obj}^*_2(x^m)$ denote the optimal objective of the second-stage
problem \eqref{opt:y-micp} at $x=x^m$. Strong duality {of its linear relaxation~\eqref{opt:y-lp}} implies that
{$\text{obj}^*_2(x^m)={a^m}^\top x^m+b^m$}. Let $L^m$ and $U^m$ be the value of $L$ and $U$
at the end of the main iteration $m$. Suppose Algorithm~\ref{alg:decomp-cutting-plane}
does not terminate in a finite number of iterations. 
Since the feasible set {$\mtc{X}\cap\{0,1\}^{l_1}$} is finite, 
there exist $m$ and $n$ ($m>n$) such that $x^m=x^n$. We have
\bdm
\ba
U^m &= c^\top x^m+\text{obj}^*_2(x^m) \\
&\ge L^m= c^\top x^m+\eta^m \\
&\ge c^\top x^m + {a^n}^\top x^m+b^n \\
&=c^\top x^n + {a^n}^\top x^n+b^n \\
&=c^\top x^n + \text{obj}^*_2(x^n) \\
&=U^n,  
\ea
\edm
where we use the fact that $\eta^m\ge{a^n}^\top x^n+b^n$ since 
$\eta\ge {a^n}^\top x+b^n$ is a constraint in the master problem 
\eqref{opt:xy-master} at the main iteration $m$ of Algorithm~\ref{alg:decomp-cutting-plane}.
The above inequalities indicate that $U^m=L^m=U^n$, and hence the algorithm should terminate
at or before the main iteration $m$. 

{Finally, we show that either the algorithm returns an optimal solution or detects infeasibility. 
Note that when the algorithm terminates, exactly one of the two cases occurs based on the termination criteria:
Case 1: The master problem \eqref{opt:xy-master} is infeasible in the last iteration 
(Line~\ref{lin:alg_JMICP_infeasible} of Algorithm~\ref{alg:decomp-cutting-plane}).  
Since it is proved before that all parametric inequalities added to the master problem are valid for the first-stage,
the infeasibility of the master problem implies that \eqref{opt:micp-xy} does not have a feasible first-stage solution. 
Therefore, the algorithm detects infeasibility. 
Case 2: The algorithm returns a solution $(x^m,y^m)$ (Line~\ref{lin:return_sol} of Algorithm~\ref{alg:decomp-cutting-plane}). 
Our previous argument shows that $(x^m,y^m)$ is valid for \eqref{opt:micp-xy}.
Since the while-loop terminates, we conclude that $L=U$, which implies the optimality of the solution. }
 \qed
\end{proof}

\begin{algorithm}
	\caption{A decomposition cutting-plane algorithm for \eqref{opt:micp-xy}.}
	\label{alg:decomp-cutting-plane}
	\begin{algorithmic}[1]
	\State{Set $m\gets 1$, $L\gets -\infty$ and $U\gets\infty$.}
	\While{$L<U$}
		\State{Solve the current first-stage problem \eqref{opt:xy-master} for the main iteration $m$
			   using the cutting-plane oracle from \cite{ChenKucukyavuzSen2011}, and get the optimal solution $x^m$.
			   \qrevision{If \eqref{opt:xy-master} is infeasible, then stop and claim that \eqref{opt:micp-xy} is infeasible.} } \label{lin:alg_JMICP_infeasible}
		\State{Set $L\gets obj\eqref{opt:xy-master}$, where $obj\eqref{opt:xy-master}$ represents the optimal objective of \eqref{opt:xy-master}.}
		\State{Substitute $x^m$ into the parametric MICP \eqref{opt:y-micp}, 
			   and apply Algorithm~\ref{alg:param-micp} to solve it.}
		\State{{Let $y^m$ be an optimal solution of \eqref{opt:y-micp} given $x^m$.}}
		\State{Set $U\gets c^\top x^m +  obj\eqref{opt:y-micp}$.}
		\If{Algorithm~\ref{alg:param-micp} returns an optimal solution to \eqref{opt:y-micp}}
			\State{A linear program \eqref{opt:y-lp} is established upon termination. Generate
				  a Benders' cut \eqref{eqn:bender-cut} from the dual objective of \eqref{opt:y-lp},
			   	  and add it to \eqref{opt:xy-master} as $\eta\ge {a^m}^\top x+b^m$.}
		\Else
			\State{\eqref{opt:y-micp} is detected to be infeasible. Generate a feasibility cut $u^{m\top}x+v^m\ge0$ of the form \eqref{eqn:feasb-cut}, 
			 and add it to the master problem \eqref{opt:xy-master}.}
		\EndIf
		\State{Set $m\gets m+1$.}
	\EndWhile
	\State{\Return {$(x^m,y^m)$} as an optimal solution of \eqref{opt:micp-xy}.} \label{lin:return_sol}
	\end{algorithmic}
\end{algorithm}

\section{A Parametric {Cutting-plane} Algorithm for Solving 
Distributionally-Robust Two-Stage Stochastic Mixed-Integer Convex Programs}
\label{sec:cut-plane-two-stage-convex-prog}
The decomposition cutting-plane algorithm (Algorithm~\ref{alg:decomp-cutting-plane}) can be 
extended to solve the following class of distributionally-robust two-stage mixed-integer convex programs
with finite convergence:
\beq\label{opt:TSS-MICP}
\ba
&\min_{x}\; c^{\top}x + \max_{P\in\mP}\E_{\xi\sim P}[\mQ(x,\xi)] \\
&\;\;\textrm{s.t. }\; Ax\le b,\; x\in\mC,\; x\in\{0,1\}^{l_1}
\ea
\tag{DR-TSS-MICP}
\eeq
where $\mC$ is a convex set for the first-stage variable, $Ax\le b$
defines a polytope (i.e., bounded),
$\xi$ is a random vector having a finite support denoted as
$\Set*{\xi^{\omega}}{\omega\in\Omega}$ ($\Omega$ is a finite set of scenarios), 
$P$ is a candidate probability distribution of $\xi$, {$\mP \subset \mathbb{R}^{|\Omega|}$} is an ambiguity set of probability distributions, 
and $\mQ(x,\omega)$ is the recourse function at the 
scenario $\omega\in\Omega$, where $\Omega$ is a finite set. {We assume that $\mP$ is a compact set, and an oracle is available to optimize a linear function over this set. Examples of such ambiguity sets include a Wasserstein-metric based sets, moment based sets, total variance based sets, etc. that can be defined for $\xi$ \cite{rahimian2019_dro-review}}.   The recourse function $\mQ(x,\xi^\omega)$ is given as
\beq
\label{opt:second-stage}
\tag{\textrm{Sub}$(x,\omega)$}
\ba
&\mQ(x,\xi^\omega)=\min_{y^{\omega}}\; {h^\omega}^\top y^\omega \\
&\qquad\qquad\quad\textrm{ s.t. }
\srevision{T^{\omega} x + W^{\omega} y^{\omega} = q^{\omega}} {,} \\
&\qquad\qquad\qquad\quad  g^\omega_i(x,y^\omega)\le 0 \qquad\forall i\in\mtc{I}, \\
&\qquad\qquad\qquad\quad y^\omega\in\mtc{Y}^\omega\cap(\mbZ^{l_2}\times\mbR^{l_3}),
\ea
\eeq

where every $g^\omega_i$ ($i\in\mtc{I}$, $\omega\in\Omega$) is a convex (but not necessarily smooth) 
function in $(x,y^\omega)$, $\cal Y^\omega$ is a polytope.
We denote the minimization problem in \eqref{opt:second-stage} as {$\textrm{Sub}(x,\omega)$}.

\qrevision{The following assumption on \eqref{opt:TSS-MICP} is the counterpart of 
Assumption~\ref{ass:non-empty relative interior}.}
\begin{assumption}\label{ass:senario-ri}
\qrevision{For any $\hat{x}\in\mX\cap\{0,1\}^{l_1}$ and any $\hat{z}\in\mathbb{Z}^{l_2}$, 
either $\mtc{S}^\omega(\hat{x},\hat{z})$ is empty or it has a non-empty relative interior,
where $\mtc{S}^\omega(\hat{x},\hat{z})=\{y\in\mbR^{l_2+l_3}:\;T^\omega\hat{x} + W^\omega{y} = q,\;g^\omega_i(\hat{x},y)\le 0 \; 
\forall i\in\mtc{I},\;y\in\mtc{Y}^\omega,\; y_j=\hat{z}_j\;\forall j\in[l_2]\}$.}
\end{assumption}

\qrevision{The pure cutting-plane method for solving \eqref{opt:TSS-MICP}
is given in Algorithm~\ref{alg:decomp-BC} which is an extension of 
Algorithm~\ref{alg:decomp-cutting-plane} with multiple second-stage
scenario-based subproblems. Note that Algorithm~\ref{alg:decomp-cutting-plane} and  Algorithm~\ref{alg:decomp-BC} can be modified to retain all or some of the previously added outer-approximation cuts for the master problem. 
At the $m^{\textrm{th}}$ main iteration, the first-stage 
master problem of \eqref{opt:TSS-MICP} is given as 
\beq\label{opt:master}
\ba
&\min_{x\in\mC\cap\{0,1\}^{l_1}}\;c^{\top}x + \eta \\
&\quad\;\textrm{ s.t. } Ax\le b, \\
&\qquad\quad \eta\ge\sum_{\omega\in\Omega}p^{k\omega} r^{k\omega\top}x
+\sum_{\omega\in\Omega}p^{k\omega}u^{k\omega} \quad\forall k\in\mtc{R}^{m-1}, \\
&\qquad\quad v^{k\top}x+t^k\ge0 \qquad \qquad \qquad \qquad \quad \,\forall k\in\mtc{A}^{m-1}, 
\ea
\eeq
where $\eta\ge\sum_{\omega\in\Omega}p^{k\omega} r^{k\omega\top}x
+\sum_{\omega\in\Omega}p^{k\omega}u^{k\omega}$
is the Benders' cut and $v^{k\top}x+t^k\ge0$ is a feasibility cut  added 
to the master problem at the $k^{\textrm{th}}$ iteration of the main algorithm.
$\mtc{R}^{m-1}$ and $\mtc{A}^{m-1}$ are respectively the sets of iteration indices 
at which the Benders' cuts and feasibility cuts are added until the end of the main iteration $m$.

The derivation of these inequalities is given below. 
For a given first-stage solution $\hat{x}$, apply Algorithm~\ref{alg:param-micp}
to solve the second-stage subproblem Sub($\hat{x},\omega$).
If the algorithm claims that Sub($\hat{x},\omega$) is infeasible, we generate the following 
feasibility cut:
\beq\label{eqn:tss-feasb-cut}
\sum_{j\in{\mtc{J}_0}}x_j +  \sum_{j\in{\mtc{J}_1}}(1-x_j) \ge 1,
\eeq
where $\mtc{J}_0=\{j:\;\hat{x}_j=0\}$ and $\mtc{J}_1=\{j:\;\hat{x}_j=1\}$.
This inequality is denoted as $v^{k\top}x+t^k\ge0$ and added to the master problem \eqref{opt:master}.

On the other hand, suppose Sub($\hat{x},\omega$) is feasible.
Then Algorithm~\ref{alg:param-micp} returns an optimal solution of Sub($\hat{x},\omega$) and generates 
a linear program upon termination. Denote the linear program as:}
\beq\label{opt:two-stage-relax-full}
\ba
&\min_{y^\omega \in\mbR^{l_2+l_3}}\; h^{\omega\top}y^{\omega} \\
&\;\textrm{ s.t. }\;
 \srevision{W^{\omega} y^{\omega} = q^{\omega} - T^{\omega} \hat{x}, \quad\textrm{dual } \theta^\omega}\\
& \qquad\;\; Q^{\omega}y^\omega \ge s^\omega - R^{\omega}\hat{x}, \quad\textrm{dual }\mu^\omega\ge 0
\ea
\eeq
where $Q^\omega$, $R^\omega$ and $s^\omega$ are appropriate 
matrices and vectors of coefficients that represent the complete set of cutting planes at termination.
Suppose the constraints are associated with the dual variables $\theta^\omega$ and $\mu^\omega$.
The LP relaxation should give the same objective as \eqref{opt:second-stage}.
This property is formally stated as follows:
\begin{lemma}\label{lem:subprob-obj=Q}
\qrevision{Suppose Assumptions~\ref{ass:no-error} and \ref{ass:senario-ri} hold,}
and Algorithm~\ref{alg:param-micp} is applied to solve a parametric MICP problem
$\textrm{Sub}(\hat{x},\omega)$ for any feasible first-stage decision $\hat{x}$.
If $\textrm{Sub}(\hat{x},\omega)$ is feasible, then the termination criteria and 
Theorem~\ref{thm:finite-oracle-mixed-int-convex} ensures that the optimal objective of  
\eqref{opt:two-stage-relax-full} must be equal to $\mQ(\hat{x},\xi^\omega)$.
\end{lemma}
We determine a worst-case probability distribution for adding a cut to the first-stage solution $\hat{x}$ by solving the following optimization problem:
\beq\label{opt:worst-case-distr}
\max_{p\in\mP}\;\sum_{\omega\in\Omega}p^{\omega}\mQ(\hat{x},\xi^\omega)
\eeq
{The objective function in \eqref{opt:worst-case-distr} is linear, and we have assumed the existence of an oracle for solving \eqref{opt:worst-case-distr}.}
Finally, we generate an aggregated Benders' cut for the first-stage master problem
based on the optimal values of dual variables associated with the constraints 
\srevision{$W^{\omega} y^{\omega} = q^{\omega} - T^{\omega} \hat{x}$}, $Q^{\omega}y^\omega \ge s^\omega - R^{\omega}\hat{x}$
and the worst-case probability distribution $p$. Specifically, 
we can construct the following aggregated Benders' cut from~\eqref{opt:worst-case-distr}:
\beq\label{eqn:agg-bender-cut}
\eta\ge\srevision{-\left(\sum_{\omega\in\Omega}p^{\omega}\mu^{\omega\top}R^\omega + \sum_{\omega\in\Omega}p^{\omega}\theta^{\omega\top}T^\omega\right)x
+\left(\sum_{\omega\in\Omega}p^{\omega}\mu^{\omega\top}s^{\omega}+\sum_{\omega\in\Omega}p^{\omega}\theta^{\omega\top}q^{\omega}\right)},
\eeq 
where $p^\omega$ is the probability of scenario $\omega$ obtained
from solving \eqref{opt:worst-case-distr}. This cut is 
added to the first-stage master problem of \eqref{opt:TSS-MICP}.
\srevision{Let $r^{\omega} = -\left(\mu^{\omega\top}R^\omega + \theta^{\omega\top}T^\omega\right)^{\top}, u^{\omega} = \left(\mu^{\omega\top}s^{\omega}+\theta^{\omega\top}q^{\omega}\right)$.}

\begin{algorithm}
	\caption{A decomposition cutting-plane algorithm for solving \eqref{opt:TSS-MICP}.}
	\label{alg:decomp-BC}
	\begin{algorithmic}[1]
	\State{Initialization: $L\gets -\infty$, $U\gets \infty$, $m\gets 1$.}
	\While{$U-L>0$} \label{lin:U-L>0}
		\State{\qrevision{Apply Algorithm~\ref{alg:micp-oracle} to} solve the first-stage problem \eqref{opt:master} to 
		optimality with main iteration index being $m$. \qrevision{If \eqref{opt:master} is infeasible, then terminate and conclude
		that \eqref{opt:TSS-MICP} is also infeasible.}} \label{lin:solve-master}
		\State{Let $(\eta^m,x^m)$ be the optimal solution.} \label{lin:master-solution}
		\State{Update the lower bound as $L\gets c^\top x^{(m)}+\eta^{(m)}$.}
				\label{lin:low-bd-update}
		\State{Set the current best solution as $x^*\gets x^m$.}
		\For{$\omega\in\Omega$:}
			\State{Apply Algorithm~\ref{alg:param-micp} to solve the second-stage problem $\textrm{Sub}(x^m,\omega)$.} 
			\If{$\textrm{Sub}(x^m,\omega)$ is detected to be infeasible}
				\State{{\textbf{break} and add a feasibility cut \eqref{eqn:feasb-cut} to the first-stage problem \eqref{opt:master}.}}
			\Else
				\State{The $\textrm{Sub}(x^m,\omega)$ has been solved to optimality, and a LP relaxation 
				           \eqref{opt:two-stage-relax-full}}\label{lin:LP-relax}
				\State{with $\hat{x}=x^m$ has been generated by Algorithm~\ref{alg:param-micp}.}
				\State{Obtain the optimal values of the dual variables of \eqref{opt:two-stage-relax-full}}
				\State{and denote them as \srevision{$\lambda^{m\omega}$} and $\mu^{m\omega}$.}
			\EndIf
		\EndFor
		\State{Obtain the worst-case probability distribution 
		$p^m=\Set*{p^{m\omega}}{\omega\in\Omega}$ by solving \eqref{opt:worst-case-distr} with $\hat{x}\gets x^m$.}
		\State{Generate the following aggregated Benders' cut with iteration index $m$:} 
		\State{$\eta\ge \srevision{\sum_{\omega\in\Omega}p^{m\omega} r^{m\omega\top}x
+\sum_{\omega\in\Omega}p^{m\omega}u^{m\omega}}$.}
		\State{Add the Benders' cut to the first-stage problem \eqref{opt:master}.}
		\State{Update the upper bound as $U\gets \textrm{min}\{U,\; c^{\top}x^m+\sum_{\omega\in\Omega}p^{m\omega}\mQ(x^m,\xi^\omega)\}$}.
		\State{$m\gets m+1$.}
	\EndWhile
	\State{Return $x^*$.}
	\end{algorithmic}
\end{algorithm}

\begin{theorem}\label{thm:finite-convg} 
\qrevision{Suppose Assumptions~\ref{ass:no-error} and \ref{ass:senario-ri} hold.}
Algorithm~\ref{alg:decomp-BC} terminates in a finite number of iterations.
Upon termination, Algorithm~\ref{alg:decomp-BC} returns an optimal solution of \eqref{opt:TSS-MICP} if feasible; otherwise, it claims the problem is infeasible.
\end{theorem}
\begin{proof}
The proof is similar to the proof of Theorem~2.1 in \cite{luo2019_DR-TSS-MICP}.
Let the first-stage problem at main iteration $k$ be denoted by Master-$k$.
Lemma~\ref{lem:subprob-obj=Q} and the strong duality of 
\eqref{opt:two-stage-relax-full} imply that 
$\mQ(x^k,\omega)= \srevision{r^{k\omega\top}x^k
+u^{k\omega}}$.
Therefore, we have
\begin{equation}\label{eqn:maxEQ}
G(x^k):=\underset{P\in\mathcal{P}}{\textrm{max}}\;\E_P[\mQ(x^k,\xi)]=
\srevision{\sum_{\omega\in\Omega}p^{k\omega} r^{k\omega\top}x^k
+\sum_{\omega\in\Omega}p^{k\omega}u^{k\omega}}.
\end{equation}
Based on the mechanism of the algorithm, it is clear that if the algorithm terminates in finitely many iterations, it returns an optimal solution. 
We only need to show that the algorithm must terminate in finitely many iterations.
Assume this property does not hold.
Then it must generate an infinite sequence of first-stage solutions $\{x^k\}^{\infty}_{k=1}$. 
We must have ${k_1}$ and ${k_2}$ so that
$x^{k_1}=x^{k_2}$, with $k_1<k_2$. At the end of iteration $k_1$ the upper bound $U^{k_1}$ satisfies
\begin{equation}\label{eqn:U}
U^{k_1}= c^{\top}x^{k_1}+\sum_{\omega\in\Omega}p^{k_1\omega}Q(x^{k_1},\xi^\omega)=
c^{\top}x^{k_1} + G(x^{k_1}),
\end{equation}
where \eqref{eqn:maxEQ} is used to obtain the last equation.
The optimal value of Master-$k_2$ gives a lower bound $L^{k_2}=c^{\top}x^{k_2}+\eta^{k_2}$. 
Since $k_2>k_1$,  Master-$k_2$ has the following constraint:
\begin{equation}\label{eqn:eta}
\eta \ge \srevision{\sum_{\omega\in\Omega}p^{k_1\omega} r^{k_1\omega\top}x
+\sum_{\omega\in\Omega}p^{k_1\omega}u^{k_1\omega}}.
\end{equation}
Therefore, we conclude that
\begin{equation}\label{eqn:Lk2}
\begin{aligned}
L^{k_2}&=c^{\top}x^{k_2}+\eta^{k_2} \\
&\ge c^{\top}x^{k_2} + \srevision{\sum_{\omega\in\Omega}p^{k_1\omega} r^{k_1\omega\top}x^{k_2}
+\sum_{\omega\in\Omega}p^{k_1\omega}u^{k_1\omega}} \\
&=c^{\top}x^{k_1} + \srevision{\sum_{\omega\in\Omega}p^{k_1\omega} r^{k_1\omega\top}x^{k_1}
+\sum_{\omega\in\Omega}p^{k_1\omega}u^{k_1\omega}} \\
&=c^{\top}x^{k_1}+G(x^{k_1})=U^{k_1},
\end{aligned}
\end{equation}
where we use the fact that $x^{k_1}=x^{k_2}$, and  inequalities \eqref{eqn:U}--\eqref{eqn:eta} to obtain \eqref{eqn:Lk2}.
Hence, we have no optimality gap at solution $x^{k_1}$, and the algorithm should have terminated at 
or before iteration $k_2$.  \qed
\end{proof}

\section{A Finite Convergent Linearize-branch-and-union Decomposition Algorithm}\label{sec:decomp-branch-union}
\qrevision{The development in Sections~\ref{sec:alg-micp}-\ref{sec:cut-plane-two-stage-convex-prog}
extend the results of \cite{balas1993-lift-proj-cut-plane-0-1-LP}, \cite{owen2001disjunctive} and \cite{ChenKucukyavuzSen2011}
to establish finite-convergent pure cutting-plane approaches (without branch-and-cut) 
for solving the MICP, JMICP and DR-TSS-MICP to optimality under certain regularity conditions. 
In practical implementation, a trade-off is often made between cut generation and branch-and-bound.}
To enable branching in the implementation, we can make certain modifications on Algorithms~\ref{alg:param-micp} and \ref{alg:decomp-BC}
by leveraging the branch-and-bound technique, 
while still achieving finite convergence for eventually solving~\eqref{opt:micp-xy} and \eqref{opt:TSS-MICP}. 
\newline
\textbf{Modified-Algorithm~\ref{alg:param-micp}}: \newline
\textbf{Step~1}: Solve the master problem \eqref{opt:master-x-y} using
a \textit{branch-and-bound algorithm} to get an optimal solution $y^{(n)}$ 
and optimal objective denoted as $L^{(n)}$.
\textbf{Step~2} and \textbf{Step~3} are exactly the same as Algorithm~\ref{alg:param-micp}.

\qrevision{In the above Modified-Algorithm~\ref{alg:param-micp}, 
we replace the cutting-plane oracle (oracle-PMILP) with a branch-and-cut algorithm. 
This change means that we no longer obtain an equivalent LP at termination 
when problem \eqref{opt:master-x-y} is solved to optimality. 
Instead, branch-and-cut replaces \eqref{opt:master-x-y} 
by a disjunction of LPs corresponding to the leaf nodes of the tree {containing all the feasible solutions to \eqref{opt:master-x-y} regardless of their objective values.}
The disjunctive program can be dualized to generate a valid first-stage inequality
that is active at a given first-stage solution.   
We now outline a procedure for deriving a valid first-stage inequality which will be 
of the form \eqref{eqn:lambda*x-1>0}. This inequality will be aggregated and 
added to the first-stage when solving \eqref{opt:TSS-MICP} {(See Section~\ref{sec:cut-plane-two-stage-convex-prog})}. For simplicity of notation, we assume that no cuts are added when using the branching.
}

\begin{proposition}\label{prop:mod-alg-2}
\qrevision{Suppose Assumptions
\ref{ass:non-empty relative interior} and \ref{ass:no-error} hold.}
The Modified-Algorithm~\ref{alg:param-micp} can solve \eqref{opt:y-micp2}
to optimality or detect infeasibility in a finite number of iterations.
\end{proposition}
\begin{proof}
See Appendix~\ref{sec:proofs}.  \qed
\end{proof}

\begin{observation}\label{lem:node-constraints}
Consider \qrevision{Step 1 in the} Modified-Algorithm~\ref{alg:param-micp}.
The linear program induced by any node $\nu$ in the branch-and-bound tree generated for solving 
\eqref{opt:master-x-y} can be written in the following form:
\beq\label{opt:node-LP}
\ba
&\underset{y}{\min}\;\; h^\top y \\
&\;\mathrm{s.t.}\;\; \srevision{T \hat{x} + W y - q = 0,}\\
&\qquad R_\nu \hat{x} + Q_\nu y + s_\nu \le 0, \\
&\qquad y\in\mbR^{l_2+l_3},
\ea
\eeq
where $R_\nu$, $Q_\nu$ and $s_\nu$ are some coefficient matrices depending on the node $\nu$.
\end{observation}

Consider solving the problem \eqref{opt:micp-xy} in a decomposed manner
similar to the approach developed in Section~\ref{sec:decomp-cut-plane}
but using a different method (branch-and-cut method) for solving the 
parametric mixed-integer linear program. 
For a given input first-stage solution $x=\hat{x}$, 
when \eqref{opt:y-micp2} has been solved to optimality using the 
Modified-Algorithm~\ref{alg:param-micp}, we consider the set $\mtc{L}(\hat{x})$ of leaf nodes
of the branch-and-cut tree generated for solving the master 
problem \eqref{opt:master-x-y} in the last main iteration 
of the Modified-Algorithm~\ref{alg:param-micp}.  
Applying Observation~\ref{lem:node-constraints}, 
we can get the following LP formulation of each node $\nu\in\mtc{L}(\hat{x})$:
\beq\label{opt:pmicp-node-LP}
\min_y\;h^\top y\quad\textrm{ s.t. } \srevision{T \hat{x} + W y - q = 0,}\; R_{\nu}\hat{x} + Q_{\nu}y + s_{\nu}\le 0,\quad y\in\mbR^{l_2+l_3}.
\eeq
Let $y_\nu$ be the optimal solution of \eqref{opt:pmicp-node-LP}.
{Define the polytope $P_\nu$ indexed by leaf node} $\nu$ as follows:
\beq\label{eqn:P_nu}
P_\nu:=\Set*{{(x,y)}}{
 \begin{array}{c}
 \srevision{T x + W {y} = q } \\
 R_{\nu}x + Q_{\nu}{y} + s_{\nu}\le 0 \\
  \bs{0} \le x\le \bs{1}
 \end{array}
}.
\eeq  
Let $\Pi(\hat{x})=\cup_{\nu\in\mtc{L}(\hat{x})} P_\nu$ be the union
of polytopes \eqref{eqn:P_nu} induced by all nodes in $\mtc{L}(\hat{x})$. 
{Note that since each $P_\nu$ lies in the $(x,y)$-space, the set $\Pi(\hat{x})$ does as well.
Because the branch-and-bound tree structure is determined by the fixed value $\hat{x}$,
$\Pi(\hat{x})$ implicitly depends on $\hat{x}$. }
Let $\mathrm{Proj}_x(\conv(\Pi(\hat{x})))$ be  
the projection of $\conv(\Pi(\hat{x}))$ on the $x$-space. 
{The following proposition shows an important property of $\conv(\Pi(\hat{x}))$ that can be utilized to generate 
a valid first-stage inequality for \eqref{opt:micp-xy}. It shows that the the feasible set of \eqref{opt:micp-xy} is contained in the convex hull of the union of nodes determined by the second-stage branch-and-bound tree by $\hat{x}$}. {Consequently, any valid inequality generated for 
the convex hull $\conv(\Pi(\hat{x}))$ is also valid for $\vartheta$. It further implies that any valid inequality generated for  $\mathrm{Proj}_x(\conv(\Pi(\hat{x})))$ 
(the projection of $\conv(\Pi(\hat{x}))$ on the $x$-space)
is valid for $\mathrm{Proj}_x(\vartheta)$.}

{\begin{proposition}\label{prop:S-partition}
For any feasible first-stage solution $\hat{x}$ of \eqref{opt:micp-xy},
the set $\vartheta=\{(x,y):\;T x + W {y} = q, g_i(x,y)\le0\;\forall i\in\mtc{I},\;y\in\mtc{Y}\cap(\mbZ^{l_2}\times\mbR^{l_3})\} \subseteq \Pi(\hat{x})$ and hence in $\conv(\Pi(\hat{x}))$.
\end{proposition}
\begin{proof}
See Appendix~\ref{sec:proofs}.  \qed
\end{proof}}

Using the lift-and-project technique \cite{balas1985}, 
the set {$\mathrm{Proj}_x(\conv(\Pi(\hat{x}))$)}  can be represented as 
\begin{equation}
{\mathrm{Proj}_x(\conv(\Pi(\hat{x})))}=\mathrm{Proj}_{x}
\Set*{
\begin{aligned}
&x, {y}, x_\nu, {y_\nu}, \alpha_\nu \\
&\forall \nu\in\mtc{L}(\hat{x})
\end{aligned}
}
{
 \begin{aligned}
 &\srevision{T x_\nu + W y_\nu \alpha_\nu - q \alpha_\nu = 0}\\
 &R_\nu x_\nu + Q_\nu y_\nu + s_\nu\alpha_\nu \le 0 \\
 &\sum_{\nu\in\mtc{L}(\hat{x})}\alpha_\nu=1,\\
 &x = \sum_{\nu\in\mtc{L}(\hat{x})} x_\nu, \\
 &{y = \sum_{\nu\in\mtc{L}(\hat{x})} y_\nu,} \\
 &\bs{0} \le x_\nu \le \bs{1}\alpha_\nu, \\
 &\alpha_\nu\ge 0 \; \forall \nu\in\mtc{L}(\hat{x})
 \end{aligned}
}.
\end{equation}
We can parameterize all the valid inequalities for $\conv(\Pi(\hat{x}))$ by dualizing
the constraints
\bdm
\ba
&{\lambda^\top_x}\left(x - \sum_{\nu \in \mtc{L}(\hat{x})} x_\nu\right)+{\lambda^\top_y\left(y - \sum_{\nu \in \mtc{L}(\hat{x})} y_\nu\right)} \\
&+\sum_{\nu \in \mtc{L}(\hat{x})} \zeta_\nu^{\top} (T x_\nu + W y_\nu - q \alpha_\nu) -\sum_{\nu \in \mtc{L}(\hat{x})} \mu^\top_\nu \left(R_\nu x_\nu + Q_\nu y_\nu + s_\nu\alpha_\nu\right) \\
& + \sum_{\nu \in \mtc{L}(\hat{x})}\gamma^\top_\nu(\bs{1}\alpha_\nu-x_\nu)
+ \sum_{\nu \in \mtc{L}(\hat{x})}\alpha_\nu -1 \ge 0 \\ 
&\implies \\
&{\lambda^\top_x} x - \sum_{\nu \in \mtc{L}(\hat{x})}({\lambda^\top_x} - \zeta_\nu^\top T + \mu^\top_\nu R_\nu + \gamma^\top_\nu)x_\nu \\
&{\lambda^\top_yy-\sum_{\nu \in \mtc{L}^\omega(\hat{x})}(\lambda^\top_y-\zeta^\top{W}+\mu^\top_\nu Q_\nu)y_\nu} \\
& + \sum_{\nu \in \mtc{L}(\hat{x})}(1 - \zeta_\nu^\top q -\mu^\top_\nu s_\nu+\gamma^\top_\nu\bs{1}) \alpha_\nu -1 \ge 0
\ea
\edm
The above Lagrangian inequality indicates that any inequality of the form 
\beq\label{eqn:lambda*x-1>0}
{\lambda^\top_x} x - 1 \ge 0
\eeq
is valid for {$\mathrm{Proj}_x(\conv(\Pi(\hat{x})))$}  if and only if the coefficients $\lambda$ satisfy the following
constraints:
\beq
\ba
&{\lambda_x} - T^\top \zeta_\nu + R^\top_\nu\mu_\nu  + \gamma_\nu \ge 0, \\
&{\lambda_y=0,\quad -\zeta^\top_\nu{W}+\mu^\top_\nu Q_\nu=0}, \\
&\zeta_\nu^\top q + \mu^\top_\nu s_\nu -1 - \gamma^\top_\nu\bs{1}\ge 0, \\
&{\lambda_x\;\textrm{free},\;\zeta_\nu\;\textrm{free}},\; \mu_\nu,\gamma_\nu\ge 0 \;\forall \nu\in\mtc{L}(\hat{x}).
\ea
\eeq 
In particular, we can choose an appropriate set of coefficients $\lambda$, $\mu_\nu$ 
and $\gamma_\nu$ by solving the following linear program:
\beq
\ba
&\min\;\;{\lambda^\top_x} \hat{x}  \\
&\textrm{ s.t. }\;\; {\lambda_x} - T^\top \zeta_\nu + R^\top_\nu\mu_\nu  + \gamma_\nu \ge 0, \\
&\qquad\;\;{-\zeta^\top_\nu{W}+\mu^\top_\nu Q_\nu=0}, \\
&\qquad\;\; \zeta_\nu^\top q + \mu^\top_\nu s_\nu -1 - \gamma^\top_\nu\bs{1}\ge 0, \\
&\qquad\;\; {\lambda_x\;\textrm{free},\;\zeta_\nu\;\textrm{free}},\; \mu_\nu,\gamma_\nu\ge 0 \;\forall \nu\in\mtc{L}(\hat{x}).
\ea
\eeq 

\subsection{Application of the linearize-branch-and-union algorithm to the second-stage scenario problem \eqref{opt:second-stage}}
\begin{observation}\label{cor:linearize-and-branch}
Consider applying Modified-Algorithm~\ref{alg:param-micp} to solve 
the second-stage scenario problem \eqref{opt:second-stage} with
the scenario index $\omega$ and $x=\hat{x}$ to optimality.
Each leaf node $\nu$ of the branch-and-cut tree should have the following form of linear program: 
\beq\label{eqn:lem-nd-LP}
\underset{y}{\min}\; h^{\omega \top} y^\omega\quad
\mathrm{ s.t. }\; \srevision{T^\omega \hat{x} + W^\omega y^\omega - q^\omega = 0,}\; R^\omega_\nu\hat{x}+Q^\omega_\nu y+s^\omega_\nu \le 0,
\; y\in\mbR^{l_2+l_3}.
\eeq
for some coefficient matrix $R^\omega_\nu$, $Q^\omega_\nu$ and $s^\omega_\nu$ 
depending on $\omega$ and $\nu$.
\end{observation}
\begin{remark}
The key feature of \eqref{eqn:lem-nd-LP} is that the linear inequalities 
are parametrized by the first-stage variable value $\hat{x}$, which is the 
foundation of generating a parametric cut for the first-stage problem that
will be discussed later.
\end{remark}
Consider the problem \eqref{opt:TSS-MICP}. When a first-stage solution
$\hat{x}$ is given, we then solve the (second-stage) scenario sub-problem 
$\textrm{Sub}(\hat{x},\omega)$ using the Modified-Algorithm~\ref{alg:param-micp}. 
Applying Observation~\ref{cor:linearize-and-branch} to each scenario sub-problem, 
we can get the following LP formulation of each node $\nu\in\mtc{L}^\omega(\hat{x})$: 
\beq\label{eqn:cor-nd-LP}
\underset{y}{\min}\; h^{\omega \top} y^\omega\quad
\mathrm{ s.t. }\; \srevision{T^\omega \hat{x} + W^\omega y^\omega - q^\omega = 0,}\; R^\omega_\nu\hat{x}+Q^\omega_\nu y+s^\omega_\nu \le 0,
\; y\in\mbR^{l_2+l_3},
\eeq
where $\mtc{L}^\omega(\hat{x})$ is the set of nodes of the branch-and-bound tree of 
the scenario $\omega$ sub-problem $\textrm{Sub}(\hat{x},\omega)$, 
and $\nu$ is the index of a node in $\mtc{L}^\omega(\hat{x})$.

{
Proposition~\ref{prop:S^omega-partition} is an application of Proposition~\ref{prop:S-partition} for a two-stage stochastic program for a specific scenario.
}

\begin{proposition}\label{prop:S^omega-partition}
Let $\vartheta^\omega=\{(x,y^\omega):\;g_i(x,y^\omega)\le0\;\forall i\in\mtc{I},\;y^\omega\in\mtc{Y}\}$.
For a feasible first-stage solution $\hat{x}$, let 
$P^\omega_\nu(\hat{x})=\left\{(x,y^\omega):R^\omega_\nu x+Q^\omega_\nu y^\omega+s^\omega_\nu\le 0\right\}$
for every $\nu\in\mtc{L}^\omega(\hat{x})$.
Then $\vartheta^\omega\subseteq \cup_{\nu\in\mtc{L}^\omega(\hat{x})}P^\omega_\nu(\hat{x})$. 
\end{proposition}
\begin{proof}
{Follows from Proposition~\ref{prop:S-partition}.  \qed}
\end{proof}

Suppose $x^*$ is an
optimal first-stage solution of \eqref{opt:TSS-MICP} and $y^{\omega*}$ is an optimal
solution of $\textrm{Sub}(x^*,\omega)$. Since $(x^*,y^{\omega*})\in\vartheta^\omega$,
Proposition~\ref{prop:S^omega-partition} indicates that $\exists\nu^*\in\mtc{L}^\omega(\hat{x})$
such that $(x^*,y^{\omega*})\in P^\omega_{\nu^*}(\hat{x})$.
The LP relaxation problem at a node $\nu\in\mtc{L}^\omega(\hat{x})$ is
\beq\label{opt:node-LP-2}
\ba
&\min_{y^\omega_\nu}\; h^{\omega\top}y^\omega_\nu \\
&\;\textrm{s.t. }\; 
\srevision{W^\omega y^\omega_\nu = - T^\omega \hat{x} + q^\omega}, \qquad\;\;\textrm{dual }\theta^\omega_{\nu} \\
&\qquad\; Q^\omega_{\nu}y^\omega_\nu\le  -R^\omega_{\nu}\hat{x} - s^\omega_\nu, \qquad\textrm{dual }\mu^\omega_{\nu} \\ 
&\qquad\; y^\omega_\nu\in\mbR^{l_2+l_3},
\ea 
\eeq
where $\theta^\omega_\nu$ and $\mu^\omega_\nu\ge 0$ are the optimal values of the dual vector obtained after solving the node LP relaxation problem. The optimal objective of \eqref{eqn:lem-nd-LP} can be represented
using the dual vector as $\theta_\nu^{\omega \top}(- T^\omega \hat{x} + q^\omega) + \mu^{\omega\top}_\nu(R^\omega_{\nu}\hat{x} + s^\omega_\nu)$.
Define the polyhedron $V^\omega_\nu(\hat{x})$ as follows:
\beq\label{eqn:P^omega_nu}
V^\omega_\nu(\hat{x}):=\Set*{(x,\eta^\omega)}{
 \begin{array}{c}
  \eta^\omega\ge\srevision{\left(\mu^{\omega\top}_\nu R^\omega_{\nu} - \theta_\nu^{\omega \top} T^\omega \right) x + \left(\mu^{\omega\top}_\nu s^\omega_\nu  + \theta_\nu^{\omega \top} q^\omega\right)}\\
  Ax\le b \\
  \bs{0} \le x\le \bs{1}
 \end{array}
}.
\eeq 
\qrevision{
\begin{proposition}\label{obs:[x,eta]inV}
Let $x^*$ be an optimal first-stage solution of \eqref{opt:TSS-MICP} and let $\eta^{\omega*}$
represent the optimal objective of $\mathrm{Sub}(x^*,\omega)$.
Then we must have $(x^*,\eta^{\omega*})\in V^\omega_{\nu^*}(\hat{x})$.
\end{proposition}
\begin{proof}
See Appendix~\ref{sec:proofs}.  \qed
\end{proof}
}
Let $\Pi^\omega(\hat{x})=\cup_{\nu\in\mtc{L}^\omega(\hat{x})} V^\omega_\nu(\hat{x})$.
\qrevision{Proposition~\ref{obs:[x,eta]inV} implies that $(x^*,\eta^{\omega*})\in\Pi^\omega(\hat{x})$,
and hence $(x^*,\eta^{\omega*})\in\conv(\Pi^\omega(\hat{x}))$.}
The convex hull $\conv(\Pi^\omega(\hat{x}))$ can be represented following
\cite{balas1993-lift-proj-cut-plane-0-1-LP,balas1998} as 
\begin{equation}
\conv(\Pi^\omega(\hat{x}))=\mathrm{Proj}_{x,\eta^\omega}
\Set*{
\begin{aligned}
&\eta^\omega, x, \\
&\eta^\omega_\nu, x_\nu, \alpha_\nu \\
&\forall \nu\in\mtc{L}^\omega(\hat{x})
\end{aligned}
}
{
 \begin{aligned}
 &\eta^\omega = \sum_{\nu\in\mtc{L}^\omega(\hat{x})} \eta^\omega_\nu \\
 &x = \sum_{\nu\in\mtc{L}^\omega(\hat{x})} x_\nu \\
 &\eta^\omega_\nu\ge\srevision{\left(\mu^{\omega\top}_\nu R^\omega_{\nu} - \theta_\nu^{\omega \top} T^\omega \right) x_\nu }\\
 & \;\;\srevision{+ \left(\mu^{\omega\top}_\nu s^\omega_\nu  + \theta_\nu^{\omega \top} q^\omega\right) \alpha_\nu}\\
 &Ax_\nu \le b\alpha_\nu,\;\bs{0} \le x_\nu \le \bs{1}\alpha_\nu \\
 &\sum_{\nu\in\mtc{L}^\omega(\hat{x})}\alpha_\nu=1,\;\alpha_\nu\ge 0 \\
 &\qquad\forall \nu\in\mtc{L}^\omega(\hat{x})
 \end{aligned}
}
\end{equation}
\begin{observation}\label{obs:(x,eta)_in_conv}
\qrevision{Let $\hat{x}$ and $\widetilde{x}$ be any two feasible first-stage solutions of \eqref{opt:TSS-MICP}, 
and $\widetilde{\eta}^{\omega}$ be the optimal objective of $\textrm{Sub}(\widetilde{x},\omega)$. 
Then $(\widetilde{x},\widetilde{\eta}^{\omega})\in\conv(\Pi^\omega(\hat{x}))$.}
\end{observation}
We can parameterize all the valid inequalities for $\conv(\Pi^\omega)$ by dualizing
the constraints in the formulation of $\conv(\Pi^\omega)$ using the dual variables
$\lambda,\beta_\nu,\gamma_\nu,\sigma$ as follows:
\bdm
\ba
&\eta^\omega - \sum_{\nu \in \mtc{L}^\omega(\hat{x})} \eta^\omega_\nu + \lambda^\top\left(\sum_{\nu \in \mtc{L}^\omega(\hat{x})} x_\nu - x\right)
+\sum_{\nu \in \mtc{L}^\omega(\hat{x})} \big(\eta^\omega_\nu - \mu^{\omega\top}_\nu R^\omega_\nu x_\nu + \theta_\nu^{\omega \top} T^\omega x_\nu - \mu^{\omega\top}_\nu s^\omega_\nu \alpha_\nu \\
&  \srevision{- \theta_\nu^{\omega \top} q^\omega \alpha_\nu} \big) + \sum_{\nu \in \mtc{L}^\omega(\hat{x})}\beta^\top_\nu(b\alpha_\nu-Ax_\nu) + \sum_{\nu \in \mtc{L}^\omega(\hat{x})}\gamma^\top_\nu(\bs{1}\alpha_\nu-x_\nu)
+\sigma\left(\sum_{\nu \in \mtc{L}^\omega(\hat{x})}\alpha_\nu -1 \right) \ge 0 \\ 
&\implies \\
&\eta^\omega -\lambda x + \sum_{\nu \in \mtc{L}^\omega(\hat{x})}(\lambda^\top-\mu^{\omega\top}_\nu R^\omega_\nu + \theta_\nu^{\omega \top} T^\omega - \beta^\top_\nu A-\gamma^\top_\nu)x_\nu \\
&\quad+ \sum_{\nu \in \mtc{L}^\omega(\hat{x})}(\sigma-\mu^{\omega\top}_\nu s^\omega_\nu - \theta_\nu^{\omega \top} q^\omega +\beta^\top_\nu b+\gamma^\top_\nu\bs{1}) \alpha_\nu - \sigma \ge 0.
\ea
\edm
The above Lagrangian inequality indicates that any inequality of the form 
\beq
\eta^\omega\ge\lambda^\top x+\sigma
\eeq
is valid for $\conv(\Pi^\omega)$ if and only if the dual variables satisfy the following
constraints:
\beq\label{eqn:constr-dual-var}
\ba
&\lambda-R^{\omega\top}_\nu\mu^{\omega}_\nu \srevision{+ \theta_\nu^{\omega \top} T^\omega} - A^\top\beta_\nu-\gamma_\nu \le 0, \\
&\sigma-\mu^{\omega\top}_\nu s^\omega_\nu \srevision{- \theta_\nu^{\omega \top} q^\omega} +\beta^\top_\nu b+\gamma^\top_\nu\bs{1}\le 0, \\
&\lambda\in\mathbb{R}^{l_1},\;\sigma\in\mathbb{R},\; \beta_\nu,\gamma_\nu\ge 0 \,\;\forall \nu\in\mtc{L}^\omega(\hat{x}).
\ea
\eeq 
In particular, we can choose an appropriate set of dual variables by solving the following
linear program:
\beq\label{eqn:opt-scen-cut}
\ba
&\max\quad \lambda^\top \hat{x} + \sigma \\
&\textrm{ s.t. }\quad \lambda-R^{\omega\top}_\nu\mu^{\omega}_\nu \srevision{+ \theta_\nu^{\omega \top} T^\omega} - A^\top\beta_\nu-\gamma_\nu \le 0, \\
&\qquad\quad \sigma-\mu^{\omega\top}_\nu s^\omega_\nu \srevision{- \theta_\nu^{\omega \top} q^\omega} +\beta^\top_\nu b+\gamma^\top_\nu\bs{1}\le 0, \\
&\qquad\quad \lambda\in\mathbb{R}^{l_1},\;\sigma\in\mathbb{R},\; \beta_\nu,\gamma_\nu\ge 0 \,\;\forall \nu\in\mtc{L}^\omega(\hat{x}).
\ea
\eeq 
The following observation summarizes this property.
\qrevision{
\begin{observation}\label{obs:eta>lambda*x+sigma}
The inequality $\eta^\omega\ge\lambda^\top{x}+\sigma$ is valid for $\conv(\Pi^\omega)$ if the dual variables satisfy \eqref{eqn:constr-dual-var}.
\end{observation}
}
\begin{lemma}\label{lem:scen-first-stage-valid-ineq}
Let $\hat{x}$ be any feasible first-stage solution of \eqref{opt:TSS-MICP} and $\omega\in\Omega$
be a scenario. Suppose the second-stage problem $\mathrm{Sub}(\hat{x},\omega)$, i.e., 
the problem \eqref{opt:second-stage} with $x=\hat{x}$ is feasible and has been solved to optimality using 
a linearize-and-branch algorithm, and $\mtc{L}^\omega(\hat{x})$ be the set of leaf nodes of the branch-and-bound
tree available at termination of the linearize-and-branch algorithm. Then we have 
\beq\label{eqn:scen-first-stage-valid-ineq}
\ba
&\mQ(\hat{x},\xi^\omega)=\hat{\lambda}^{\omega\top}\hat{x}+\hat{\sigma}^{\omega}, \\
&\mQ(x,\xi^\omega)\ge\hat{\lambda}^{\omega\top}x+\hat{\sigma}^{\omega},
\ea
\eeq
for all feasible first-stage solutions $x$,
where the coefficients $\hat{\lambda}^{\omega}$ and $\hat{\sigma}^{\omega}$
are the optimal solution of \eqref{eqn:opt-scen-cut}.
\end{lemma}
\begin{proof}
By Observation~\ref{obs:eta>lambda*x+sigma}, 
the inequality $\eta\ge\hat{\lambda}^{\omega\top}x+\hat{\sigma}^{\omega}$ is valid for 
all points in $\conv(\Pi^\omega)$ and by Observation~\ref{obs:(x,eta)_in_conv}, 
for any feasible first-stage solution $x$, the point $(x,\eta^\omega)\in\conv(\Pi^\omega)$,
where $\eta^\omega=\mQ(x,\xi^\omega)$ is the optimal objective of $\mathrm{Sub}(x,\omega)$.
This shows that $\mQ(x,\xi^\omega)\ge\hat{\lambda}^{\omega\top}x+\hat{\sigma}^{\omega}$ holds
for any feasible first-stage solution $x$.

It remains to show that the first equation in \eqref{eqn:scen-first-stage-valid-ineq} holds. 
We claim that $(\hat{x},\hat{\eta}^\omega)$ is an extremal point of $\conv(\Pi^\omega)$, 
where $\hat{\eta}^\omega=\mQ(\hat{x},\omega)$.
To prove the claim, we observe that 
$\mathrm{Proj}_x\left(V^\omega_{\nu}(\hat{x})\right)=\{x:\;Ax\le{b},\;{\textbf 0}\le{x}\le{\textbf 1}\}$
for all $\nu\in\mtc{L}^\omega(\hat{x})$, and hence
$\mathrm{Proj}_x\left(\conv(\Pi^\omega)\right)=\{x:\;Ax\le{b},\;{\textbf 0}\le{x}\le{\textbf 1}\}$.
Furthermore, since $\hat{x}$ is a feasible first-stage solution which is also a binary vector, 
it must be an extremal point of the set $\{x:\;Ax\le{b},\;{\textbf 0}\le{x}\le{\textbf 1}\}$.
Let $\nu^\prime\in\mtc{L}^\omega(\hat{x})$ be the node that contains the optimal solution
$\hat{y}$ of the problem $\textrm{Sub}(\hat{x},\omega)$. 
It follows that the point $(\hat{x},\hat{\eta}^\omega)$ is an extremal point of the set 
$V^\omega_{\nu^\prime}(\hat{x})$ and hence an extremal point of the set $\conv(\Pi^\omega)$.
It is also clear that $(\hat{x},\hat{\eta}^\omega)$ is the only extremal point of $\conv(\Pi^\omega)$
with $x=\hat{x}$. We can also see that
\bdm
\hat{\eta}^\omega = \min \{\eta: (\hat{x},\eta)\in\conv(\Pi^\omega)\},
\edm
and the valid inequality $\eta\ge\hat{\lambda}^{\omega\top}x+\hat{\sigma}^{\omega}$
passes from the extremal point $(\hat{x},\hat{\eta}^\omega)$ due to the use of \eqref{eqn:opt-scen-cut}
for generating this inequality. Therefore, we have 
$\hat{\eta}^\omega=\hat{\lambda}^{\omega\top}\hat{x}+\hat{\sigma}^{\omega}$.  \qed
\end{proof}

Lemma~\ref{lem:scen-first-stage-valid-ineq} indicates that we can aggregate
the scenario valid inequality $\eta^\omega\ge\lambda^{\omega*\top}x+\sigma^{\omega*}$
with the weight given by the worst-case probability distribution to generate a valid inequality 
to add to the first-stage master problem. With this property,
we can then establish a decomposition linearize-branch-and-union based algorithm (Algorithm~\ref{alg:decomp-BU}) 
for solving \eqref{opt:TSS-MICP}. The idea is similar to Algorithm~\ref{alg:decomp-BC}.
The difference is in the approach to solving the scenario sub-problems to generate
a valid scenario cut. In Algorithm~\ref{alg:decomp-BC}, the scenario sub-problem is solved using a pure parametric cutting-plane algorithm, and the scenario cut is generated by taking
the dual of the eventual LP relaxation. While in Algorithm~\ref{alg:decomp-BU}, it is 
solved using the linearize-and-branch algorithm, and the scenario cut is generated 
using a polytope-union technique. Together we call it the decomposition algorithm with linearize-branch-and-union technique. 

In particular, the master problem at iteration $m$ of Algorithm~\ref{alg:decomp-BU}  
has the following formulation:
\beq\label{opt:master-BU}
\ba
&\min_{x}\;c^{\top}x + \eta \\
&\;\textrm{ s.t. } Ax\le b, \\
&\qquad\; \eta\ge\sum_{\omega\in\Omega}p^{k\omega}\lambda^{k\omega*\top} x
+\sum_{\omega\in\Omega}p^{k\omega}\sigma^{k\omega*} \quad\forall k\in\mtc{R}^{m-1}, \\
&\qquad\; v^{k\top}x+t^k\ge0 \qquad \qquad \qquad \qquad \qquad \,{ \forall k\in\mtc{A}^{m-1}},\\
&\qquad\; x\in\mC\cap\{0,1\}^{l_1},
\ea
\eeq
where $\eta\ge\sum_{\omega\in\Omega}p^{k\omega}\lambda^{k\omega*\top} x
+\sum_{\omega\in\Omega}p^{k\omega}\sigma^{k\omega*}$ is the (aggregated)
first-stage valid inequality (Benders' cut) generated at the end of iteration $k$ and 
$p^k=\Set*{p^{k\omega}}{\omega\in\Omega}$ is an optimal solution of
\eqref{opt:worst-case-distr} with $\hat{x}\gets x^k$, and $x^k$ is an
optimal solution of the master problem \eqref{opt:master-BU} at the beginning 
of iteration $k$. \qrevision{Other notations in \eqref{opt:master-BU} are defined the same as \eqref{opt:master}.}

\begin{theorem}
\qrevision{Suppose Assumptions~\ref{ass:no-error} and \ref{ass:senario-ri} hold.}
If the master problem and scenario sub-problems are solved to optimality \qrevision{or detected infeasible} 
in every iteration of Algorithm~\ref{alg:decomp-BU}, 
the algorithm {returns} an optimal solution of \eqref{opt:TSS-MICP}
\qrevision{or detect its infeasibility} in a finite number of iterations. 
\end{theorem}
\begin{proof}
The proof is similar to the proof of Theorem 4.1 in \cite{luo2019_DR-TSS-MICP}.  \qed
\end{proof}

\begin{algorithm}
	\caption{A decomposition linearize-branch-and-union based algorithm for solving \eqref{opt:TSS-MICP}.}
	\label{alg:decomp-BU}
	\begin{algorithmic}[1]
	\State{Initialization: $L\gets -\infty$, $U\gets \infty$, $m\gets 1$.}
	\While{$U-L>0$} 
		\State{Solve the first-stage master problem \eqref{opt:master-BU} to 
		optimality with main iteration index being $m$. \qrevision{If \eqref{opt:master-BU} is infeasible,
		then stop and claim that \eqref{opt:TSS-MICP} is infeasible.}} 
		\State{Let $(\eta^m,x^m)$ be the optimal solution.} 
		\State{Update the lower bound as $L\gets c^\top x^{m}+\eta^{m}$.}
		\State{Set the current best solution as $x^*\gets x^m$.}
		\For{$\omega\in\Omega$:}
			\State{Use \textbf{Modified-Algorithm~\ref{alg:param-micp}} to solve the second-stage problem $\textrm{Sub}(x^m,\omega)$.} 
            {
            \If{$\textrm{Sub}(x^m,\omega)$ is detected to be infeasible}
				\State{\textbf{break} and add a feasibility cut \eqref{eqn:feasb-cut} to the first-stage master problem \eqref{opt:master-BU}.}
            \EndIf}
			\State{When the Modified-Algorithm~\ref{alg:param-micp} terminates with an optimal solution, it generates a 
			branch-and-bound tree with the set $\mtc{L}^\omega(x^m)$ of nodes. } 
			\State{Solve \eqref{eqn:opt-scen-cut} with $\hat{x}\gets x^m$ to determine the coefficients 
			$\lambda^{m\omega*}$ and $\sigma^{m\omega*}$.}
		\EndFor
		\State{Obtain the worst-case probability distribution 
		$p^m=\Set*{p^{m\omega}}{\omega\in\Omega}$ by solving \eqref{opt:worst-case-distr} with $\hat{x}\gets x^m$.}
		\State{Generate the following aggregated Benders' cut with iteration index $m$:
        $\eta\ge\sum_{\omega\in\Omega}p^{k\omega}\lambda^{k\omega*\top} x
+\sum_{\omega\in\Omega}p^{k\omega}\sigma^{k\omega*}$.}
		\State{Add the Benders' cut to the first-stage master problem \eqref{opt:master-BU}.}
		\State{Update the upper bound as $U\gets \textrm{min}\{U,\; c^{\top}x^m+\sum_{\omega\in\Omega}p^{m\omega}\mQ(x^m,\xi^\omega)\}$}.
		\State{$m\gets m+1$.}
	\EndWhile
	\State{Return $x^*$.}
	\end{algorithmic}
\end{algorithm}

\section{Numerical Experimentation}
\label{sec:num_exp}
To evaluate the practical performance of our proposed algorithm, we apply it to a process-integrated chemical supply chain planning problem, adapted from \cite{li2019finite, li2018improved, norton1994strategic}. The main adaptation involves using the Cobb-Douglas production function \cite{cobb1928theory, dey2024algorithm} for processes exhibiting a nonlinear input-output relationship between the main product and the associated chemical ingredients. We also consider the plant's capacity variables as discrete so that all first-stage variables are binary. For completeness, the adapted model details are presented in Appendix~\ref{sec:Num-Details}. The complete lists of notations for indices, sets, parameters, and decision variables are given in  Tables~\ref{tab:index-set}, \ref{tab:model-parameter}, \ref{tab:model-decision-var}, and \ref{tab:dimension-parameter}, respectively. These are followed by the mathematical model in Appendix Section~\ref{sec:chem-model}. Figures~\ref{fig:ChemSC-Network} and ~\ref{fig:ChemSC} respectively present the general network connections of suppliers and customers with a plant, and the process flow schematic (applicable for all plants). In Appendix Section~\ref{sec:implementation details}, we provide further details used in our experimental setup.

\subsection{Results and discussion}  
We first consider a two-stage stochastic setting and solve the problem for five different scenarios both using our method and commercial solver Gurobi (v11.0.0). All the computations are performed on a server having Xeon 2.20 GHz CPUs.  For each instance, scenario problems are solved in parallel, allocating one core per scenario, with a maximum of 50 cores. Our implementation uses Python 3.9.  The `multiprocessing' package is used for assigning a scenario problem to a CPU core and thus solving the scenario problems in parallel. The Gurobi solver (`gurobipy' package) is used for solving mixed-integer linear programming master problems, the LP sub-problems arising in the {branch-and-bound (B$\&$B)} tree, and cut-generation-linear-programs {(CGLPs)} under each scenario problem. 

In the master problem, we use cover inequalities to cut out an infeasible solution.  The Modified-Algorithm 2 outlined in Section~\ref{sec:decomp-branch-union} is used for scenario problem solving. We begin by iterating between step 1 and step 2 of the Modified-Algorithm 2 for a certain number of iterations before proceeding to step 3. We observe that for the concerned problem, consecutive iterations between the first two steps refine the approximation. In our implementation of the linearize-branch-and-union algorithm, we use a custom data structure combining a binary search tree, recursion, and a heap queue to accelerate tree construction. We set 1\% optimality {gap} and a 12-hour time limit as the termination criteria. Lower bound, upper bound, and the corresponding optimality gap at termination are reported in Table~\ref{tab:comparison-TSS}. Additionally, total number of feasibility and optimality cuts added throughout master iterations are reported. The results demonstrate the superiority of our proposed decomposition approach. Even for the two-scenario problem, our approach achieves a better optimality gap upon termination. Also, our method is scaling better with a large scenario size.

\begin{table}[h]
\centering
\caption{Comparison of results for two-stage stochastic setting with 12-hour time limit.}
\setlength{\tabcolsep}{4pt}
\resizebox{\textwidth}{!}{
\begin{tabular}{@{}r r r r r r r@{}} 
\toprule
\multirow{2}{*}{\makecell[r]{Scenario\\Size}} 
& \multicolumn{4}{c}{Our Method} 
& \multicolumn{2}{c}{Gurobi} \\
\cmidrule(lr){2-5} \cmidrule(lr){6-7}
& \makecell[r]{Optimality\\Gap \%\\(Lower Bound)}
& \makecell[r]{Best Obj\\Value}
& \makecell[r]{No. of\\Feasibility\\Cuts}
& \makecell[r]{No. of\\Optimality\\Cuts}
& \makecell[r]{Optimality\\Gap \%\\(Lower Bound)}
& \makecell[r]{Best Obj\\Value} \\
\midrule
2   & 3.86\% (754.45) & 784.82 & 1753 & 165 & 18.6\% (642.41) & 790.44 \\
5   & 4.08\% (547.01) & 570.30  & 627  & 176 & 17.5\% (479.25) & 581.66 \\
10  & 2.37\% (535.36) & 548.32  & 155  & 114 & 19.8\% (440.66) & 550.11 \\
50  & 4.37\% (419.12) & 438.28  & 622  & 37  & 17.8\% (360.27) & 438.28 \\
100 & 3.98\% (413.75) & 430.94  & 284  & 58  & 27.7\% (345.96) & 478.36 \\
\bottomrule
\end{tabular}}
\label{tab:comparison-TSS}
\end{table}

Next, we consider the distributionally robust counterpart of the process-integrated chemical supply chain model. A Wasserstein-distance based ambiguity set $\mathcal{P}^{W}$ with finite support as defined in~\eqref{eq:wasserstein} is used in the distributional robustness formulation. In~\eqref{eq:wasserstein}, $p^0$ is a reference probability distribution and $q$ is a joint distribution over the finite support, $\rho$ is a given radius, and $\xi_i$ represents a realization of the random vector $\xi$ for scenario $i \in \Omega$. 
\begin{align}\label{eq:wasserstein}
&\mathcal{P}^{W} = \left\{ 
         p \in \mathbb{R}^{|\Omega|}_+
        \ \middle\vert 
        \begin{array}{l}
            \exists \, q \in \mathbb{R}^{|\Omega| \times |\Omega|}_+ \, :\\
            \sum_{i = 1}^{|\Omega|} \sum_{j = 1}^{|\Omega|} \|\xi_i - \xi_j\| q_{i j} \leq \rho\\
            \sum_{i = 1}^{|\Omega|} q_{i j} = p^0_j \qquad \;\; \forall j \in \Omega\\
            \sum_{j = 1}^{|\Omega|} q_{i j} = p_i \quad \forall i \in \Omega\\
            \sum_{i = 1}^\Omega p_i = 1\\
             p_i \geq 0 \quad \qquad \qquad \quad \, \forall i \in {\Omega}
        \end{array}
        \tag{Wass}
    \right\},
\end{align}

Since a commercial solver cannot solve the distributional robust model directly, we report the results obtained using our method in Table~\ref{tab:DRO-result} for two different radii: $\rho =0.05$ and $\rho =0.1$. As shown in Table~\ref{tab:DRO-result}, for these radii, the time taken by the distributionally robust model is approximately similar to that for the two-stage stochastic setting in closing a certain optimality gap. As expected, the optimal objective value obtained increases with the radius $\rho$, and is minimum for $\rho =0$, i.e., for the two-stage stochastic setting. 

\begin{table}[h]
\centering
\caption{Results for distributional robust setting using our method with two different Wasserstein distance tolerances $\rho$ with a 12-hour time limit.}
\setlength{\tabcolsep}{4pt}
\resizebox{\textwidth}{!}{
\begin{tabular}{@{}r r r r r r r r r@{}} 
\toprule
\multirow{2}{*}{\makecell[r]{Scenario\\Size}} 
& \multicolumn{4}{c}{$\rho = 0.05$} 
& \multicolumn{4}{c}{$\rho = 0.1$} \\
\cmidrule(lr){2-5} \cmidrule(lr){6-9}
& \makecell[r]{Optimality\\Gap \%}
& \makecell[r]{Best Obj\\Value}
& \makecell[r]{No. of\\Feasibility\\Cuts}
& \makecell[r]{No. of\\Optimality\\Cuts}
& \makecell[r]{Optimality\\Gap \%}
& \makecell[r]{Best Obj\\Value}
& \makecell[r]{No. of\\Feasibility\\Cuts}
& \makecell[r]{No. of\\Optimality\\Cuts} \\
\midrule
2   & 8.10\% & 785.77 & 3275 &  74 & 4.49\% & 786.73 & 1410 & 141 \\
5   & 5.72\% & 571.30 &  691 & 162 & 4.07\% & 572.29 &  428 & 145 \\
10  & 1.71\% & 549.19 &  158 & 115 & 1.72\% & 549.98 &  121 & 213 \\
50  & 4.34\% & 439.93 &  470 &  41 & 3.63\% & 441.58 &  445 &  48 \\
100 & 6.81\% & 434.20 &  208 &  53 & 6.63\% & 436.89 &  199 &  53 \\
\bottomrule
\end{tabular}}
\label{tab:DRO-result}
\end{table}

\section{Concluding Remarks}
We have developed a pure cutting-plane algorithm for solving a general mixed-integer convex program to 
exact optimality in a finite number of iterations. This algorithm has been utilized as an oracle for developing
similar algorithms in solving more advanced mixed-integer convex problems with a two-stage decision process.
The performance of our algorithm is evaluated on a practical problem under both the two-stage stochastic and distributionally robust settings. The results demonstrate that, with efficient implementation, our approach is viable for real-world problems where commercial solvers struggle, particularly for large-scale instances with tens of scenarios.

Although in the numerical example, we obtain an exact solution without invoking the polishing step, it may not be the case for all problems. Convex functions for which the polishing step is not needed {\em to prove finite convergence} can be a topic of future research.

\bibliographystyle{spmpsci_unsrt}
\bibliography{reference-database-08-06-2019}

\appendix
\normalsize
\section{Omitted Proofs}\label{sec:proofs}
\subsubsection*{Proof of Proposition~\ref{prop:lin-approx-convex-prog}}
\qrevision{
First, the property of $\alpha$ ensures that \eqref{opt:conv-linear-opt}
is a relaxation of \eqref{opt:convex-opt}.
Let $S$ be the feasible set of \eqref{opt:convex-opt}. 
Since $x^*$ is an optimal solution of \eqref{opt:convex-opt},
the optimality condition indicates that $-c\in{\cal N}_S(x^*)$.
Given that $\ri(S)$ is non-empty, the normal cone has the Minkowski decomposition 
$\cN_S(x^*)=\cN_{\bar{A}}(x^*)+\cN_{\cal B}(x^*)$, where $\cN_{\bar{A}}(x^*)$
is the normal cone of $\{x:\bar{A}x=b\}$ at $x^*$. Since by Assumption~\ref{ass:no-error},
the normal cones $\cN_{\bar{A}}(x^*)$ and $\cN_{\cal B}(x^*)$ can be generated, it is possible to compute a $r\in\cN_{\bar{A}}(x^*)$ and 
a $\alpha\in\cN_{\cal B}(x^*)$ such that $-c=r+\alpha$. It is clear that $\alpha$
is a vector in $\cN_{\{x:\;\alpha^\top(x-x^*)\le0\}}(x^*)$, and hence
$-c\in\cN_{\bar{A}}(x^*)+ {\cN_{\{x:\;\alpha^\top(x-x^*)\le0\}}(x^*)}$.
This is the optimality condition of \eqref{opt:conv-linear-opt}, 
which indicates that $x^*$ is an optimal solution of \eqref{opt:conv-linear-opt}.
Therefore, $x^*$ is an optimal solution of \eqref{opt:conv-linear-opt} which 
gives the optimal objective $c^\top{x^*}$.  \qed
}

\subsubsection*{Proof of Proposition~\ref{prop:int-seq-convg}}
{We use the fact that the limit of a convergent sequence in a closed set necessarily belongs to that set.
By realizing that the set $\mathbb{Z}^{l_1}\times\mathbb{R}^{l_2}$ is closed, we can claim that 
the limit point $x^*\in \mathbb{Z}^{l_1}\times\mathbb{R}^{l_2}$. Since the $\{x^{(n)}\}$ converges to $x^*$, 
the first $l_1$ integer components of $x^{(n)}$ are identical to that of $x^*$ for sufficiently large $n$
due to the discrete natural of $\mathbb{Z}^{l_1}$.  \qed
}

\subsubsection*{Proof of Proposition~\ref{prop:mod-alg-2}}
\qrevision{Modified-Algorithm 2 replaces the \textbf{oracle-PMILP} with a branch-and-bound step 
to solve the parametric MILP \eqref{opt:master-x-y}. }
Outer-approximation cuts generated in Step 2 and Step 3 only depend on the optimal solution $y^{(n)}$
of \eqref{opt:master-x-y}, which is independent of how the master problem is solved. 
Therefore, the proof of finite convergence is exactly the same as the proof
of Theorem~\ref{thm:convg-gen-param-cut-plane}.  \qed

\subsubsection*{Proof of Proposition~\ref{prop:S-partition}}
Based on the proof of Observation~\ref{lem:node-constraints},
the inequalities {$R_\nu x+Q_\nu y+s_\nu\le 0$}
can be partitioned into two groups of inequalities. Group 1: the valid inequalities
of the form $a^{(k)\top}x+b^{(k)\top}y + c^{(k)} \le 0$
and $D^{(k)}_xx +D^{(k)}_yy \le d^{(k)}$
in the {$(x,y)$}-space added respectively at Line~\ref{lin:add_ax+by+c<0} and Line~\ref{lin:Dx+Dy}
of Modified-Algorithm~\ref{alg:param-micp};
Group 2: the simple-bound inequalities of the form 
 $y_i\le b$ and $y_i\ge b$ in the {$y$}-space used to define each node.
{Define the following sets induced by Group 1 and Group 2, respectively:
\begin{displaymath}
\begin{aligned} 
&P^1=\{(x,y):\; a^{(k)\top}x+b^{(k)\top}y + c^{(k)} \le 0, \;D^{(k)}_xx +D^{(k)}_yy \le d^{(k)},\;\forall k\in {\cal I}\}, \\
&P^2_\nu=\{(x,y):\;x\textrm{ free},\; y\in\textrm{Bound}_\nu\},
\end{aligned}
\end{displaymath}
where $\cal I$ is the set of iteration indices of Modified-Algorithm~\ref{alg:param-micp} at which the Group 1 inequalities have been added,
and $\textrm{Bound}_\nu$ is the region defined by the simple-bound inequalities in the $y$-space for the node $\nu$.}
Theorem~\ref{thm:convg-gen-param-cut-plane} indicates that inequalities in
Group 1 are valid for {$\vartheta$}.
Therefore, we have  
\bdm
\vartheta\subseteq P_1,\quad{\{(x,y):\;x\textrm{ free, } y\in\mtc{Y}\cap(\mbZ^{l_2}\times\mbR^{l_3})\}}\subseteq\cup_{\nu\in\mtc{L}(\hat{x})} P^2_\nu.
\edm 
It then implies that
\bdm
\vartheta\subseteq \cup_{\nu\in\mtc{L}(\hat{x})}(P^1\cap P^2_\nu)
=\cup_{\nu\in\mtc{L}(\hat{x})} P_\nu(\hat{x}).
\edm
This concludes the proof.  \qed

\subsubsection*{Proof of Proposition~\ref{obs:[x,eta]inV}}
Following \eqref{opt:node-LP-2}, the LP relaxation of the leaf node $v^*\in\mtc{L}^\omega(\hat{x})$ can be written as 
\beq
\ba
&\min_y\; h^{\omega\top}y \\
&\;\textrm{s.t. }\;W^\omega y = - T^\omega \hat{x} + q^\omega,  \\
&\qquad\; Q^\omega_{\nu^*}y \le  -R^\omega_{\nu^*}\hat{x} - s^\omega_{\nu^*}, \\
&\qquad\; y\in\mbR^{l_2+l_3}.
\ea 
\eeq
We claim that the optimal solution $y^{\omega*}$ of $\mathrm{Sub}(x^*,\omega)$ is a feasible solution to the following linear program:
\beq\label{opt:Sub(x*,w)_node}
\ba
&\min_{y}\; h^{\omega\top}y \\
&\;\textrm{s.t. }\;W^\omega y = - T^\omega x^* + q^\omega,  \\
&\qquad\; Q^\omega_{\nu^*}y \le  -R^\omega_{\nu^*}x^* - s^\omega_{\nu^*}, \\
&\qquad\; y\in\mbR^{l_2+l_3}.
\ea 
\eeq
Given that this claim holds, we can apply the dual values obtained from \eqref{opt:node-LP-2}
to dualize \eqref{opt:Sub(x*,w)_node} which leads to the following inequalities:
\bdm
\eta^{\omega*}=h^{\omega\top}y^{\omega*}\ge\textrm{obj}\eqref{opt:Sub(x*,w)_node} 
\ge\left(\mu^{\omega\top}_{\nu^*} R^\omega_{\nu^*} - \theta_{\nu^*}^{\omega \top} T^\omega \right) x^* 
+ \left(\mu^{\omega\top}_{\nu^*} s^\omega_{\nu^*}  + \theta_{\nu^*}^{\omega \top} q^\omega\right).
\edm
This shows $(x^*,\eta^{\omega*})\in V^\omega_{\nu^*}(\hat{x})$.

Now it remains to prove the claim. First, we note that the constraints 
$Q^\omega_{\nu^*}y \le  -R^\omega_{\nu^*}x^* - s^\omega_{\nu^*}$
can be divided into two groups: 
\bdm
\ba
&M^\omega_2 y \le  -M^\omega_1x^* - M^\omega_3, \\
&N_{\nu^*}y\le{g}_{\nu^*}. 
\ea
\edm
The matrices $M^\omega_1$, $M^\omega_2$
and $M^\omega_3$ in the first group of constraints are coefficients of convexification
cutting planes satisfying that $M^\omega_1x+M^\omega_2 y\le{M^\omega_3}$ are valid
for the set $\vartheta^\omega=\{(x,y^\omega):\;g_i(x,y^\omega)\le0\;\forall i\in\mtc{I},\;y^\omega\in\mtc{Y}\}$,
and these matrices are independent of nodes in the branch-and-bound tree.
The second group of constraints represents the simple bounds that specify the node $\nu^*$.
Since $(x^*,y^{\omega*})\in\vartheta^\omega$, it implies that 
$M^\omega_2 y^{\omega*}\le -M^\omega_1x^*-M^\omega_3$, i.e., $y^{\omega*}$
satisfies the first group of constraints. Furthermore, by the definition of node $\nu^*$, 
one has $(x^*,y^{\omega*})\in P^\omega_{\nu^*}(\hat{x})$, indicating that $y^{\omega*}$
should satisfies that simple-bound constraints $N_{\nu^*}y\le{g}_{\nu^*}$. This shows 
that $y^{\omega*}$ is a feasible solution to \eqref{opt:Sub(x*,w)_node}.  \qed

\section{A Numerical Example of the Cutting-Plane Algorithm}
\label{sec:num}  

Let us consider the following TSS-MICP:
\beq\label{eq:toy-example-TSS}
\ba
&\min_{x_1,x_2\in\{0,1\}}\; 0.9x_1+1.11x_2 + \E_\xi[\mQ(x,\xi^\omega)] \\
&\quad\textrm{ s.t.}\;\; 3x_1+2x_2\ge 2, 
\ea
\eeq
where there are only two scenarios $\Omega=\{\omega_1,\omega_2\}$ with
equal probability. The recourse function at the two scenarios is given as
\beq\label{opt:scenario1}
\ba
\mQ(x,\xi^{\omega_1})=&\min_{y_{11},y_{12}\in\mathbb{Z}_+}\;0.5y_{11}+y_{12} \\
&\textrm{ s.t. } -2y_{11}-[\log(1+e)] y_{12} + \log(1+e^{y_{11}+y_{12}}) \le x_1+2x_2-1. 
\ea
\eeq  
\beq\label{opt:scenario2}
\ba
\mQ(x,\xi^{\omega_2})=&\min_{ y_{21},y_{22}\in\mathbb{Z}_+}\;y_{21}+y_{22} \\
&\textrm{ s.t. } -[\log(1+e)](y_{21}+y_{22})+\log(1+e^{y_{21}+y_{22}})\le x_1+x_2-1.
\ea
\eeq
The master problem at Iteration 1 is simply 
\bdm
\min\;0.9x_1+1.11x_2+\eta\quad \textrm{s.t.}\;3x_1+2x_2\ge 2,\;\eta\ge0,\;x_1,x_2\in\{0,1\}.
\edm
The relaxation problem yields the fractional solution $(x_1,x_2)=(2/3,0)$.
We further add {an optimality cut $x_1\ge1$ (since  $x_1 = 1$ is feasible and objective coefficient of $x_1$ is smaller)} and {re-solve} the master problem, which 
gives the optimal solution $(x_1,x_2)=(1,0)$. Substituting the first-stage
solution $(x_1,x_2)=(1,0)$ into \eqref{opt:scenario1} and solving the continuous
relaxation sub-problem yields a fractional solution $(y_{11},y_{12})=(0.5,0.48)$ which 
is on the curve $-2y_{11}-[\log(1+e)] y_{12} + \log(1+e^{y_{11}+y_{12}})=0$.
We add the tangent cut $1.272y_{11}+0.586y_{12}\ge 0.918$ 
induced by this fraction solution to the relaxation sub-problem.
Then we solve the following relaxed MILP:
\bdm
\min\; 0.5y_{11}+y_{12} \quad\textrm{s.t. }
-1.272y_{11}-0.586y_{12}+ 0.918\le 0,\; y_{11},y_{12}\in\mathbb{Z}_+,
\edm
by adding a cut $y_{11}+y_{12}\ge 1$ to the linear relaxation of 
the above MILP {for feasibility (since 0 $\ngeq 0.918$}, which leads to the integral solution $(y_{11},y_{12})=(1,0)$
of \eqref{opt:scenario1} for the given first-stage value $(x_1,x_2)=(1,0)$.
Similarly, substituting $(x_1,x_2)=(1,0)$ into \eqref{opt:scenario2}, 
and solving the continuous relaxation gives the solution $(y_{21},y_{22})=(0.5,0.5)$.
We add a tangent cut $y_{21}+y_{22}\ge 1$ induced by this {fractional} solution
to the relaxation sub-problem. Then we solve the following relaxed MILP:
\bdm
\min\; y_{11}+y_{12} \quad\textrm{s.t. }
-y_{11}-y_{12}+1\le 0,\; y_{11},y_{12}\in\mathbb{Z}_+,
\edm
which leads to an integral optimal solution $(y_{21},y_{22})=(1,0)$.

{Since $x_1 \geq 1$, always $x_1 + 2 x_2 - 1\geq 0$ for any value of $x_2$.} Hence, that the eventual linear program that leads to the integral solution
$(y_{11},y_{12})=(1,0)$ of \eqref{opt:scenario1} is:
\bdm
\min\; 0.5y_{11}+y_{12} \quad\textrm{s.t. }
-y_{11}-y_{12}+ 1\le x_1+2x_2-1,\; y_{11},y_{12}\ge 0.
\edm
The Benders' cut associated with this scenario problem is $\eta^{\omega_1}\ge 1-0.5x_1-x_2$. Similarly, the Benders' cut associated with the scenario $\omega_2$ problem is
$\eta^{\omega_2}\ge 2-x_1-x_2$. The aggregated Benders' cut is $\eta\ge 1.5-0.75x_1-x_2$. Adding this aggregated Benders' cut to the master problem gives:
\bdm
\min\;0.9x_1+1.11x_2+\eta\quad \textrm{s.t.}\;3x_1+2x_2\ge 2,
\;\eta\ge 1.5-0.75x_1-x_2,\;x_1,x_2\in\{0,1\},\; \eta\ge0.
\edm
Solving the above updated master problem gives the first-stage
solution $(x_1,x_2)=(0, 1)$ with objective 1.61 since objective value with previous iteration solution $(x_1,x_2)=(1, 0)$ is 1.65. Since only two solutions are possible, the optimal solution is $(x^*_1,x^*_2)=(0,1)$.

\qrevision{Now let us consider the distributional robust version of ~\eqref{eq:toy-example-TSS} with Wasserstein ambiguity set as defined in \eqref{eq:wasserstein}. From \eqref{opt:scenario1} and \eqref{opt:scenario2}, in our setting, random parameters $\xi_1 = (0.5, -2, 2)$ and $ \xi_2 = (1, -\log(1 + e), 2)$. We further set $\rho = 0.13$. In contrast to the two-stage setting, after obtaining the scenario-specific Benders' cut
$\eta^{\omega_1}\ge 1-0.5x_1-x_2$ and $\eta^{\omega_2}\ge 2-x_1-x_2$, we solve the optimization problem to obtain the {worst-case} probability: $\max_{(p^{\omega_1}, p^{\omega_2}) \in \mathcal{P}^W} p^{\omega_1} (1 - 0.5 x_1 - x_2) + p^{\omega_2} (2 - x_1 - x_2) = \max_{(p^{\omega_1}, p^{\omega_2}) \in \mathcal{P}^W} 0.5 p^{\omega_1} + p^{\omega_2}$. {worst-case} probability we obtain is $(p^{\omega_1}, p^{\omega_2}) = (0.4, 0.6)$ (in comparison to $(0.5, 0.5)$ without ambiguity incorporation) which gives us the aggregated optimality cut $\eta \geq 1.6 - 0.8 x_1 - x_2$. The aggregated optimality cut obtained through worst-case scenario probabilities is then added to the master problem: 
\bdm
\min\;0.9x_1+1.11x_2+\eta\quad \textrm{s.t.}\;3x_1+2x_2\ge 2,
\;\eta\ge 1.6-0.8 x_1-x_2,\;x_1,x_2\in\{0,1\},\; \eta\ge0.
\edm
Solving the above updated master problem gives the first-stage
solution $(x_1,x_2)=(1, 0)$ with objective 1.70 while the two-stage-stochastic solution was $(x_1,x_2)=(0, 1)$ with objective $1.61$. Thus, due to consideration of the worst-case scenario, the optimal solution has changed with a larger objective value.}

\section{Solving MICP, JMICP and TSS-MICP to a Desired Tolerance}
\label{sec:micp-epsilon}

\subsection{Analysis for the MICP problem}
\label{sec:error-single-stage}
We have shown that the algorithms given in the previous sections can solve the 
{mixed-integer} convex and conic programs (MICP) to optimality by purely generating cutting
planes in a finite number of iterations. This result was established under the assumptions that oracles for solving the projection problem (\ref{eqn:IterateProjection}), the convex optimization problem (\ref{opt:convex-FixedInteger}), and the supporting inequalities from Proposition~\ref{prop:lin-approx-convex-prog} are known. However, in practice, we may only be able to solve (\ref{eqn:IterateProjection}) and (\ref{opt:convex-FixedInteger}) to a desired precision. Thus, we now develop an algorithm that solves the mixed-integer convex program
\eqref{opt:convex-opt-linear-obj} with a linear objective to $\epsilon$-accuracy only. In fact, this allows us to simplify the steps given in Algorithms~\ref{alg:micp-oracle}. Some careful definitions of optimality, feasibility, and stability in the error-engaged case 
are necessary to depict the algorithm and the convergence results. These concepts are given in 
Definition~\ref{def:e-cont} to \ref{def:epsilon-micp}.
\begin{definition}\label{def:e-cont}
Consider a convex optimization problem:
\beq\label{opt:e-cont}
\min\;f(x)\;\textrm{ s.t. } \srevision{Ax=b,} \, x\in\mtc{C},
\eeq
where \srevision{$Ax=b$ represents linear equality constraints, }$\mtc{C}$ is a {full-dimensional} closed set in $\mbR^{l_1+l_2}$.
For any $\epsilon>0$,
define the set $\mtc{C}^{\epsilon}_+$ as
\bdm
\begin{aligned}
&\mtc{C}^{\epsilon}_+=\Set*{x\in\mbR^{l_1+l_2}}{\exists z\in\mtc{C}\;\textrm{s.t.}\;\dist(x,z)\le\epsilon}. \\
\end{aligned}
\edm 
\newline
A point $x$ is an ($\epsilon$,$\epsilon^\prime$)-accurate
solution of \eqref{opt:e-cont}, if \eqref{opt:e-cont} is feasible,
$x\in\mtc{C}^{\epsilon^\prime}_+$ and $\textrm{OPT}-\epsilon\le f(x)\le\textrm{OPT}+\epsilon$,
where $\textrm{OPT}$ is the optimal value of \eqref{opt:e-cont}.
\end{definition}

\begin{definition}\label{def:epsilon-micp}
Consider a general mixed-integer programming problem 
\beq\label{opt:math-prog}
\min\;f(x)\quad\textrm{s.t.} \, \srevision{Ax=b,}\, \;x\in\mtc{C},\;x\in\mbZ^{l_1}\times\mbR^{l_2},
\eeq 
where $\mtc{C}$ is a {full-dimensional} closed set in $\mbR^{l_1+l_2}$.
Let $\epsilon$, $\epsilon^\prime$, and $\lambda$ be any three positive real numbers. 
Let $\mtc{C}^{\epsilon}_+$ be defined the same
as in Definition~\ref{def:e-cont}, and define $\mtc{C}^{\epsilon}_-$ as 
\bdm
\mtc{C}^{\epsilon}_-=\mtc{C}\setminus\cup_{x\in bd(\mtc{C})}B(x,\epsilon),
\edm
where $bd(\mtc{C})$ is the boundary of $\mtc{C}$ and $B(x,\epsilon)=\Set*{z\in\mbR^{l_1+l_2}}{\norm{x-z}\le\epsilon}$.
\newline
(i) The problem \eqref{opt:math-prog} is called  $\epsilon$-feasible (infeasible) if 
$\mtc{C}^{\epsilon}_-\srevision{\cap\{Ax=b\}}\cap(\mbZ^{l_1}\times\mbR^{l_2})$ is non-empty (empty). \newline
(ii) The problem is $(\epsilon,\lambda)$-stable if it is $\epsilon$-feasible and there exist $x^1$ and $x^2$
such that $x^1$ is an optimal solution of 
\beq\label{opt:C^e_+}
\min\;f(x)\quad\textrm{s.t.},\;\srevision{Ax=b},\;x\in\mtc{C}^{\epsilon}_+,\;x\in\mbZ^{l_1}\times\mbR^{l_2},
\eeq 
$x^2$ is an optimal solution of 
\beq\label{opt:C^e_-}
\min\;f(x)\quad\textrm{s.t.},\;\srevision{Ax=b},\;x\in\mtc{C}^{\epsilon}_-,\;x\in\mbZ^{l_1}\times\mbR^{l_2},
\eeq 
$x^1_Z = x^2_Z$ and $\norm{x^1-x^2}\le\lambda\epsilon$. \newline
(iii) A point $x\in\mbZ^{l_1}\times\mathbb{R}^{l_2}$ is an $(\epsilon,\epsilon^{\prime},\lambda)$-accurate
solution of \eqref{opt:math-prog}, if \eqref{opt:math-prog} is $(\epsilon^\prime,\lambda)$-stable,
$x\in\mtc{C}^{\epsilon^\prime}_+$, and $\textrm{OPT}-\epsilon\le f(x)\le\textrm{OPT}+\epsilon$,
where $\textrm{OPT}$ is the optimal value of \eqref{opt:math-prog}.
\end{definition}
For a geometric interpretation, the set $\mtc{C}^{\epsilon}_+$ is 
an $\epsilon$-covering of the set $\mtc{C}$, 
and the set $\mtc{C}^{\epsilon}_-$ can be thought 
of an epsilon-shrinkage of the set $\mtc{C}$.
We consider the following mixed-integer convex program with a linear objective function
\beq\label{opt:convex-opt-linear-obj}
\min\;c^\top x\quad\textrm{s.t.}\;\srevision{Ax=b},\;x\in\mtc{C},\;x\in\mbZ^{l_1}\times\mbR^{l_2}.
\eeq
We focus on analyzing the case where the convex optimization oracle
can only identify a solution within a tolerant accuracy of optimality and feasibility.  
We make the following assumptions about the oracle on the solution accuracy. 
\begin{assumption}\label{ass:convex-opt-oracle-errors}
For any $\epsilon,\epsilon^\prime>0$,
there exists an oracle that can find a $(\epsilon,\epsilon^\prime)$-accurate solution of a general
convex optimization problem \eqref{def:e-cont} given that \eqref{def:e-cont} is feasible.
\end{assumption} 
\srevision{Note that, for simplicity, under Assumption~\ref{ass:convex-opt-oracle-errors} we do not consider the numerical error in satisfying the linear equality constraint within the feasibility tolerance. Furthermore, throughout the analysis, it is assumed that a linear program can be solved exactly.} By utilizing the oracle provided in Assumption~\ref{ass:convex-opt-oracle-errors}, 
Algorithm~\ref{alg:micp-e} is developed for solving \eqref{opt:convex-opt-linear-obj} 
to $\epsilon$-accuracy. The algorithm is a modification of Algorithm~\ref{alg:micp-oracle}
by removing the polishing step (Step 3), as it is not required to find an exact optimal solution
in this case. The algorithm is essentially an outer approximation algorithm,
where at each iteration, a mixed-integer program obtained from the previous outer approximations is solved. 
The outer approximation is tightened by solving a projection problem approximately. 
The projection problems are defined at the solutions of the mixed-integer linear programs.

\qrevision{Consider a point $v$ outside $\mtc{C}$. If we project $v$ on $\mtc{C}$ by solving the projection problem with an oracle involving errors, we can only obtain an approximated projection point $x^\epsilon$.
The cutting plane that passes $x^\epsilon$ can be written as 
\bdm
\{x:\;(v-x^\epsilon)^\top(x-x^\epsilon)=0\}.
\edm
Note that this cutting plane is not tangent to $\mtc{C}$ in general due to errors.
See Figure~\ref{fig:approx-proj} for example.
We are concerned about the case where the cutting plane intersects with $\mtc{C}$,
and want to estimate the maximum depth by which it may cut $\mtc{C}$.
To be more precise, the maximum depth it may cut $\mtc{C}$ is given
by the following quantity:
\beq\label{eqn:rmax}
r_{\max}=\underset{z\in\mtc{C}\cap V^-}{\max}\;\dist(z,V^+),
\eeq
where $V^{+(-)}=\Set*{x\in\mbR^{l_1+l_2}}{(v-\xe)^\top (x-\xe)\le(\ge)0}$
with the convention that $r_{\max}=0$ if $\mtc{C}\cap V^-=\emptyset$.
The following lemma provides an upper bound on $r_{\max}$. 
}

\begin{algorithm}
	\caption{An Algorithm for solving a {mixed-integer} convex program to a given accuracy.}
	\label{alg:micp-e}
	\begin{algorithmic}[1]
	\State{\textbf{Input}: accuracy control parameters $\epsilon_1,\epsilon_2>0$.}
	\State{Set $n\gets 1$, $flag\gets 0$ and $\mI^{(0)}\gets\emptyset$.}
	\While{$flag=0$} 
		\State{(Start iteration $n$.)}
		\State{\textbf{Step 1:} Solve the following master problem to optimality and let $x^{(n)}$ be the solution.}
		\State{If it is feasible}
		\beq\label{opt:micp-master-e}
		\ba
		&\min\; c^\top x \\
		&\text{ s.t. } \srevision{Ax=b},\\
         &\qquad (x^{(k)}-z^{(k)})^{\top}x\le (x^{(k)}-z^{(k)})^{\top}z^{(k)} \qquad\forall k\in [n-1], \\
		&\qquad x\in\mathbb{Z}^{l_1}\times\mathbb{R}^{l_2},
		\ea
		\eeq
		\State{where $(x^{(k)}-z^{(k)})^{\top}x\le (x^{(k)}-z^{(k)})^{\top}z^{(k)}$ is the projection cut
		generated at $k$-th iteration .}
		\State{If the master problem \eqref{opt:micp-master-e} is not feasible, then stop and return infeasible status.}
		\If{$x^{(n)}\in\cal C$} 
			\State{Stop and return $x^{(n)}$ as an optimal solution of \eqref{opt:convex-opt-linear-obj};}
		\Else
			\State{Continue.}
		\EndIf
		\State{\textbf{Step 2:} Solve (\ref{eqn:IterateProjection}) to $(\epsilon_1,\epsilon_1)$-accuracy. Let $z^{(n)}$ and $d^{(n)}$
		be {respectively} a $(\epsilon_1,\epsilon_1)$-accurate solution (Definition~\ref{def:e-cont}) and objective of (\ref{eqn:IterateProjection}).} \label{lin:z^n_d^n}
		\If{$d^{(n)}\le\epsilon_2/2$}, \label{lin:termination_condition}
			\State{Stop and return $x^{(n)}$.}
		\Else
			\State{Add the separation cut $(x^{(n)}-z^{(n)})^{\top}x\le (x^{(n)}-z^{(n)})^{\top}z^{(n)}$ 
					to \eqref{opt:micp-master-e}.}
		\EndIf
		\State{$n\gets n+1$.}
	\EndWhile
	\end{algorithmic}
\end{algorithm}

\begin{lemma}\label{lem:e-accuracy-proj}
Let $\mtc{C}$ be a {full-dimensional} closed convex set in $\mbR^m$ and $v$
be a point in $\mbR^m\setminus\mtc{C}$.
Consider the projection problem
\beq \label{eqn:BasicProjectionProblem}
\min\;\norm{x-v}\quad\emph{s.t. } x\in\mtc{C}.
\eeq
Let $x^*$ and $d$ be the optimal solution and optimal value of \eqref{eqn:BasicProjectionProblem}, respectively. 
Let $\xe$ be a ($\epsilon,\epsilon$)-accurate solution of \eqref{eqn:BasicProjectionProblem} (Definition~\ref{def:e-cont}), 
with $\epsilon<\frac{d}{64}$. 
The following properties hold:\newline
(a) $\norm{\xe-x^*}\le2\sqrt{d\epsilon}+\epsilon$. \newline
(b) Let $u$ be a point satisfying $\norm{u-v}\le d-3\epsilon$,
then $(v-\xe)^\top (u-\xe)>0$, i.e., $u$ violates the approximated supporting inequality
$(v-\xe)^\top (x-\xe)\le0$ of $\mtc{C}$. \newline
(c) Suppose $\mtc{C}$ is bounded and $D:=\max_{y_1,y_2\in\mtc{C}}\norm{y_1-y_2}$ 
is the diameter of $\mtc{C}$. 
Then the maximum cutting depth $r_{\max}$ can be bounded as
\bdm
r_{\max}\le \frac{9(D+2d)\sqrt{d\epsilon}}{d}.
\edm
\end{lemma}
\begin{remark}
Note that in general the approximated projection point $x^\epsilon$ can be located inside $\mtc{C}$,
and in this case the $r_{\max}$ represents the maximum violated depth of the cut $(v-\xe)^\top (x-\xe)\le0$.
See Figure~\ref{fig:approx-proj} for an illustration.
\end{remark}
\begin{proof}
(a) Note that $x^*=\textrm{Proj}_{\mtc{C}}\;v$ and $d=\dist(v,\mtc{C})\qrevision{=\norm{v-x^*}}$. 
Since $\xe$ is an $(\epsilon,\epsilon)$-accurate solution, by definition we have
$\norm{\xe-v}\le d+\epsilon$ and $\dist(\xe,\mtc{C})\le\epsilon$.
It follows that
\bdm
\begin{aligned}
&(d+\epsilon)^2\ge\norm{\xe-v}^2=\norm{\xe-x^*+x^*-v}^2 \\
&=\norm{\xe-x^*}^2+\norm{x^*-v}^2+2\inp{\xe-x^*}{x^*-v} \\
&=\norm{\xe-x^*}^2+d^2+2\inp{\xe-z}{x^*-v}+2\inp{z-x^*}{x^*-v},
\end{aligned}
\edm 
where $z=\textrm{Proj}_{\mtc{C}}\;\xe$. 
Note that the \qrevision{the last term in the above inequality satisfies} 
$\inp{z-x^*}{x^*-v}\ge 0$ due to convexity.
Therefore, we have
\bdm
\begin{aligned}
(d+\epsilon)^2&\qrevision{\ge\norm{\xe-x^*}^2+d^2+2\inp{\xe-z}{x^*-v} }\\
&\qrevision{\ge\norm{\xe-x^*}^2+d^2-2\norm{\xe-z}\cdot\norm{x^*-v} } \\
&\qrevision{\ge\norm{\xe-x^*}^2+d^2-2\cdot\dist(\xe,\mtc{C})\cdot\dist(v,\mtc{C}) } \\
&=\norm{\xe-x^*}^2+d^2-2d\epsilon,
\end{aligned}
\edm 
which implies that $\norm{\xe-x^*}\le2\sqrt{d\epsilon}+\epsilon$.
\qrevision{Note that this is a bound on the distance between the approximated projection point $\xe$
and the exact projection point $x^*$ of $v$.}

(b) Note that we have
\bdm
\begin{aligned}
&(v-\xe)^\top (u-\xe)=\inp{v-\xe}{u-\xe} \\
&=\inp{v-\xe}{v-\xe} + \inp{v-\xe}{u-v}  \\
&\ge \norm{v-\xe}^2-\norm{v-\xe}\norm{u-v} \\
&\ge (d-\epsilon)^2-(d+\epsilon)\norm{u-v} \\
&\ge (d-\epsilon)^2-(d+\epsilon)(d-3\epsilon)=4\epsilon^2>0, \\
\end{aligned}
\edm
where the above derivation uses $d - \epsilon \leq \dist(v,x^\epsilon) \leq d + \epsilon$ following definition of $(\epsilon, \epsilon)-$accurate solution.

(c) Since $(v-x^*)^\top (x-x^*)\le0$ is a supporting valid inequality of \qrevision{$x\in\mtc{C}$},
it follows that
\bdm
\mtc{C}\cap V^-\subseteq\Set*{x\in\mbR^{l_1+l_2}}{ (v-x^*)^\top (x-x^*)\le0,\;(v-\xe)^\top (x-\xe)\ge0},
\edm
\qrevision{by the definitions of $\mtc{C}$ and $V^-$.}
Note that if the plane $(v-\xe)^\top (x-\xe)=0$ does not intersect with $\mtc{C}$,
then $r_{\max}=0$ by definition. Let us consider the case in which this plane intersects with $\mtc{C}$.
Suppose $r_{\max}=\dist(y,z)$ where $y\in\text{bd}(\mtc{C})$ and $z=\textrm{Proj}_{(V^+)}y$ (illustrated in Figure~\ref{fig:approx-proj}).   
Note that in this case we have $x^*, x^\epsilon, y, z\in\mtc{C}$, 
which implies that $\norm{x^\epsilon - z}\le D$ by the definition of $D$.
We also note that $y$ is on the line that passes through $z$ and is perpendicular to the plane
$(v-\xe)^\top (x-\xe)=0$, i.e., $y\in\{x: x-z=t(v-x^\epsilon),\;t\in\mbR\}$. 
This line intersects with the exact projection plane $(v-x^*)^\top (x-x^*)=0$ at the point
\bdm
x^{\prime} = z-\frac{\inp{v-x^*}{z-x^*}}{\inp{v-x^*}{v-\xe}}(v-\xe).
\edm 
Since $y$ can be written as a convex combination of points $z$ and $x^{\prime}$,
we have
\bdm
\begin{aligned}
&r_{\max}=\norm{y-z}\le\norm{x^\prime - z}=\norm{\frac{\inp{v-x^*}{z-x^*}}{\inp{v-x^*}{v-\xe}}(v-\xe)} \\
&\le \frac{|\inp{v-x^*}{z-x^*}|}{\norm{v-x^*}(\norm{v-x^*}-\norm{x^*-x^\epsilon})}(\norm{v-x^*}+\norm{x^*-\xe})
\end{aligned}
\edm
We can evaluate the inner product in the numerator as follows
\bdm
\begin{aligned}
&\inp{v-x^*}{z-x^*}=\inp{v-\xe}{z-x^*}+\inp{\xe-x^*}{z-x^*} \\
&=\inp{v-\xe}{z-\xe} + \inp{v-\xe}{\xe-x^*}+\inp{\xe-x^*}{z-\xe}+\inp{\xe-x^*}{\xe-x^*} \\
&=\inp{v-\xe}{\xe-x^*}+\inp{\xe-x^*}{z-\xe}+\inp{\xe-x^*}{\xe-x^*},
\end{aligned}
\edm
where we have used the property $\inp{v-\xe}{z-\xe}=0$ to get the last equality.
Therefore, the following bound holds
\bdm
\begin{aligned}
&|\inp{v-x^*}{z-x^*}|\qrevision{\le\norm{\xe-x^*}\cdot(\norm{v-\xe}+\norm{z-\xe})+\norm{\xe-x^*}^2} \\
&\qrevision{\le\norm{\xe-x^*}\cdot(\norm{v-x^*}+\norm{x^*-\xe}+\norm{z-\xe})+\norm{\xe-x^*}^2 } \\
&\qrevision{\le(\norm{v-x^*}+\norm{\xe-x^*} + D)\norm{\xe-x^*} 
+\norm{\xe-x^*}^2   \quad\textrm{(using $\norm{x^*-\xe}\le{D}$)} } \\
&=(\norm{v-x^*}+2\norm{\xe-x^*}+D)\norm{\xe-x^*}.
\end{aligned}
\edm
We Claim that $\epsilon<\frac{d}{64}$ implies $\epsilon\le\sqrt{d\epsilon}$
and $2\sqrt{d\epsilon}+\epsilon\le\frac{d}{2}$. To prove the claim, we first verify the first inequality:
$\sqrt{d\epsilon}>\sqrt{64\epsilon^2}=8\epsilon>\epsilon$, and verify the second inequality: 
$\frac{d}{2}-\epsilon>\frac{d}{2}-\frac{d}{64}=\frac{31d}{64}\ge\frac{d}{3}
=\sqrt{d}\cdot\frac{\sqrt{d}}{3}>\sqrt{d}\cdot\frac{8\sqrt{\epsilon}}{3}>2\sqrt{d\epsilon}$.

Applying the results from Part $(a)$, we have the following bound for $r_{\max}$:
\bdm
\ba
r_{\max}&\le \frac{(d+D+4\sqrt{d\epsilon}+2\epsilon)(2\sqrt{d\epsilon}+\epsilon)}{d(d-2\sqrt{d\epsilon}-\epsilon)}
(d+2\sqrt{d\epsilon}+\epsilon) \\
&\le\frac{(d+D+d)\cdot 3\sqrt{d\epsilon}}{\frac{d^2}{2}}\cdot\frac{3d}{2} \\
&=\frac{9(D+2d)\sqrt{d\epsilon}}{d},
\ea
\edm
where the two inequalities in the claim are used to obtain the second inequality in the above derivation.  \qed
\end{proof}

\begin{figure}
\center
\includegraphics[scale=0.6,trim=4cm 13.5cm 5cm 4cm]{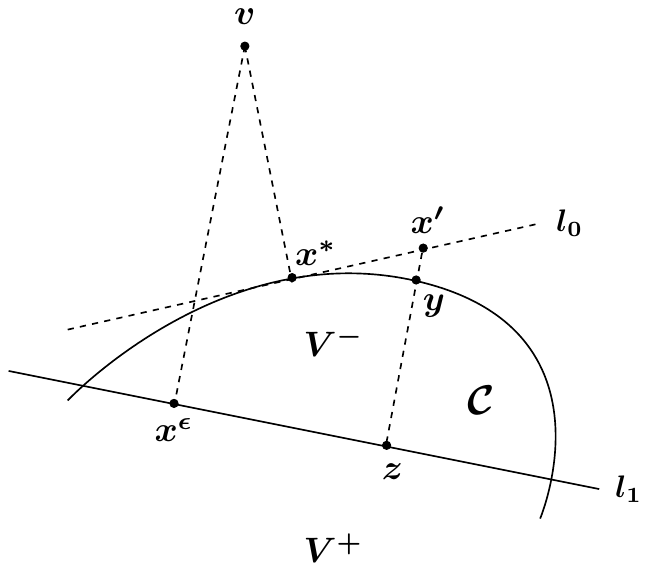}
\caption{\scriptsize Geometry associated with Lemma~\ref{lem:e-accuracy-proj}. The equations
of hyperplanes $l_0$ and $l_1$ are $(v-x^*)^\top(x-x^*)=0$ and $(v-x^\epsilon)^\top(x-x^\epsilon)=0$,
respectively. Note that the approximated projection point $x^\epsilon$ is located inside $\mtc{C}$
in this example.} \label{fig:approx-proj}
\end{figure}

\qrevision{The following theorem characterizes the performance of 
the cutting-plane Algorithm~\ref{alg:micp-e}
on solving the MICP \eqref{opt:e-cont} with a convex optimization oracle 
producing errors.}
\begin{theorem}\label{thm:e-lambda-micp}
Suppose $\mtc{C}$ is a full-dimensional closed and bounded convex set. 
Suppose $\mtc{C}$ is contained in a ball of diameter $D\ge 1$. 
Consider applying Algorithm~\ref{alg:micp-e}
to solve \eqref{opt:convex-opt-linear-obj}.
If we set $\epsilon_1=\epsilon^2_2$ 
(the parameters $\epsilon_1$ and $\epsilon_2$ are used in Algorithm~\ref{alg:micp-e})
and $\epsilon_2<\frac{1}{64}\min\{1,D\}$, then the following properties hold:
\begin{enumerate}
\item[(a)] Algorithm~\ref{alg:micp-e} terminates after a finite number of iterations.
\item[(b)] If the modified Algorithm~\ref{alg:micp-oracle} returns infeasible status, 
then \eqref{opt:convex-opt-linear-obj} is $(20D\sqrt{\epsilon_2})$-infeasible.
\item[(c)] If the modified Algorithm~\ref{alg:micp-oracle} returns a solution $x^{(m)}$, 
then $x^{(m)}$ is a \newline
$(20\lambda D\sqrt{\epsilon_2}\norm{c}, 20 D\sqrt{\epsilon_2},\lambda)$-accurate solution 
of \eqref{opt:convex-opt-linear-obj}
given that \eqref{opt:convex-opt-linear-obj} is \newline
$(20 D\sqrt{\epsilon_2},\lambda)$-stable.
\end{enumerate}
\begin{remark}\label{rmk:sol-error}
\qrevision{After rescaling the error parameter $\epsilon_2$, Part(c) of Theorem~\ref{thm:e-lambda-micp}
{can be restated} as: The returned solution is a $(\rho\lambda\epsilon,\epsilon,\lambda)$-accurate solution of
\eqref{opt:convex-opt-linear-obj} if \eqref{opt:convex-opt-linear-obj} is $(\epsilon,\lambda)$-stable,
where $\rho$ is a constant only depending on $\|c\|$. The choice of $D$ is rather flexible as long as 
it dominates the maximum possible value of each variable for attaining optimality.}
\end{remark}
\end{theorem}
\begin{proof}
(a) We prove by contradiction that the modified Algorithm~\ref{alg:micp-oracle} must terminate after a finite number of iterations. If not, let $x^*$ be a limit point of $\{x^{(n)}\}^\infty_{n=1}$,
and let $x^{(n_k)}\to x^*$ as $k\to\infty$.
Case 1: $\dist(x^*,\mtc{C})<\epsilon_2/4$. In this case, there exists an $n_k$ such that
$\dist(x^{(n_k)},\mtc{C})\le\epsilon_2/4$. Recall that by Line~\ref{lin:z^n_d^n} of Algorithm~\ref{alg:micp-e},
$z^{(n_k)}$ and $d^{(n_k)}$ are the $(\epsilon_1,\epsilon_1)$-accurate solution and objective of
\eqref{eqn:IterateProjection} at the main iteration $n_k$. \qrevision{By Definition~\ref{opt:e-cont}, we have
\beq\label{eqn:d(z,C)<e1}
\dist(z^{(n_k)},C)\le\epsilon_1.
\eeq}
It follows that
\beq\label{eqn:d<e2/2}
d^{(n_k)}=\norm{x^{(n_k)}-z^{(n_k)}}\le\dist(x^{(n_k)},C) + \dist(z^{(n_k)},C)\le\frac{\epsilon_2}{4} + \epsilon_1\le\frac{\epsilon_2}{2},
\eeq
\qrevision{where the last inequality holds due to $\epsilon_1=\epsilon^2_2\le\epsilon_2/4$.}
It implies that the algorithm should terminate by Line~\ref{lin:termination_condition} of 
Algorithm~\ref{alg:micp-e}, which leads to a contradiction. 
Case 2: $\dist(x^*,\mtc{C})\ge\epsilon_2/4$. 
In this case, there exists an $n_k$ such that
$\dist(x^{(n_k)},\mtc{C})\ge 3\epsilon_2/16$, and $\norm{x^{(n_i)}-x^{(n_j)}}\le\frac{\epsilon_2}{8}$
for all $n_i,n_j\ge n_k$. We have
\bdm
d^{(n_k)}=\norm{x^{(n_k)}-z^{(n_k)}}\ge\dist(x^{(n_k)},C) - \dist(z^{(n_k)},C)\ge \frac{3\epsilon_2}{16}-\epsilon_1.
\edm
and hence
\bdm
d^{(n_k)}-3\epsilon_1 \ge \frac{3\epsilon_2}{16}-4\epsilon_1 = \frac{\epsilon_2}{8} + \left(\frac{\epsilon_2}{16}-\epsilon_1\right)\ge \frac{\epsilon_2}{8}.
\edm
Therefore, we have 
\bdm
\norm{x^{(n_i)}-x^{(n_k)}}\le \frac{\epsilon_2}{8}\le d^{(n_k)}-3\epsilon_1, 
\edm
for any $n_i>n_k$.
Note that the cutting plane $(x^{(n_k)}-z^{(n_k)})^{\top}(x-z^{(n_k)})\le 0$
is added to the master problem. By Lemma~\ref{lem:e-accuracy-proj}(b),
we have 
\beq\label{eqn:violated_at_xni}
(x^{(n_k)}-z^{(n_k)})^{\top}(x^{(n_i)}-z^{(n_k)})>0.
\eeq
However, the cutting plane $(x^{(n_k)}-z^{(n_k)})^{\top}(x-z^{(n_k)})\le0$
has been added to the master problem at iteration $n_k$.
\eqref{eqn:violated_at_xni} indicates that this constraint is violated
by the master problem solution $x^{(n_i)}$, which is a contradiction.

(b) Suppose the modified algorithm terminates and returns an infeasible status. 
It implies that the master problem \eqref{opt:micp-master-e} 
is infeasible at some iteration $n$. 
Note that we have 
\beq\label{eqn:d>e2/2}
d^{(k)}>\frac{\epsilon_2}{2}
\eeq
for every $k\le n-1$, otherwise the algorithm will return 
a solution according to the termination criteria at Line~\ref{lin:termination_condition}.  
Let $d^{*(k)}$ denote the exact optimal objective of the projection problem 
at the $k^{\text{th}}$ iteration. We should have 
\beq\label{eqn:d*k}
\ba
d^{*(k)}&=\dist(x^{(k)},\mtc{C})\\
&\ge\dist(x^{(k)},z^{(k)})-\dist(z^{(k)},\mtc{C})\\
&=d^{(k)} - \dist(z^{(k)},\mtc{C})\\
&>\frac{\epsilon_2}{2}-\epsilon_1>\frac{\epsilon_2}{4}. \quad\textrm{\qrevision{(using \eqref{eqn:d(z,C)<e1} and \eqref{eqn:d>e2/2})}}
\ea
\eeq
\qrevision{Similar to \eqref{eqn:rmax},} we define the following violation depth of each added cutting plane
$(x^{(k)}-z^{(k)})^{\top}(x-z^{(k)})=0$ as
\bdm
r^{(k)}=\left\{
\begin{array}{ll}
\underset{z\in\mtc{C}\cap V^{(k)-}}{\max}\;\dist(z,V^{(k)+}) & \textrm{ if }\mtc{C}\cap V^{(k)-}\neq\emptyset, \\
0 &\textrm{ if }\mtc{C}\cap V^{(k)-}=\emptyset,
\end{array}
\right.
\edm
where $V^{(k)+}=\Set*{x\in\mbR^{l_1+l_2}}{(x^{(k)}-z^{(k)})^{\top}(x-z^{(k)})\le0}$
and
\newline 
$V^{(k)-}=\Set*{x\in\mbR^{l_1+l_2}}{(x^{(k)}-z^{(k)})^{\top}(x-z^{(k)})\ge0}$.
\qrevision{We need to derive an upper bound for $r^{(k)}$ to estimate the infeasibility.}
By Lemma~\ref{lem:e-accuracy-proj}(c), we have 
\bdm
r^{(k)}\le \frac{9(D+2d^{*(k)})\sqrt{d^{*(k)}\epsilon_1}}{d^{*(k)}} 
=9\sqrt{\epsilon_1}\left(\frac{D}{\sqrt{d^{*(k)}}}+2\sqrt{d^{*(k)}}\right).
\edm
Using $\frac{\epsilon_2}{4}\le d^{*(k)}\le D$ from \eqref{eqn:d*k}, it follows that
\bdm
\frac{D}{\sqrt{d^{*(k)}}}+2\sqrt{d^{*(k)}}\le\max\left\{3\sqrt{D},\;\frac{2D}{\sqrt{\epsilon_2}}+\sqrt{\epsilon_2}\right\}
=\frac{2D}{\sqrt{\epsilon_2}}+\sqrt{\epsilon_2},
\edm 
where we use the assumption $\epsilon_2\le D/64$ to get the last equality.
Therefore, we obtain the following inequality:
\bdm
r^{(k)}\le 9\sqrt{\epsilon_1}\left(\frac{2D}{\sqrt{\epsilon_2}}+\sqrt{\epsilon_2} \right)
=18D\sqrt{\epsilon_2}+9\epsilon_2\sqrt{\epsilon_2} \le 20D\sqrt{\epsilon_2},
\edm
where we use the setting $\epsilon_1=\epsilon^2_2$ to get the equality.
It implies that the set $\mtc{C}^{\delta-}\cap(\mathbb{Z}^{l_1}\times\mathbb{R}^{l_2})$ 
where $\delta=20D\sqrt{\epsilon_2}$ is contained in the feasible set  
of the master problem at each iteration. 
Since the master problem at iteration $n$ is infeasible,
\eqref{opt:convex-opt-linear-obj} is $(20D\sqrt{\epsilon_2})$-infeasible.
  
(c) If the algorithm terminates and returns a solution $x^{(m)}$,
we know that 
\bdm
\ba
\dist(x^{(m)},\mtc{C})&\le\dist(x^{(m)},z^{(m)}) + \dist(z^{(m)},\mtc{C}) \\
&=d^{(m)} + \dist(z^{(m)},\mtc{C}) \\
&\le\frac{\epsilon_2}{2}+\epsilon_1<20D\sqrt{\epsilon_2}. \quad\textrm{\qrevision{(using \eqref{eqn:d<e2/2} and \eqref{eqn:d(z,C)<e1})}}
\ea
\edm
It implies that $x^{(m)}\in\mtc{C}^{20D\sqrt{\epsilon_2}}_+$.
Let $P$ be the polytope obtained after relaxing the integral constraints
of the master problem at the termination of the algorithm. 
In the proof of Part (b), it has been shown that $\mtc{C}^{20D\sqrt{\epsilon_2}}_-\subseteq\mtc{C} \subseteq P$.
Consider the following optimal values:
\bdm
\begin{aligned}
c^{\top}x^{(m)} &= \min\;c^{\top}x\;\textrm{ s.t. } x\in P\,\srevision{\cap\{Ax=b\}}\cap(\mbZ^{l_1}\times\mbR^{l_2}), \\
\textrm{OPT}^- &= \min\;c^{\top}x\;\textrm{ s.t. } x\in \mtc{C}^{20D\sqrt{\epsilon_2}}_-\,\srevision{\cap\{Ax=b\}}\cap(\mbZ^{l_1}\times\mbR^{l_2}), \\
\textrm{OPT}^+ &= \min\;c^{\top}x\;\textrm{ s.t. } x\in \mtc{C}^{20D\sqrt{\epsilon_2}}_+\,\srevision{\cap\{Ax=b\}}\cap(\mbZ^{l_1}\times\mbR^{l_2}).
\end{aligned}
\edm
The following inequalities hold:
\bdm
\ba
&\textrm{OPT}^+\le c^{\top}x^{(m)}\le\textrm{OPT}^-, \\
&\opt^+\le\opt\le\opt^-,
\ea
\edm
where $\textrm{OPT}$ is the optimal objective of \eqref{opt:convex-opt-linear-obj}.
By Definition~\ref{def:epsilon-micp}(ii) and the assumption of $(20D\sqrt{\epsilon_2},\lambda)$-stability, 
there exist $x^+$ and $x^-$ satisfying
\begin{displaymath}
\begin{aligned}
&x^+\in\mtc{C}^{20D\sqrt{\epsilon_2}}_+\,\srevision{\cap\{Ax=b\}}\cap(\mbZ^{l_1}\times\mbR^{l_2}), \\
&x^-\in\mtc{C}^{20D\sqrt{\epsilon_2}}_-\,\srevision{\cap\{Ax=b\}}\cap(\mbZ^{l_1}\times\mbR^{l_2}), \\
&c^\top x^+=\opt^+, \\
&c^{\top}x^-=\opt^-, \\
&x^+_Z=x^-_Z=b \textrm{ for some } b\in\mbZ^{l_1}, \\
&\norm{x^+-x^-}\le 20\lambda D\sqrt{\epsilon_2}.
\end{aligned}
\end{displaymath}
It follows that $\opt^--\opt^+\le 20\lambda D\sqrt{\epsilon_2}\norm{c}$, and consequently
\bdm
|c^{\top}x^{(m)}-\text{OPT}|\le \opt^- - \opt^+ \le 20\lambda D\sqrt{\epsilon_2}\norm{c},
\edm
which concludes the proof.  \qed
\end{proof}

The above theorem addresses the quality of the solution in a rigorous manner. It says that one cannot 
purely talk about the accuracy of the solution and ignore stability in $\epsilon$-accuracy analysis, since it is possible that a small perturbation of the feasible region could cut off a near-optimal incumbent solution due to the discontinuity of the integral part of a solution. Theoretically, the theorem connects the level of accuracy and the level of stability. 	
We remark that the $(\epsilon,\lambda)$-stability is not easily verifiable exactly. 
This is because verification requires additional computation, which may bring additional numerical errors. However, it is possible to heuristically check this condition numerically. For example, one can solve problems \eqref{opt:C^e_+} and \eqref{opt:C^e_-}  
to get the optimal solutions $x^1$ and $x^2$, respectively and check if $\norm{x^1-x^2}$ is proportional to $\epsilon$. Note that the ratio $\norm{x^1-x^2}/\epsilon$ is related to the Lipschitz coefficients of constraint functions if they are Lipschitz continuous. On the other side, the parameters $(\epsilon,\lambda)$ used in the statement of the theorem are not necessarily pre-specified. Instead, they are adjustable to make the stability condition satisfied by the problem. One can think that there is a family of stability conditions parameterized by $(\epsilon,\lambda)$. If one of the conditions is satisfied by the problem, the corresponding parameters can be used to characterize the quality of the solution returned by the solver. Based on this interpretation, there is no need to identify the most restrictive stability condition, but it is possible to numerically find a less restrictive stability condition satisfied by the problem. For example, one can select a sufficiently small $\epsilon$ such that the problem is $\epsilon$-feasible and then select a sufficiently large $\lambda$ (can depend on $\epsilon$) to make 
the problem $(\epsilon,\lambda)$-stable. Again, there are numerical errors generated during this process of selecting $\epsilon$ and $\lambda$. Ultimately, verification of this condition requires an exact computation model.  

\subsection{Analysis for the JMICP problem}
\label{sec:error-JMICP}
For a given $x$, denote $\mtc{Q}(x)$ as the optimal objective of the following parametric MICP:
\beq\label{eqn:app-PMICP} 
\min_y\;h^\top y\quad\mathrm{ s.t. }\;\;\srevision{T x + W y = q,}\quad g_i(x,y)\le0\;\forall i\in\mtc{I},\;\;y\in\mtc{Y}\cap(\mbZ^{l_2}\times\mbR^{l_3}).
\eeq
\qrevision{
Algorithm~\ref{alg:micp-e} can be applied to solve the above PMICP to a desired accuracy. 
We can apply the results from Theorem~\ref{thm:e-lambda-micp} to analyze the algorithm's
performance on solving \eqref{eqn:app-PMICP}. }

\begin{definition}
A feasible solution $x$ of \eqref{opt:micp-xy} is $\epsilon$-optimal 
if $\textrm{Obj}(x)\le\textrm{Opt}+\epsilon$, where $\textrm{Obj}(x)$ is the
objective value of \eqref{opt:micp-xy} at $x$ and $\textrm{Opt}$ is the 
optimal objective of \eqref{opt:micp-xy}.
\end{definition}

\begin{assumption}\label{ass:PMICP-stable}
For any $\hat{x}\in\mtc{X}\cap\{0,1\}^{l_1}$, the convex set 
$\mtc{S}(\hat{x})=\{y\in\mbR^{l_2+l_3}:\;g_i(\hat{x},y)\le 0\;\forall i\in\mtc{I},\;y\in\mtc{Y}\}$ is in full dimension,
and the parametric MICP
$$\min\;h^\top y\quad\mathrm{ s.t. }\;\; \srevision{T \hat{x} + W y = q,}\;g_i(\hat{x},y)\le 0\;\forall i\in\mtc{I},\;y\in\mtc{Y}\cap(\mbZ^{l_2}\times\mbR^{l_3})$$
is $(\epsilon,\lambda)$-stable for some $\epsilon,\lambda>0$.
\end{assumption}

\qrevision{In the following {lemma}, we delve into how to quantify $\mtc{Q}(\hat{x})$
for any $\hat{x}\in\mtc{X}\cap\{0,1\}^{l_1}$ when 
Algorithm~\ref{alg:micp-e} is applied to solve \eqref{eqn:app-PMICP} with $x=\hat{x}$.
Intuitively, we can quantify $\mtc{Q}(\hat{x})$ using the primal and dual solutions of 
the tight LP-relaxation of \eqref{eqn:app-PMICP} ($x=\hat{x}$) generated by the algorithm and an error term.
Furthermore, since the PMICP will be solved repeatedly in a TSS-MICP with a changing $x$, 
we should also understand how the quantification works on a general $x$ even though
the quantities we used are derived from a specific $\hat{x}$. }
\begin{lemma}\label{lem:e-valid-PMICP}
Suppose that Assumptions~\ref{ass:convex-opt-oracle-errors} and \ref{ass:PMICP-stable} hold. 
Let $\hat{x}\in\mtc{X}\cap\{0,1\}^{l_1}$. Suppose Algorithm~\ref{alg:micp-e} is applied to  
the following parametric MICP problem:
\beq\label{eqn:lem-PMICP}
\min\;h^\top y\quad\mathrm{ s.t. }\;\;\srevision{T \hat{x} + W y = q,}\; g_i(\hat{x},y)\le 0\;\forall i\in\mtc{I},\;y\in\mtc{Y}\cap(\mbZ^{l_2}\times\mbR^{l_3})
\eeq
for solving it to the $(\rho\epsilon_1,\epsilon_1,\lambda)$-accuracy
for some sufficiently small parameters $\epsilon_1>0$ and some constant $\rho>0$. (This is achievable due to Theorem~\ref{thm:e-lambda-micp}.)
Suppose Algorithm~\ref{alg:micp-e} generates the following linear program at termination:
\beq\label{eqn:lem-LP}
\min_y\;h^\top y\quad\mathrm{ s.t. }\;\;\srevision{T \hat{x} + W y = q,}\; A\hat{x}+By+C\le 0,
\eeq
where $A$, $B$ and $C$ are some coefficient matrices. 
Then there exists some constant $\rho_1>0$ such that the following two inequalities hold:
\begin{align}
\mtc{Q}(x)&\ge \srevision{(-\lambda^{\top}T + \mu^\top A) x + (\lambda^{\top} q + \mu^\top C)} - \rho_1\lambda\epsilon_1\quad\forall x\in\mtc{X}\cap\{0,1\}^{l_1}, \label{eqn:param_valid1} \\
\mtc{Q}(\hat{x})&\le \srevision{(-\lambda^{\top}T + \mu^\top A) \hat{x} + (\lambda^{\top} q + \mu^\top C)} +\rho\epsilon_1, \label{eqn:param_valid2}
\end{align}
where $\lambda$ and $\mu\ge 0$ is the optimal dual vector of equality and inequality constraints in \eqref{eqn:lem-LP} and $\mtc{Q}(\cdot)$ is defined in \eqref{opt:y-micp}.
\end{lemma}
\begin{proof}
According to the mechanism of the cutting-plane method, the set $Ax+By+C\le 0$ of constraints
can be partitioned into two groups denoted as $A_1x+B_1y+C_1\le 0$ and $A_2x+B_2y+C_2\le 0$,
where the first group of constraints are projection cutting planes obtained by solving the projection problem
to a certain accuracy, while the second group of constraints are valid inequalities generated following the oracle in \cite{ChenKucukyavuzSen2011} for solving an MILP with pure cutting planes. 

Let $\mtc{S}$ be defined in \eqref{eqn:set_S}. 
Since the inequalities $A_1x+B_1y+C_1\le 0$ are generated by solving the projection problem
to a certain accuracy. It follows that there exists some constant $\rho_1>0$ such that
\begin{align}
&\mtc{S}^{\rho_1\epsilon_1}_{-}\subseteq\{(x,y): A_1x+B_1y+C_1\le 0\}. \label{eqn:S_in_Ax+B+C}
\end{align}
Furthermore, the oracle in \cite{ChenKucukyavuzSen2011}
generate the valid inequalities $A_2x+B_2y+C_2\le 0$ such that
\begin{align}
&\conv\{(x,y): \srevision{T x + W y = q},\;A_1x+B_1y+C_1\le 0,\; y\in\mbZ^{l_2}\times\mbR^{l_3}\} \nonumber \\
& \qquad \subseteq\{(x,y): \srevision{T x + W y = q},\; Ax+By+C\le 0\}.
\end{align}
Let $\wt{Q}(\hat{x})$ be the approximated objective. The accuracy control implies that 
\beq\label{eqn:Q(x0)_approx}
\ba
&\mtc{Q}(\hat{x})+\rho\epsilon_1\ge\wt{Q}(\hat{x})  \\
 &= \min_y\; h^{\top}y\; \{\textrm{s.t. } \srevision{T \hat{x} + W y = q},\; A\hat{x}+By+C\le 0\} \\
 &\ge \mtc{Q}(\hat{x})-\rho\epsilon_1.
 \ea
\eeq
The strong duality gives
\beq\label{eqn:Q(x0)_approx_dual}
\srevision{(-\lambda^{\top}T + \mu^\top A) \hat{x} + \lambda^{\top} q + \mu^\top C}=\min_y\; h^{\top}y\; \{\textrm{s.t. } \srevision{T \hat{x} + W y = q},\; A\hat{x}+By+C\le 0\}.
\eeq
Combining \eqref{eqn:Q(x0)_approx} and \eqref{eqn:Q(x0)_approx_dual} leads to \eqref{eqn:param_valid2}.
For any $x\in\mtc{X}\cap\{0,1\}^{l_1}$, Assumption~\ref{ass:PMICP-stable} about stability implies that
\beq\label{eqn:Qx+lambda_rho_e}
\ba
&\mtc{Q}(x)=\min_y\; h^{\top}y\; \{\textrm{s.t. } (x,y)\in\mtc{S}\;\srevision{\cap\{T x + W y = q\}},\; y\in\mbZ^{l_2}\times\mbR^{l_3}\}, \\
&\mtc{Q}(x) + \lambda\rho_1\epsilon_1\ge \min_y\; h^{\top}y\; \{\textrm{s.t. } (x,y)\in\mtc{S}^{\rho_1\epsilon_1}_-\;\srevision{\cap\{T x + W y = q\}},\;y\in\mbZ^{l_2}\times\mbR^{l_3}\}. 
\ea
\eeq
Using \eqref{eqn:S_in_Ax+B+C} and \eqref{eqn:Qx+lambda_rho_e} leads to the following inequalities:
\beq
\ba
\mtc{Q}(x) + \lambda\rho_1\epsilon_1&\ge \min_y\; h^{\top}y\; \{\textrm{s.t. } \srevision{T x + W y = q},\, A_1x+B_1y+C_1\le0,\;y\in\mbZ^{l_2}\times\mbR^{l_3}\} \\
&\ge \min_y\; h^{\top}y\;\{\textrm{s.t. }\srevision{T x + W y = q},\, Ax+By+C\le 0 \} \\
&\ge \srevision{(-\lambda^{\top}T + \mu^\top A) x + (\lambda^{\top} q + \mu^\top C)} ,
\ea
\eeq
which proves \eqref{eqn:param_valid1}.  \qed
\end{proof}

\begin{theorem}
Let Assumption~\ref{ass:convex-opt-oracle-errors} and \ref{ass:PMICP-stable} hold. 
Then Algorithm~\ref{alg:decomp-JMICP-e} can identify a 
$K\epsilon_1$-optimal solution of \eqref{opt:micp-xy} in a finite number of iterations,
where $K$ is a constant and $\epsilon_1$ is the accuracy control parameter used in
Algorithm~\ref{alg:decomp-JMICP-e}.
\end{theorem}
\begin{proof}
We first show that if the algorithm terminates, it will return a $K\epsilon_1$-optimal solution of \eqref{opt:micp-xy}.
Suppose the algorithm terminates at iteration $m$ and returns a solution $x^m$. 
The upper bound $U^m$ at the termination step is determined by the following equation
\beq\label{eqn:JMICP-e-Um}
U^m=c^\top x^m + \wt{\mtc{Q}}(x^m)\ge c^\top x^m + \mtc{Q}(x^m) - \rho_1\epsilon_1.
\eeq
The lower bound $L^m$ at the termination step is equal to the optimal value of the master problem:
\beq
\ba
L^m=&\min_{x\in\mX\cap\{0,1\}^{l_1}}\; c^\top x + \eta \\
&\qquad\text{ s.t. }\; \eta \ge a^k x+b^k \quad \forall k\in[m-1], \\
&\qquad\qquad R^{m-1}x + S^{m-1} \le 0,
\ea
\eeq
where $R^{m-1}x + S^{m-1} \le 0$ are valid for $x\in\mX\cap\{0,1\}^{l_1}$,
and by Lemma~\ref{lem:e-valid-PMICP} \eqref{eqn:param_valid1} 
the coefficients $a^k,b^k$ in the inequality $\eta\ge a^kx+b^k$ satisfy
\beq
\mtc{Q}(x)\ge a^kx+b^k-\rho_2\epsilon_1 \quad\forall x\in\mtc{X}\cap\{0,1\}^{l_1}
\eeq
for some constant $\rho_2$. Therefore,
\beq\label{eqn:JMICP-e-opt>Lm}
\ba
\textrm{Opt}&=\min_{x\in\mX\cap\{0,1\}^{l_1}}\; c^\top x + \mtc{Q}(x) \\
&\ge\min_{x\in\mX\cap\{0,1\}^{l_1}}\; c^\top x + \eta\quad\{\textrm{s.t. } \eta\ge a^kx+b^k-\rho_2\epsilon_1\} \\
&\ge L^m-\rho_2\epsilon_1.
\ea
\eeq
Combining \eqref{eqn:JMICP-e-Um} and \eqref{eqn:JMICP-e-opt>Lm} gives
\bdm
\ba
 U^m + \rho_1\epsilon_1\ge c^\top x^m + \mtc{Q}(x^m)\ge\textrm{Opt}\ge L^m-\rho_2\epsilon_1
\ea
\edm
and hence
\bdm
|c^\top x^m + \mtc{Q}(x^m)-\textrm{Opt}|\le U^m-L^m+(\rho_1+\rho_2)\epsilon\le (\rho_0+\rho_1+\rho_2)\epsilon_1,
\edm
where the last inequality is due to the termination criteria of Algorithm~\ref{alg:decomp-JMICP-e}. This shows that the returned solution $x^m$ satisfies the desired accuracy by setting
$K=\rho_0+\rho_1+\rho_2$. 

It suffices to show that Algorithm~\ref{alg:decomp-JMICP-e} can terminate in a finite number of iterations. We prove it by contradiction. Suppose it does not terminate, then there exist two iteration indices $k_1$ and $k_2$ with $k_1<k_2$ such that $x^{k_1}=x^{k_2}$, and
\begin{equation}\label{eqn:Lk2-epsilon}
\begin{aligned}
L^{k_2}&=c^{\top}x^{k_2}+\eta^{k_2} \\
&\ge c^{\top}x^{k_2} + a^{k_1}x^{k_2} + b^{k_1} \\
& = c^{\top}x^{k_1} + a^{k_1}x^{k_1} + b^{k_1} \\
& = c^{\top}x^{k_1} + \wt{\mtc{Q}}(x^{k_1}) \\
& = U^{k_1}=U^{k_2}.
\end{aligned}
\end{equation}
It shows that the algorithm should terminate.  \qed
\end{proof}

\begin{algorithm}
	\caption{A decomposition cutting-plane algorithm for \eqref{opt:micp-xy}.}
	\label{alg:decomp-JMICP-e}
	\begin{algorithmic}[1]
	\State{Set $m\gets 1$, $L\gets -\infty$ and $U\gets\infty$.}
	\While{$U-L>\rho_0\epsilon_1$}
		\State{Solve the current first-stage MILP problem \eqref{opt:xy-master} to optimality for the main iteration $m$
			   using the cutting-plane oracle from \cite{ChenKucukyavuzSen2011}, and get the optimal solution $x^m$. }
		\State{Set $L\gets obj\eqref{opt:xy-master}$, where $obj\eqref{opt:xy-master}$ represents the optimal objective of \eqref{opt:xy-master}.}
		\State{Substitute $x^m$ into the parametric MICP \eqref{opt:y-micp}, 
			   and apply Algorithm~\ref{alg:micp-e} to get a $(\rho_1\epsilon_1,\epsilon_1,\lambda)$-accurate solution
			   of \eqref{opt:y-micp} for some constant $\rho_1$. (This is achievable due to Theorem~\ref{thm:e-lambda-micp}.)}

        {
            \If{\eqref{opt:y-micp} is detected to be infeasible}
				\State{Add a feasibility cut \eqref{eqn:feasb-cut} to the first-stage master problem \eqref{opt:xy-master}.}
            \EndIf}
		\State{Set $U\gets c^\top x^m +  \wt{Q}(x^m)$, where $\wt{Q}(x^m)$ is the 
			 approximated objective of \eqref{opt:y-micp} given by Algorithm~\ref{alg:micp-e}.}
		\State{When Algorithm~\ref{alg:micp-e} terminates with an approximated optimal solution
			   to \eqref{opt:y-micp}, a linear program \eqref{opt:y-lp} is established. Generate
			   a Benders' cut \eqref{eqn:bender-cut} from the dual objective of \eqref{opt:y-lp},
			   and add it to \eqref{opt:xy-master} as $\eta\ge {a^m}^\top x+b^m$.}
		\State{Set $m\gets m+1$.}
	\EndWhile
	\State{\Return $x^m$ as an approximated optimal solution of \eqref{opt:micp-xy}.}
	\end{algorithmic}
\end{algorithm}

\subsection{Analysis for the TSS-MICP problem}
\label{sec:error-two-stage}

\begin{definition}
A feasible solution $x$ of (\eqref{opt:TSS-MICP}) is $\epsilon$-optimal 
if $\textrm{Obj}(x)\le\textrm{Opt}+\epsilon$, where $\textrm{Obj}(x)$ is the
objective value of (\eqref{opt:TSS-MICP}) at $x$ and $\textrm{Opt}$ is the 
optimal objective of (\eqref{opt:TSS-MICP}).
\end{definition}

\begin{theorem}
Let Assumption~\ref{ass:convex-opt-oracle-errors} hold. 
Suppose the master problem (\ref{opt:master}) and all scenario sub-problems (\ref{opt:second-stage})
at every iteration is $(\epsilon,\lambda)$-stable, 
and they are solved to $(\rho\epsilon,\epsilon,\lambda)$-accuracy using Algorithm~\ref{alg:micp-e}
for some constant $\rho>0$. (This is achievable due to Theorem~\ref{thm:e-lambda-micp}.)
Suppose the worst-case scenario detection problem (\ref{opt:worst-case-distr}) at every iteration is 
solved to $(\epsilon,\epsilon)$-accuracy (Definition~\ref{def:e-cont}).  
Then Algorithm~\ref{alg:decomp-BC-error} can identify a 
$K\epsilon$-optimal
solution of \eqref{opt:TSS-MICP} in a finite number of iterations, where $K>0$ is a constant.
\end{theorem}
\begin{proof}
The proof is similar to that of Theorem~\ref{thm:finite-convg}, but considering solution inaccuracy. We first show that if the algorithm terminates, it will return
a $K\epsilon$-optimal 
solution of (\eqref{opt:TSS-MICP}) for some $K>0$. Suppose the algorithm terminates at iteration $m$ and returns a solution $x^m$. The upper bound $U^m$ at the termination iteration
is determined by the following equation
\beq
U^m=c^{\top}x^m+\sum_{\omega\in\Omega}p^{m\omega}_{\epsilon}\wt{\mQ}(x^m,\xi^\omega),
\eeq
where $\wt{\mQ}(\cdot,\xi^\omega)$ is an approximated objective of $\sub(\cdot,\xi^\omega)$
as the output of applying Algorithm~\ref{alg:micp-e}, i.e., it is an approximated evaluation of $\mQ(\cdot,\xi^\omega)$.
The vector $p^{m}_{\epsilon}$ is a $(\epsilon,\epsilon)$-accurate solution of the convex program
$\max_{p\in\mP}\;\sum_{\omega\in\Omega}p^{\omega}\wt{\mQ}(x^m,\xi^\omega)$. 
By the assumption of the convex optimization oracle, we should have 
\beq
\ba
& |\wt{\mQ}(x^m,\xi^\omega)-\mQ(x^m,\xi^\omega)|\le\rho\epsilon \qquad\forall\omega\in\Omega,  \\
& \Big|\sum_{\omega\in\Omega}p^{m\omega}_{\epsilon}\wt{\mQ}(x^m,\xi^\omega)
-\underset{P\in\mathcal{P}}{\textrm{max}}\;\E_P[\wt{\mQ}(x^m,\xi)]\Big|\le\epsilon. 
\ea
\eeq
Let $\textrm{Obj}(x^m)$ be the objective value of (\eqref{opt:TSS-MICP}) evaluated at $x^m$.
We can bound the error $|U^m-\textrm{Obj}(x^m)|$ as follows:
\beq\label{eqn:U-Obj}
\ba
&|U^m-\textrm{Obj}(x^m)|=
\Big|\sum_{\omega\in\Omega}p^{\omega}_{\epsilon}\wt{\mQ}(x^m,\xi^\omega)
-\underset{P\in\mathcal{P}}{\textrm{max}}\;\E_P[\mQ(x^m,\xi)]\Big| \\
&\le \Big| \sum_{\omega\in\Omega}p^{\omega}_{\epsilon}\wt{\mQ}(x^m,\xi^\omega)
-\underset{P\in\mathcal{P}}{\textrm{max}}\;\E_P[\wt{\mQ}(x^m,\xi)]\Big|
+\Big|\underset{P\in\mathcal{P}}{\textrm{max}}\;\E_P[\wt{\mQ}(x^m,\xi)]
-\underset{P\in\mathcal{P}}{\textrm{max}}\;\E_P[\mQ(x^m,\xi)]\Big| \\
&\le\epsilon + \rho\epsilon.
\ea
\eeq
Let $\textrm{Opt}_m$ be the exact optimal objective value of \eqref{opt:term-master} at iteration $m$.
Since $(x^m,\eta^m)$ is a $(\rho\epsilon,\epsilon,\lambda)$-accurate solution of \eqref{opt:term-master},
we have $\textrm{Opt}_m\ge L-\rho\epsilon$, where $L=c^\top x^m+\eta^m$.
Furthermore, the strong duality implies that
\bdm
\ba
&\wt{\mQ}(x^k,\xi^\omega)=r^{k\omega\top}x^k+u^{k\omega} \quad \forall \omega, k.
\ea
\edm
Note that since each scenario sub-problem is a JMICP, we can apply Lemma~\ref{lem:e-valid-PMICP}
to get the following set of inequalities for some constant $\rho_1>0$:
\beq\label{eqn:inacc-dual-ineq-1}
\ba
&r^{k\omega\top}x+u^{k\omega}\le\mQ(x,\xi^\omega) + \rho_1\epsilon \quad \quad \forall \omega, k, x, \\
&r^{k\omega\top}x^k+u^{k\omega}\ge\mQ(x^k,\xi^\omega) - \rho_1\epsilon \quad  \forall \omega, k.
\ea
\eeq
The inequality \eqref{eqn:U-Obj} further indicates that
\beq\label{eqn:E<sumpQ<E}
\underset{P\in\mathcal{P}}{\textrm{max}}\;\E_P[\mQ(x^k,\xi)]
-(\rho+1)\epsilon
\le\sum_{\omega\in\Omega}p^{k\omega}_\epsilon\wt{\mQ}(x^k,\omega)\le
\underset{P\in\mathcal{P}}{\textrm{max}}\;\E_P[\mQ(x^k,\xi)]
+(\rho+1)\epsilon.
\eeq

Claim: Let $x^*$ be the optimal solution of \eqref{opt:TSS-MICP} and 
$\eta^\prime=\max_{P\in\mtc{P}}\E_P[\mQ(x^*,\xi)]+\rho_1\epsilon$. Then
$(x^*,\eta^\prime)$ is a feasible solution of \eqref{opt:term-master} at any master iteration.
\newline
Proof of the claim: it suffices to verify that the following inequality holds for each $k$
\beq
\max_{P\in\mtc{P}}\E_P[\mQ(x^*,\xi)]+\rho_1\epsilon\ge \sum_{\omega\in\Omega}p^{k\omega}_\epsilon r^{k\omega\top}x^*
+\sum_{\omega\in\Omega}p^{k\omega}_\epsilon u^{k\omega}.
\eeq
Indeed, we have 
\bdm
\ba
&\max_{P\in\mtc{P}}\E_P[\mQ(x^*,\xi)]+\rho_1\epsilon\ge \sum_{\omega\in\Omega}p^{k\omega}_\epsilon\mQ(x^*,\xi)+\rho_1\epsilon \\
&\ge \sum_{\omega\in\Omega}p^{k\omega}_\epsilon(r^{k\omega\top}x^*+u^{k\omega}-\rho_1\epsilon) + \rho_1\epsilon 
\quad (\textrm{using }\eqref{eqn:inacc-dual-ineq-1}) \\
&=\sum_{\omega\in\Omega}p^{k\omega}_\epsilon(r^{k\omega\top}x^*+u^{k\omega}),
\ea
\edm
which concludes the proof of the claim.

Let $\textrm{Opt}$ be the optimal objective of \eqref{opt:TSS-MICP}. The claim indicates that
\beq\label{eqn:opt>Lm-rho_e}
\ba
\textrm{Opt}&=c^\top x^*+\max_{P\in\mtc{P}}\E_P[\mQ(x^*,\xi)] \\
&= c^\top x^* + \eta^\prime -  \rho_1\epsilon \\
&\ge\textrm{Opt}_m -  \rho_1\epsilon \\
&=L^m- \rho_1\epsilon,
\ea
\eeq
where $L^m$ is the value of $L$ at the iteration $m$.
Therefore, the following inequalities hold:
\bdm
\ba
\textrm{Obj}(x^m)&\le U^m + (1+\rho)\epsilon \quad \textrm{(using \eqref{eqn:U-Obj})} \\
&\le L^m + \rho_0\epsilon + (1+\rho)\epsilon \quad \textrm{(termination criteria)} \\
&\le \textrm{Opt} + \rho_1\epsilon+ \rho_0\epsilon + (1+\rho)\epsilon \quad \textrm{(using \eqref{eqn:opt>Lm-rho_e})}.
\ea
\edm
The above shows that $x^m$ is a $(1+\rho_0+\rho+\rho_1)\epsilon$-optimal solution.

Next we show that the algorithm terminates in a finite number of iterations. 
Suppose it does not terminate, then there exist two iteration indices $k_1$
and $k_2$ with $k_1<k_2$ such that $x^{k_1}=x^{k_2}$. Then
\begin{equation}
\begin{aligned}
L^{k_2}&=c^{\top}x^{k_2}+\eta^{k_2} \\
&\ge c^{\top}x^{k_2}+\sum_{\omega\in\Omega}p^{k_1\omega}_{\epsilon}r^{k_1\omega\top}x^{k_2}
+\sum_{\omega\in\Omega}p^{k_1\omega}_{\epsilon}u^{k_1\omega} \\
&=c^{\top}x^{k_1}+\sum_{\omega\in\Omega}p^{k_1\omega}_{\epsilon}\wt{\mQ}(x^{k_1},\xi^\omega) \\
&=U^{k_1}=U^{k_2}, 
\end{aligned}
\end{equation}
which shows that the algorithm should terminate.  \qed
\end{proof}

\begin{algorithm}
	\caption{A decomposition cutting-plane algorithm for solving (\eqref{opt:TSS-MICP})  to a given accuracy.}
	\label{alg:decomp-BC-error}
	\begin{algorithmic}[1]
	\State{Initialization: $L\gets -\infty$, $U\gets \infty$, $m\gets 1$.}
	\While{$U-L>\rho_0\epsilon$} 
		\State{Solve the following approximated first-stage master problem} 
		\State{to obtain 
		$(\rho\epsilon,\epsilon,\lambda)$-accurate solution $(\eta^m,x^m)$
			using Algorithm~\ref{alg:micp-e}} 
		\State{for some constant $\rho$: (This is achievable due to Theorem~\ref{thm:e-lambda-micp}.)}
		\beq\label{opt:term-master}
		\ba
		&\min_x\;c^{\top}x + \eta \\
		&\textrm{ s.t. } Ax\le b, \\
		&\qquad \eta\ge\sum_{\omega\in\Omega}p^{k\omega}_{\epsilon}r^{k\omega\top}x
		+\sum_{\omega\in\Omega}p^{k\omega}_{\epsilon}u^{k\omega} \quad\forall k\in[m-1], \\
		&\qquad x\in\mC\cap\{0,1\}^n,
		\ea
		\eeq
		\State{where $\{p^{k\omega}_\epsilon:\omega\in\Omega\}$ is a $(\epsilon,\epsilon)$-accurate solution 
		 of \eqref{opt:worst-case-distr} obtained by the convex optimization oracle. (See Line~\ref{lin:get-worst-case-P}.)
		 The coefficients $r^{k\omega}$ and $u^{k\omega}$ are specified in Line~\ref{lin:call-alg1}.}
		\State{Update the lower bound as $L\gets c^\top x^m+\eta^m$.}
		\For{$\omega\in\Omega$:}
			\State{Solve the second-stage problem $\textrm{Sub}(x^m,\omega)$ to get a 
			$(\rho\epsilon,\epsilon,\lambda)$-accurate
			solution using Algorithm~\ref{alg:micp-e} for some constant $\rho$,
			and let $\wt{\mtc{Q}}(x^m,\xi^\omega)$ be the approximated objective given by the algorithm.
			(This is achievable due to Theorem~\ref{thm:e-lambda-micp}.)
			When Algorithm~\ref{alg:micp-e} terminates, it generates a LP
			relaxation in the form of (\ref{opt:two-stage-relax-full}) with $\hat{x}=x^m$,
			which yields an approximated Benders' inequality written as
			$\mtc{Q}(x^m,\xi^\omega)+\rho\epsilon\ge r^{m\omega\top}x+u^{m\omega}$.} \label{lin:call-alg1}
		\EndFor
		\State{Solve (\ref{opt:worst-case-distr}) with $\hat{x}=x^m$ (using the convex-optimization oracle) 
			to get a $(\epsilon,\epsilon)$-accurate solution but satisfying $\sum_{\omega\in\Omega}p^{m\omega}_\epsilon=1$.
			(Note that the normalization equality can be achieved by perturbing entries of the $p^m_\epsilon$ vector.)} 
		\State{which is denoted as $p^m_{\epsilon}=\Set*{p^{m\omega}_{\epsilon}}{\omega\in\Omega}$.} \label{lin:get-worst-case-P}
		\State{Generate the following aggregated Benders cut with iteration index $m$:} 
		\State{$\eta\ge\sum_{\omega\in\Omega}p^{m\omega}_\epsilon r^{m\omega\top}x
+\sum_{\omega\in\Omega}p^{m\omega}_\epsilon u^{m\omega}$, and add it to \eqref{opt:term-master}.}
		\State{Update the upper bound as $U\gets c^{\top}x^m+\sum_{\omega\in\Omega}p^{m\omega}_\epsilon\wt{\mQ}(x^m,\xi^\omega)$}.
		\State{$m\gets m+1$.}
	\EndWhile
	\State{Return $x^m$.}
	\end{algorithmic}
\end{algorithm}

\section{Numerical Experimentation Details}\label{sec:Num-Details}

\begin{table}[!hbtp]
\centering
\caption{List of index and set.}
\resizebox{\textwidth}{!}{%
\begin{tabular}{|cl|ll|}
\hline
\multicolumn{2}{|c|}{\textbf{Indices}} & \multicolumn{2}{c|}{\textbf{Sets}} \\ \hline
$i$ & process & $I(j)$ & set of processes that consume chemical $j$ \\ 
 $j$ & chemical & $O(j)$ & set of processes that produce chemical $j$\\ 
$s$ & production scheme & $R(j)$ & set of suppliers that provide chemical $j$\\ 
$r$ & supplier &  $PS(I)$ & set of production schemes for process $i$\\ 
$p$ & plant & $JM(i, s)$ & set of main products for production scheme $s$ of process $i$\\ 
$r$ & customer & $L(i, s)$ & set of chemicals involved in production scheme $s$ of process $i$ (linear)  \\ 
$\omega$ & scenarios &  $\bar{L}(i, s)$ & index set of chemicals involved in production scheme $s$ of process $i$ (nonlinear)\\ \hline
\end{tabular}
}
\label{tab:index-set}
\end{table}

\begin{table}[!hbtp]
\centering
\caption{List of model parameters.}
\resizebox{\textwidth}{!}{%
\begin{tabular}{|c|l|}
\hline
\textbf{Notation} & \textbf{Description}\\ \hline
$PU^U_{r p j}$ & maximum amount of chemical $j$ transported from supplier $r$ to plant $j$ \\
 $Q^U_{p i}$ & maximum capacity of process $i$ in plant $p$ \\ 
  $Q^c_{p i l}$ & $\ell^{th}$ discrete capacity choice for process $i$ of plant $p$\\ 
 $F^U_{p c j}$ & maximum amount of chemical $j$ transported from plant $j$ to customer $c$ \\
$\mu_{i j s w} $ & material balance coefficients for chemical $j$ in process $i$ under production scheme $s$ and scenario $\omega$ \\
$\rho_{i j s} $ & relative maximum production rate of main product $j \in JM(i, s)$ for production scheme $s$ \\
& for continuous flexible process $i$ reference to a base scheme $\Bar{s}$\\ 
$H_i$ & maximum time for which process $i$ is available for production \\ 
$D_{c j \omega}$ & demand of chemical $j$ from customer $c$ in scenario $\omega$\\ 
$\beta^S_{r j}$ & price for purchase of chemical $j$ from supplier $r$ \\ 
$\beta^C_i$ & variable cost of installation of process $i$ \\ 
$\alpha^C_i$ & fixed cost of installation of process $i$ \\ 
$\beta^{R P}_{r p}$ & variable cost for transporting chemical from supplier $r$ to plant $p$ \\
$\alpha^{R P}_{r p} $ & fixed cost for transporting chemical from supplier $r$ to plant $p$ \\
$\beta^{P C}_{p c} $ & variable cost for transporting chemical from plant $p$ to customer $c$ \\
$\alpha^{P C}_{p c} $ & fixed cost for transporting chemical from plant $p$ to customer $c$  \\
 $\delta_{i s} $ & unit operating cost for production scheme $s$ for process $i$ \\
 $\tau_{\omega} $ & probability of scenario $\omega$ \\ \hline 
\end{tabular}
}
\label{tab:model-parameter}
\end{table}

\begin{table}[!hbtp]
\centering
\caption{List of decision variables.}
\resizebox{\textwidth}{!}{%
\begin{tabular}{|c|l|l|}
\hline
\textbf{Notation}& \textbf{Type} & \textbf{Description}\\ \hline
 $x_{pi} $ & Binary & 1 if process $i$ is installed in plant $p$ \\ 
$Q_{p i \ell} $ & Binary & capacity of process $i$ in plant $p$ with discrete value indicator index $\ell = {0, 1, 2, 3, 4}$\\ 
$PU_{rpj\omega}$ & Continuous & purchase amount  of chemical $j$ for plant $p$ from supplier $r$ in scenario $\omega$\\ 
$y^R_{r p \omega}$ & Binary & 1 if a chemical is transported from supplier $r$ to plant $p$ in scenario $\omega$ \\ 
$F_{pcj\omega}$ & Continuous & amount of chemical $j$ transported from plant $p$ to customer $c$ in scenario $\omega$\\ 
$y^C_{p c \omega}$ & Binary & 1 if a chemical is transported from plant $p$ to customer $c$ in scenario $\omega$  \\ 
$T_{p i j s \omega}$ & Continuous & length of time assigned to the production of main product $j$ for production scheme $s$ \\
& &  of process $i$ in plant $p$ in scenario $\omega$ \\
 $\theta_{p i j s \omega}$ & Continuous & amount of main product $j, j \in JM(i, s)$ produced from production scheme $s$ of process $i$ in plant $p$  \\
& &  in scenario $\omega$ \\
$W_{p i j s \omega} $ & Continuous & amount of chemical $j$ produced from production scheme $s$ of process $i$ in plant $p$  \\
& & in scenario $\omega$ \\ \hline
\end{tabular}
}
\label{tab:model-decision-var}
\end{table}

\begin{table}[!hbtp]
\centering
\caption{Problem dimension.}
\setlength{\tabcolsep}{6pt} 
\resizebox{\textwidth}{!}{ 
\begin{tabular}{ccccccc}
\toprule
Supplier (r) & Plant (p) & Capacity ($\ell$) & Process (i) & Scheme (s) & Product (j) & Customer (c) \\
\midrule
4 & 3 & 5 & 4 & 4 & 6 & 4 \\
\bottomrule
\end{tabular}}
\label{tab:dimension-parameter}
\end{table}

\subsection{A model for chemical supply chain problem}\label{sec:chem-model}
\small
\textbf{Objective function}
\begin{align*}
& \min \;\; \text{Cost} = \sum_p \sum_i \left( \beta_i^c Q_{pi} + \alpha_i^c x_{pi} \right)  
+ \sum_\omega \tau_{\omega} \left( \sum_p \sum_{i \in S(p)} \delta_{is} \sum_{j \in JM(i,s)} \rho_{ijs} \theta_{pijs\omega} \right) \\
& + \sum_p \sum_j \sum_{r \in R(j)} \left( \beta_j^s + \beta_{rp}^{RP} \right) PU_{rpj\omega} 
+ \sum_p \sum_r \alpha_{rp}^{RP} y_{rp\omega}^R \\
& + \sum_p \sum_c \alpha_{pc}^{PC} P_{pc\omega}^C 
+ \sum_p \sum_c \sum_j \beta_{pc}^{PC} F_{pcj\omega}
\end{align*}

\textbf{First-stage constraints}

 Assignment constraint: $\sum_{\ell} Q^c_{p i \ell} Q_{p i \ell} \leq Q^U_{p i} x_{p i}, \;\; \sum_{\ell} Q_{p i \ell} \leq 1 \quad \forall p, i$

\textbf{Second-stage constraints}
\begin{enumerate}
\item Supplier to plant shipment: $PU_{r p j \omega} \leq PU^U_{r p j} y^R_{r p \omega} \quad j, r, r\in R(j), p, \omega$

\item Plant to customer shipment: $ F_{pcj\omega} \leq F^U_{pcj} y^C_{p c \omega} \quad \forall j, p, c, \omega$

\item Demand satisfaction: $
\sum_{p} F_{p c j \omega} \geq D_{c j \omega}\;\; \forall c, j, \omega$
\item Flow/mass-balance: $$\sum_{r \in R(j)} PU_{r p j \omega} + \sum_{i \in O(j)} \sum_{s \in PS(i)} W_{p i j s \omega} = \sum_{c} F_{p c j \omega} + \sum_{i \in I(j)} \sum_{s \in PS(i)} W_{p i j s \omega} \; \forall p, j, \omega$$ 

\item Relation between main product and chemical (mass unit = production rate  $\times$ mass unit-time): $W_{p i j s \omega} = \rho_{i j s} \theta_{p i j s \omega} \quad \forall p, i, s \in PS(i), j \in JM(i, s), \omega$
\item Capacity bound in mass unit-time:
\begin{align*}
    & \sum_{s \in PS(i)} \sum_{j \in JM(i, s)} \theta_{p i j s \omega} \leq H_i\; Q_{p i}, \quad \forall p, i, \omega
\end{align*}
\item Linear:
$$W_{p i j s \omega} = \mu_{i j s \omega} W_{p i j' s \omega} \quad \forall p, i, j \in L(i, s), j' \in JM(i, s), s\in PS(i), \omega$$
\item Nonlinear production rate: $$ a_{p i s} W^{\gamma_1}_{p i j s \omega} W^{\gamma_2}_{p i j s \omega} \geq \mu_{i j s \omega} W_{p i j' s \omega} \quad \forall p, i, j \in \bar{L}(i, s), j' \in JM(i, s), s\in PS(i), \omega,$$
where $\gamma_1 + \gamma_2 =1$ and we set $\gamma_1 = 0.45$ and $\gamma_2 = 0.55$.
\end{enumerate}

\begin{figure}
\centering
\includegraphics[scale=0.4]{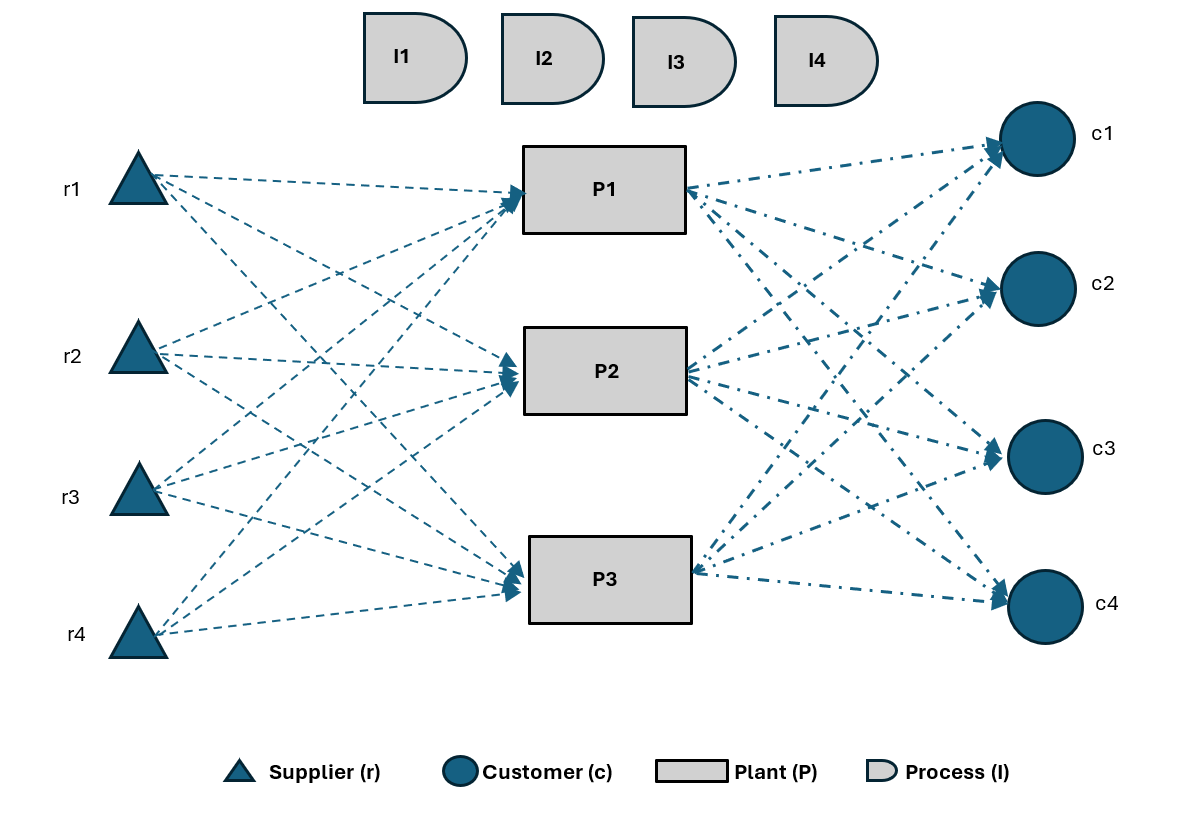}
\caption{Connection of suppliers and customers with plants}
\label{fig:ChemSC-Network}
\end{figure} 
\begin{figure}
\centering
\includegraphics[scale=0.4]{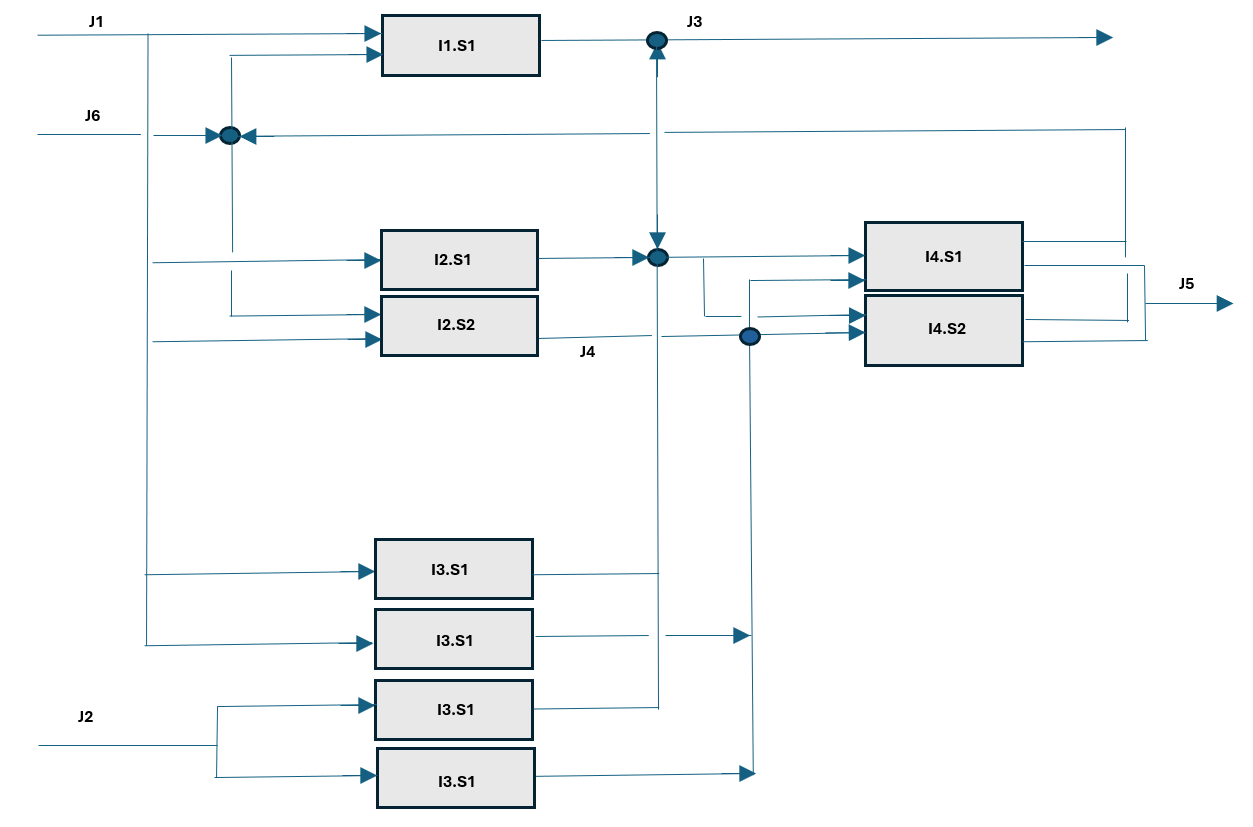}
\caption{Process($i$)-scheme($s$) connection in producing output ($j$) for each plant($p$)}
\label{fig:ChemSC}
\end{figure}

Given the first-stage decision to open a plant and its capacity to meet demand, the second-stage constraints (1-3) determine how to connect suppliers and customers to the opened plants so that the corresponding demand is satisfied at minimum cost. Figure~\ref{fig:ChemSC-Network} schematically shows the connections possible among the supply chain players. These connection decisions depend on the types and quantities of raw materials, their conversion rates to the final product, and compliance with capacity limits defined by the constraints (4-6). Each raw material supplied must pass through specific processing pathways to produce the final output, as illustrated in Figure~\ref{fig:ChemSC}. In certain processing units, production follows a linear rate, while in others, a nonlinear rate applies, as respectively indicated by constraint sets 7 and 8. The latter follows the Cobb-Douglas production function.

\subsection{Data generation}\label{sec:implementation details}
We follow the same data generation procedure followed by the literature \cite{li2019finite, li2018improved, norton1994strategic}. For adapted discrete first stage decision $Q_{p i \ell}$, we consider five possible values for each $Q^c_{p i} : 0, 25, 50, 75$ and $100$, with capacity bound $Q^U_{p i} = 100$. Additionally, in the Cobb-Douglas production function-based constraint, we set $a_{p i s} = 0.65$ while $\gamma_1 = 0.45$ and $\gamma_2 = 0.55$ (summation of $\gamma_1$ and $\gamma_2$ must be one) as specified in the model. Other than the demand $D_{c j \omega}$, we also consider the material balance coefficient $\mu_{i j s \omega}$ random and it can hold any value from the interval $[0.5, 1.5]$ with equal probability.

 
\end{document}